\documentclass[a4paper,reqno,11pt]{amsart}
\usepackage{amsfonts,amsmath,amsthm,amssymb,stmaryrd}
\usepackage{mathrsfs}

\newtheorem{lem}{Lemma}[section]

\newtheorem{prop}[lem]{Proposition}
\newtheorem{thm}[lem]{Theorem}
\newtheorem{remark}[lem]{Remark}

\numberwithin{equation}{section}
\usepackage{color}

\usepackage[left=1 in, right=1 in,top=1 in, bottom=1 in]{geometry}

\providecommand{\abs}[1]{\left\vert#1\right\vert}
\providecommand{\norm}[1]{\left\Vert#1\right\Vert}

\providecommand{\Rn}[1]{\mathbb{R}^{#1}}

\providecommand{\br}[1]{\left\langle #1 \right\rangle}

\providecommand{\ns}[1]{\norm{#1}^2}

\providecommand{\ip}[2]{\left(#1,#2\right)}
\providecommand{\ipp}[2]{\left( \mspace{-2.5mu} \left(#1,#2\right) \mspace{-2.5mu} \right) }

\def\Lbrack{\left \llbracket}
\def\Rbrack{\right \rrbracket}

\DeclareMathOperator{\diverge}{div}

\providecommand{\Rn}[1]{\mathbb{R}^{#1}}
\providecommand{\norm}[1]{\left\Vert#1\right\Vert}
\def\ls{\lesssim}

\def\dt{\partial_t}
\def\ddt{\frac{d}{dt}}
\def\H{{}_0H^1}

\def\Hd{(\H)^\ast}

\def\pa{\partial}

\def\rj{\Lbrack \bar{\rho} \Rbrack}
\def\EEE{\E_0^\sigma}
\def\TE{\tilde{\E}_0^\sigma}
\providecommand{\jump}[1]{\left\llbracket #1 \right\rrbracket }

\def\RRvert2{\right \vert\! \right\vert}
\def\Lvert3{\left \vert\!\left\vert\!\left\vert}
\def\Rvert3{\right \vert\!\right\vert\!\right\vert}

\def\nab{\nabla}

\def\dt{\partial_t}

\def\hal{\frac{1}{2}}
\def\ep{\varepsilon}

\def\ls{\lesssim}
\def\p{\partial}
\def\sg{\mathbb{D}}
\def\sgz{\mathbb{D}^0}

\def\da{\Delta_{\mathcal{A}}}
\def\naba{\nab_{\mathcal{A}}}
\def\diva{\diverge_{\mathcal{A}}}
\def\Sa{\S_{\mathcal{A}}}

\def\pal{\p^\alpha}

\def\a{\mathcal{A}}
\def\f{\mathcal{F}}

\def\i{\mathcal{I}}

\def\fj1{\mathcal{J}^{-1}}

\def\M{\mathcal{M}}
\def\n{\mathcal{N}}

\def\z{\mathcal{Z}}
\def\q{q}
\def\S{\mathbb{S}}

\def\E{\mathcal{E}}
\def\D{\mathcal{D}}

\def\Ef{\mathfrak{E}}
\def\Kf{\mathfrak{K}}
\def\Lf{\mathfrak{L}}
\def\Jf{\mathfrak{J}}
\def\If{\mathfrak{I}}

\def\j{\mathfrak{A}}

\def\i{\mathfrak{B}}
\def\ih{\hat{\i}}
\def\Qf{\mathfrak{Q}}
\def\Rf{\mathfrak{R}}
\def\Wf{\mathfrak{W}}

\def\zf{\mathfrak{Z}}

\def\XT{\mathfrak{X}(M_1,M_2,r,T)}

\title[Compressible viscous surface-internal waves]{The compressible viscous surface-internal wave problem: local well-posedness }

\author{Juhi Jang}
\address{
Department of Mathematics\\
University of California, Riverside\\
Riverside, CA 92521, USA
}
\email[J. Jang]{juhijang@math.ucr.edu}
\thanks{J. Jang was supported in part by NSF grants DMS-1212142 and DMS-1351898.}

\author{Ian Tice}
\address{
Department of Mathematical Sciences\\
Carnegie Mellon University\\
Pittsburgh, PA 15213, USA
}
\email[I. Tice]{iantice@andrew.cmu.edu}

\author{Yanjin Wang}
\address{
School of Mathematical Sciences\\
Xiamen University\\
Xiamen, Fujian 361005, China}
\email[Y. J. Wang]{yanjin$\_$wang@xmu.edu.cn}
\thanks{Y. J. Wang was supported by the National Natural Science Foundation of China (No. 11201389)}


\begin{document}

\begin{abstract}
This paper concerns the dynamics of two layers of compressible, barotropic, viscous fluid lying atop one another. The lower fluid is bounded below by a rigid bottom, and the upper fluid is bounded above by a trivial fluid of constant pressure.  This is a free boundary problem: the interfaces between the fluids and above the upper fluid are free to move. The fluids are acted on by gravity in the bulk, and at the free interfaces we consider both the case of surface tension and the case of no surface forces.  We prove that the problem is  locally well-posed.  Our method relies on energy methods in Sobolev spaces for a collection of related linear and nonlinear  problems.
\end{abstract}

\maketitle


\section{Introduction}
\subsection{Formulation in Eulerian coordinates}

We consider two distinct,
immiscible, viscous, compressible, barotropic fluids evolving in a moving domain $\Omega(t)=\Omega_+(t)\cup \Omega_-(t)$ for time $t\ge0$. One fluid $(+)$, called the ``upper fluid,'' fills the upper domain
\begin{equation}\label{omega_plus}
\Omega_+(t)=\{y\in  \mathrm{T}^2\times \mathbb{R}\mid \eta_-(y_1,y_2,t)<y_3< \ell +\eta_+(y_1,y_2,t)\},
\end{equation}
and the other fluid $(-)$, called the ``lower fluid,'' fills the lower domain
\begin{equation}\label{omega_minus}
\Omega_-(t)=\{y\in  \mathrm{T}^2\times \mathbb{R}\mid  -b <y_3<\eta_-(y_1,y_2,t)\}.
\end{equation}
Here we assume the domains are horizontal periodic by setting $\mathrm{T}^2=(2\pi L_1\mathbb{T}) \times (2\pi L_2\mathbb{T})$ for $\mathbb{T} = \mathbb{R}/\mathbb{Z}$ the usual 1--torus and $L_1,L_2>0$ the periodicity lengths.  We assume that $\ell,b >0$ are two fixed and given constants, but the two surface functions $\eta_\pm$ are free and unknown.  The surface $\Gamma_+(t) = \{y_3= \ell  + \eta_+(y_1,y_2,t)\}$ is the moving upper boundary of $\Omega_+(t)$ where the upper fluid is in contact with the atmosphere, $\Gamma_-(t) = \{y_3=\eta_-(y_1,y_2,t)\}$ is the moving internal interface between the two fluids, and $\Sigma_b = \{y_3=-b \}$ is the fixed lower boundary of $\Omega_-(t)$.

The two fluids are described by their density and velocity functions, which are given for each $t\ge0$ by $\tilde{\rho}_\pm
(\cdot,t):\Omega_\pm (t)\rightarrow \mathbb{R}^+$ and $\tilde{u}_\pm (\cdot,t):\Omega_\pm (t)\rightarrow \mathbb{R}^3$, respectively.  In each fluid the pressure is a function of density: $P_\pm =P_\pm(\tilde{\rho}_\pm)>0$, and the pressure function is assumed to be smooth, positive, and strictly increasing.  For a vector function $u\in \Rn{3}$ we define the symmetric gradient by $(\sg u)_{ij} =   \p_i u_j + \p_j u_i$ for $i,j=1,2,3$;  its deviatoric (trace-free) part  is then
\begin{equation}\label{deviatoric_def}
 \sgz u = \sg u - \frac{2}{3} \diverge{u} I,
\end{equation}
where $I$ is the $3 \times 3$ identity matrix.  The viscous stress tensor in each fluid is then given by
\begin{equation}
\S_\pm(\tilde{u}_\pm) := \mu_\pm \sgz \tilde u_\pm +\mu'_\pm \diverge\tilde{u}_\pm I,
\end{equation}
where $\mu_\pm$ is the shear viscosity and  $\mu'_\pm$ is the bulk viscosity; we assume these satisfy the usual physical conditions
\begin{equation}\label{viscosity}
\mu_\pm>0,\quad \mu_\pm'\ge 0.
\end{equation}
The tensor $P_\pm(\tilde{\rho}_\pm) I-\S_\pm(\tilde{u}_\pm)$ is known as the stress tensor.  The divergence of a symmetric tensor $\mathbb{M}$ is defined to be the vector with components $(\diverge \mathbb{M})_i = \p_j \mathbb{M}_{ij}$.  Note then that
\begin{equation}
 \diverge\left( P_\pm(\tilde{\rho}_\pm) I-\S_\pm(\tilde{u}_\pm) \right) = \nab P_\pm(\tilde{\rho}_\pm) - \mu_\pm \Delta \tilde{u}_\pm - \left(\frac{\mu_\pm}{3} + \mu_\pm' \right)\nab \diverge{\tilde{u}_\pm}.
\end{equation}

For each $t>0$ we require that $(\tilde{u}_\pm, \tilde{\rho}_\pm,\eta_\pm)$ satisfy the following equations:
\begin{equation}\label{ns_euler}
\begin{cases}
\partial_t\tilde{\rho}_\pm+\diverge (\tilde{\rho}_\pm \tilde{u}_\pm)=0 & \text{in }\Omega_\pm(t)
\\\tilde{\rho}_\pm    (\partial_t\tilde{u}_\pm  +   \tilde{u}_\pm \cdot \nabla \tilde{u}_\pm ) +\nab P_\pm(\tilde{\rho}_\pm) -  \diverge \S_\pm(\tilde{u}_\pm) =-g\tilde{\rho}_\pm e_3 & \text{in } \Omega_\pm(t)
\\\partial_t\eta_\pm=\tilde{u}_{3,\pm}-\tilde{u}_{1,\pm}\partial_{y_1}\eta_\pm-\tilde{u}_{2,\pm}\partial_{y_2}\eta_\pm &\hbox{on } \Gamma_\pm(t)
\\(P_+(\tilde{\rho}_+)I-\S_+(\tilde{u}_+))n_+=p_{atm}n_+-\sigma_+ \mathcal{H}_+ n_+ &\hbox{on }\Gamma_+(t)
\\ (P_+(\tilde{\rho}_+)I-\S_+( \tilde{u}_+))n_-=(P_-(\tilde{\rho}_-)I-\S_-( \tilde{u}_-))n_-+ \sigma_- \mathcal{H}_-n_- &\hbox{on }\Gamma_-(t) \\\tilde{u}_+=\tilde{u}_-  &\hbox{on }\Gamma_-(t) \\\tilde{u}_-=0 &\hbox{on }\Sigma_b.
\end{cases}
\end{equation}
In the equations $-g \tilde{\rho}_\pm e_3$ is the gravitational force with the constant $g>0$ the acceleration of gravity and $e_3$ the vertical unit vector. The constant $p_{atm}>0$ is the atmospheric pressure, and we take $\sigma_\pm\ge 0$ to be the constant coefficients of surface tension. In this paper, we let $\nabla_\ast$ denote the horizontal gradient, $\diverge_\ast$ denote the horizontal divergence and $\Delta_\ast$ denote the horizontal Laplace operator. Then the upward-pointing unit normal of $\Gamma_\pm(t)$, $n_\pm$,  is given by
\begin{equation}
n_\pm=\frac{(-\nabla_\ast\eta_\pm,1)}
{\sqrt{1+|\nabla_\ast\eta_\pm|^2}},
\end{equation}
and  $\mathcal{H}_\pm$, twice the mean curvature of the surface $\Gamma_\pm(t)$, is given by the formula
\begin{equation}
\mathcal{H}_\pm=\diverge_\ast\left(\frac{\nabla_\ast\eta_\pm}
{\sqrt{1+|\nabla_\ast\eta_\pm|^2}}\right).
\end{equation}
The third equation in \eqref{ns_euler} is called the kinematic boundary condition since it implies that the free surfaces are advected with the fluids. The  boundary equations in \eqref{ns_euler} involving the stress tensor are called the dynamic boundary conditions. Notice that on $\Gamma_-(t)$, the continuity of velocity, $\tilde{u}_+ = \tilde{u}_-$,  means that it is the common value of $\tilde{u}_\pm$ that advects the interface. For a more physical description of the equations \eqref{ns_euler} and the boundary conditions in \eqref{ns_euler}, we refer to \cite{3WL}.

To complete the statement of the problem, we must specify the initial conditions. We suppose that the initial surfaces $\Gamma_\pm(0)$ are given by the graphs of the functions $\eta_\pm(0)$, which yield the open sets $\Omega_\pm(0)$ on which we specify the initial data for the density, $\tilde{\rho}_\pm(0): \Omega_\pm(0) \rightarrow  \mathbb{R}^+$, and the velocity, $\tilde{u}_\pm(0): \Omega_\pm(0) \rightarrow  \mathbb{R}^3$. We will assume that $\ell+\eta_+(0)>\eta_-(0)>-b $ on $\mathrm{T}^2$, which means that at the initial time the boundaries do not intersect with each other.

\subsection{Equilibria}

Here we seek a steady-state equilibrium solution to \eqref{ns_euler} with $\tilde{u}_\pm=0, \eta_\pm =0$, and the equilibrium domains given by
\begin{equation}
\Omega_+=\{y\in   \mathrm{T}^2\times \mathbb{R}\mid 0 < y_3< \ell \} \text{ and }
\Omega_-=\{y\in  \mathrm{T}^2\times \mathbb{R}\mid  -b <y_3< 0 \}.
\end{equation}
Then \eqref{ns_euler} reduces to an ODE for the equilibrium densities  $\tilde\rho_\pm = \bar{\rho}_\pm(y_3)$:
\begin{equation}\label{steady}
\begin{cases}
\displaystyle\frac{d(P_+ (\bar{\rho}_+ ))}{dy_3} = -g\bar{\rho}_+, & \text{for }y_3 \in (0,\ell), \\
\displaystyle\frac{d(P_- (\bar{\rho}_- ))}{dy_3} = -g\bar{\rho}_-, & \text{for } y_3 \in (-b,0), \\
P_+(\bar{\rho}_+(\ell)) = p_{atm}, \\
P_+(\bar{\rho}_+(0))  =P_-(\bar{\rho}_-(0)).
\end{cases}
\end{equation}
The system \eqref{steady} admits a solution $\bar{\rho}_\pm >0$ if and only if we assume that the equilibrium heights $b,\ell>0$, the pressure laws $P_\pm$, and the atmospheric pressure $p_{atm}$ satisfy a collection of admissibility conditions.  These are enumerated in detail in our companion paper \cite{JTW_GWP}.  For the sake of brevity we will not mention these here, but we will assume they are satisfied so that an equilibrium exists.  The resulting function $\bar{\rho}$ is strictly positive and smooth when restricted to to $[-b,0]$ and $[0,\ell]$.

We give special names to the equilibrium density at the fluid interfaces:
\begin{equation}
 \bar{\rho}_1 = \bar{\rho}_+(\ell), \;  \bar{\rho}^+ = \bar{\rho}_+(0),\;  \bar{\rho}^- = \bar{\rho}_-(0).
\end{equation}
Notice in particular that the equilibrium density can jump across the internal interface.   The jump in the equilibrium density, which we denote by
\begin{equation}\label{rho+-}
\rj := \bar{\rho}_+(0)-\bar{\rho}_-(0)= \bar{\rho}^+ - \bar{\rho}^-,
\end{equation}
is of fundamental importance in the the analysis of solutions to \eqref{ns_euler} near equilibrium.  Indeed, if $\rj > 0$ then the upper fluid is heavier than the lower fluid along the equilibrium interface, and the fluid is susceptible to the well-known Rayleigh-Taylor gravitational instability.  This is not particularly important for the local theory developed in this paper,  but of fundamental importance in the stability theory.

In studying perturbations of the equilibrium density it will be useful throughout the paper to employ the enthalpy functions.  These are defined in terms of the pressure laws $P_\pm$ and the equilibrium density values via
\begin{equation}\label{h'}
 h_+(z) = \int_{\bar{\rho}_1}^z \frac{P'_+(r)}{r}dr \text{ and }  h_-(z) = \int_{\bar{\rho}^-}^z \frac{P'_+(r)}{r}dr.
\end{equation}

\subsection{Reformulation in flattened coordinates}

The movement of the free surfaces $\Gamma_\pm(t)$ and the subsequent change of the domains $\Omega_\pm(t)$ create numerous mathematical difficulties. To circumvent these, we will switch to a  coordinate system in which  the boundaries and the domains stay fixed in time.  In order to be consistent with our study of the nonlinear stability of the equilibrium state in \cite{JTW_GWP}, we will use the equilibrium domain. We will not use a Lagrangian coordinate transformation, but rather utilize a special flattening coordinate transformation motivated by Beale \cite{B2}.

To this end, we define the fixed domain
\begin{equation}
\Omega = \Omega_+\cup\Omega_-\text{ with }\Omega_+:=\{0<x_3<\ell \} \text{ and } \Omega_-:=\{-b<x_3<0\},
\end{equation}
for which we have written the coordinates as $x\in \Omega$. We shall write $\Sigma_+:=\{x_3= \ell\}$ for the upper boundary, $\Sigma_-:=\{x_3=0\}$ for the internal interface and $\Sigma_b:=\{x_3=-b\}$ for the lower boundary.  Throughout the paper we will write $\Sigma = \Sigma_+ \cup \Sigma_-$.   We think of $\eta_\pm$ as a function on $\Sigma_\pm$ according to $\eta_+: (\mathrm{T}^2\times\{\ell\}) \times \mathbb{R}^{+} \rightarrow\mathbb{R}$ and $\eta_-:(\mathrm{T}^2\times\{0\}) \times \mathbb{R}^{+} \rightarrow \mathbb{R}$, respectively. We will transform the free boundary problem in $\Omega(t)$ to one in the fixed domain $\Omega $ by using the unknown free surface functions $\eta_\pm$. For this we define
\begin{equation}
\bar{\eta}_+:=\mathcal{P}_+\eta_+=\text{Poisson extension of }\eta_+ \text{ into }\mathrm{T}^2 \times \{x_3\le \ell\}
\end{equation}
and
\begin{equation}
\bar{\eta}_-:=\mathcal{P}_-\eta_-=\text{specialized Poisson extension of }\eta_-\text{ into }\mathrm{T}^2 \times \mathbb{R},
\end{equation}
where $\mathcal{P}_\pm$ are defined by \eqref{P+def} and \eqref{P-def}. The Poisson extensions $\bar{\eta}_\pm$ allow us to flatten the coordinate domains via the following special coordinate transformation:
\begin{equation}\label{cotr}
\Omega_\pm \ni x\mapsto(x_1,x_2, x_3+ \tilde{b}_1\bar{\eta}_++\tilde{b}_2\bar{\eta}_-):=\Theta (t,x)=(y_1,y_2,y_3)\in\Omega_\pm(t),
\end{equation}
where we have chosen $\tilde{b}_1=\tilde{b}_1(x_3), \tilde{b}_2=\tilde{b}_2(x_3)$ to be two smooth functions in $\mathbb{R}$ that satisfy
\begin{equation}\label{b function}
\tilde{b}_1(0)=\tilde{b}_1(-b)=0, \tilde{b}_1(\ell)=1\text{ and }\tilde{b}_2(\ell)=\tilde{b}_2(-b)=0, \tilde{b}_2(0)=1.
\end{equation}
Note that $\Theta(\Sigma_+,t)=\Gamma_+(t),\ \Theta (\Sigma_-,t)=\Gamma_-(t)$ and $\Theta(\cdot,t) \mid_{\Sigma_b} = Id \mid_{\Sigma_b}$.

Note that if $\eta $ is sufficiently small (in an appropriate Sobolev space), then the mapping $\Theta $ is a diffeomorphism.  This allows us to transform the problem \eqref{ns_euler} to one in the fixed spatial domain $\Omega$ for each $t\ge 0$. In order to write down the equations in the new coordinate system, we compute
\begin{equation}\label{A_def}
\begin{array}{ll} \nabla\Theta  =\left(\begin{array}{ccc}1&0&0\\0&1&0\\A  &B  &J  \end{array}\right)
\text{ and }\mathcal{A}  := \left(\nabla\Theta
^{-1}\right)^T=\left(\begin{array}{ccc}1&0&-A   K  \\0&1&-B   K  \\0&0&K
\end{array}\right)\end{array}.
\end{equation}
Here the components in the matrix are
\begin{equation}\label{ABJ_def}
A  =\p_1\theta ,\
B  =\p_2\theta,\
J = 1 + \p_3\theta,\  K  =J^{-1},
\end{equation}
where we have written
\begin{equation}\label{theta}
\theta:=\tilde{b}_1\bar{\eta}_++\tilde{b}_2\bar{\eta}_-.
\end{equation}
Notice that $J={\rm det}\, \nabla\Theta $ is the Jacobian of the coordinate transformation. It is straightforward to check that, because of how we have defined $\bar{\eta}_-$ and $\Theta $, the matrix $\mathcal{A}$ is regular across the interface $\Sigma_-$. 

We now define the density $\rho_\pm$ and the velocity $u_\pm$ on $\Omega_\pm$ by the compositions $\rho_\pm(x,t)=\tilde \rho_\pm(\Theta_\pm(x,t),t)$ and $  u_\pm(x,t)=\tilde u_\pm(\Theta_\pm(x,t),t)$. Since the domains $\Omega_\pm$ and the boundaries $\Sigma_\pm$ are now fixed, we henceforth consolidate notation by writing $f$ to refer to $f_\pm$ except when necessary to distinguish the two; when we write an equation for $f$ we assume that the equation holds with the subscripts added on the domains $\Omega_\pm$ or $\Sigma_\pm$. To write the jump conditions on $\Sigma_-$, for a quantity $f=f_\pm$, we define the interfacial jump as
\begin{equation}
\jump{f} := f_+ \vert_{\{x_3=0\}} - f_- \vert_{\{x_3=0\}}.
\end{equation}
Then in the new coordinates, the PDE \eqref{ns_euler} becomes the following system for $(u,\rho,\eta)$:
\begin{equation}\label{ns_geometric}
\begin{cases}
\partial_t \rho-K\p_t\theta\p_3\rho +\diverge_\a (  {\rho}   u)=0 & \text{in }
\Omega  \\
\rho (\partial_t    u -K\p_t\theta\p_3 u+u\cdot\nabla_\a u  ) + \nabla_\a P ( {\rho} )    -\diva \S_\a (u) =- g\rho e_3 & \text{in }
\Omega
\\ \partial_t \eta = u\cdot \n &
\text{on }\Sigma
\\ (P ( {\rho} ) I- \S_{\a}(u))\n
=p_{atm}\n -\sigma_+  \mathcal{H} \n  &\hbox{on }\Sigma_+
 \\ \jump{P ( {\rho} ) I- \S_\a(u)}\n
= \sigma_-  \mathcal{H} \n  &\hbox{on }\Sigma_-
 \\\jump{u}=0   &\hbox{on }\Sigma_-\\  {u}_- =0 &\text{on }\Sigma_b.
\end{cases}
\end{equation}
Here we have written the differential operators $\naba$, $\diva$, and $\da$ with their actions given by
\begin{equation}
 (\naba f)_i := \a_{ij} \p_j f,\; \diva X := \a_{ij}\p_j X_i, \text{ and }\da f := \diva \naba f
\end{equation}
for appropriate $f$ and $X$.  We have also written
\begin{equation}\label{n_def}
\n := (-\p_1 \eta, - \p_2 \eta,1)
\end{equation}
for the non-unit normal to $\Sigma(t)$, and we have written
\begin{multline}\label{deviatoric_a_def}
(\sg_{\a} u)_{ij} =  \a_{ik} \p_k u_j + \a_{jk} \p_k u_i, \qquad \sgz_{\a} u = \sg_{\a} u - \frac{2}{3} \diva u I,\\
 \text{and } \S_{\a,\pm}(u): =\mu_\pm \sgz_{\a} u+ \mu_\pm' \diva u I.
\end{multline}
Note that if we extend $\diva$ to act on symmetric tensors in the natural way, then $\diva \S_{\a} u =\mu\Delta_\a u+(\mu/3+\mu')\nabla_\a \diverge_\a u$. Recall that $\a$ is determined by $\eta$ through \eqref{A_def}. This means that all of the differential operators in \eqref{ns_geometric} are connected to $\eta$, and hence to the geometry of the free surfaces.

\begin{remark}
The equilibrium state given by \eqref{steady} corresponds to the static solution  $(u,\rho, \eta)=(0,\bar{\rho},0)$ of \eqref{ns_geometric}.
\end{remark}

\subsection{Perturbation equations}

We will now rephrase the PDE \eqref{ns_geometric} in a perturbation formulation around the steady state solution $(0,\bar\rho,0)$. We define a special density perturbation by
\begin{equation}\label{q_def}
 \q=\rho-\bar\rho- \p_3\bar\rho\theta.
\end{equation}
In order to deal with the pressure term $P(\rho)=P(\bar\rho+\q+ \p_3\bar\rho\theta)$ we introduce the Taylor expansion: by \eqref{steady} we have
\begin{equation}\label{R1}
P (\bar\rho+\q+\p_3\bar\rho\theta)=P (\bar{\rho} )+P '(\bar{\rho} )(\q+\p_3\bar\rho\theta)+\mathcal{R}=P (\bar{\rho} )+P '(\bar{\rho} ) \q -g\bar\rho\theta+\mathcal{R},
\end{equation}
where the remainder term is defined via
\begin{equation}\label{R_def}
\mathcal{R} =\int_{\bar{\rho} }^{\bar{\rho} +\q+\p_3\bar\rho\theta}(\bar{\rho} +\q+\p_3\bar\rho\theta-z)  P ^{\prime\prime}(z)\,dz.
\end{equation}
The advantage of defining the perturbation in this manner is seen in the following cancellation:
\begin{equation}
\begin{split}
&\a_{ij}\p_j P ( {\rho} )+g\rho \delta_{i3}=\a_{ij}\p_j P ( {\rho} )+g\rho \a_{ij}\p_j\Theta_3
\\&\quad=\a_{ij}\p_j (P (\bar{\rho} )+P '(\bar{\rho} ) \q -g\bar\rho\theta+\mathcal{R})+g(\bar\rho+\q+\p_3\bar\rho\theta) \a_{ij}\p_j(x_3+\theta)
\\&\quad=\a_{ij}\p_j ( P '(\bar{\rho} ) \q )-g\a_{ij}\p_j  (\bar\rho \theta )+\a_{ij}\p_j\mathcal{R}+g \bar\rho  \a_{ij}\p_j\theta +g(\q+\p_3\bar\rho\theta) \a_{i3} +g( \q+\p_3\bar\rho\theta) \a_{ij}\p_j\theta
\\&\quad=\a_{ij}\p_j ( P '(\bar{\rho} ) \q ) +\a_{ij}\p_j\mathcal{R}+g \q  \a_{i3}+g( \q+\p_3\bar\rho\theta) \a_{ij}\p_j\theta
\\&\quad=\bar\rho\a_{ij}\p_j ( h '(\bar{\rho} ) \q ) +\a_{ij}\p_j\mathcal{R}+g( \q+\p_3\bar\rho\theta) \a_{ij}\p_j\theta,
\end{split}
\end{equation}
where we have used \eqref{steady} and \eqref{h'}. Recalling also \eqref{rho+-}, \eqref{b function}, and \eqref{theta}, we find that
\begin{equation}
 -g\bar\rho_+\theta=-\bar{\rho}_1  g\eta_+\text{ on }\Sigma_+,\text{ and } \jump{-g\bar\rho\theta}  =-\rj g \eta_-\text{ on }\Sigma_-.
\end{equation}

The equations \eqref{ns_geometric}  become the following system when perturbed around the equilibrium $(0,\bar{\rho},0)$:
\begin{equation}\label{geometric}
\begin{cases}
\partial_t \q +\diverge_\a((\bar{\rho} + \q+\p_3\bar\rho\theta) u ) - \p_3^2\bar\rho K \theta \p_t\theta - K\p_t\theta \pa_3  \q =0 & \text{in } \Omega
\\
( \bar{\rho} +  \q+\p_3\bar\rho\theta)\partial_t    u
+( \bar{\rho} +  \q+\p_3\bar\rho\theta) (-K\p_t\theta \pa_3  u  +   u \cdot \nab_\a  u )
+ \bar{\rho}\nabla_\a \left(h'(\bar{\rho})\q\right)  \\
\quad -\diva \S_{\a} u =- \nabla_\a\mathcal{R}-g( \q+\p_3\bar\rho\theta ) \nabla_\a \theta & \text{in }
\Omega
\\
\partial_t \eta = u\cdot \n & \text{on }\Sigma
\\
(  P'(\bar\rho)\q I- \S_{\a}(  u))\n  =  \bar{\rho}_1  g \eta \n-\sigma_+ \mathcal{H}   \n
- \mathcal{R}_+ \n
 & \text{on } \Sigma_+
 \\
 \jump{P'(\bar\rho)\q I- \S_\a(u)}\n = \rj g\eta\n +\sigma_- \mathcal{H} \n  - \jump{ \mathcal{R} }\n
 &\text{on }\Sigma_-
\\
 \jump{u}=0 &\text{on } \Sigma_-
\\
u_-=0 &\hbox{on }\Sigma_b.
\end{cases}
\end{equation}

\begin{remark}
The introduction of the special density perturbation $q$ given by \eqref{q_def} and the subsequent perturbation equations of form \eqref{geometric} is crucial for our study of the nonlinear stability in \cite{JTW_GWP}; it is not  essential for the local theory developed in this paper. Indeed, we could consider $\rho$ or $\rho-\bar\rho$ directly. We choose here to consider $q$ in order to be consistent with the study in \cite{JTW_GWP}.
\end{remark}

\section{Main results and discussion}

\subsection{Previous work}

Free boundary problems in fluid mechanics have attracted much interest in the mathematical community.  A thorough survey of the literature would prove impossible here, so we will primarily mention the work most relevant to our present setting, namely that related to layers of viscous fluid.   We refer to the review of Shibata and Shimizu \cite{ShSh} for a more proper survey of the literature.

The dynamics of a single layer of viscous incompressible fluid lying above a rigid bottom, i.e. the incompressible viscous surface wave problem, have attracted the attention of many mathematicians since the pioneering work of Beale \cite{B1}.  For the case without surface tension, Beale \cite{B1} proved the local well-posedness in the Sobolev spaces.   Hataya \cite{H} obtained the global existence of small, horizontally periodic solutions with an algebraic decay rate in time.    Guo and Tice \cite{GT_per, GT_inf,GT_lwp} developed a two-tier energy method to prove global well-posedness and decay of this problem.  They proved that if the free boundary is horizontally infinite, then the solution decays to equilibrium at an algebraic rate; on the other hand, if the free boundary is horizontally periodic, then the solution decays at an almost exponential rate.  The proofs were subsequently refined by the work of Wu \cite{Wu}.  For the case with surface tension, Beale \cite{B2} proved global well-posedness of the problem, while Allain \cite{A} obtained a local existence theorem in two dimension using a different method.  Bae \cite{B} showed the global solvability in Sobolev spaces via energy methods.  Beale and Nishida \cite{BN} showed that the solution obtained in \cite{B2} decays in time with an optimal algebraic decay rate.   Nishida, Teramoto and Yoshihara \cite{NTY}  showed the global existence of periodic solutions with an exponential decay rate in the case of a domain with a flat fixed lower boundary.  Tani \cite{Ta} and Tani and Tanaka \cite{TT} also discussed the solvability of the problem with or without surface tension by using methods developed by Solonnikov  in \cite{So,So_2,So_3}.  Tan and Wang \cite{TW} studied the vanishing surface tension limit of the problem.

There are fewer results on  two-phase incompressible problems, i.e. the incompressible viscous surface-internal wave or internal wave problems. Hataya \cite{H2} proved an existence result for a periodic free interface problem with surface tension, perturbed around Couette flow; he showed the local  existence of small  solution for any physical constants, and the existence of exponentially decaying small solution if the viscosities of the two fluids are sufficiently large and their difference is small. Pr\"uss and Simonett \cite{PS} proved the local well-posedness of a free interface problem with surface tension in which two layers of viscous fluids fill  the whole space and are separated by a horizontal interface.  For two horizontal fluids of finite depth with surface tension,  Xu and Zhang \cite{XZ} proved  the local solvability for general  data and global solvability for data near the equilibrium state using Tani and Tanaka's method.  Wang and Tice \cite{WT} and Wang, Tice and Kim \cite{WTK} adapted the two-tier energy methods of \cite{GT_per, GT_inf,GT_lwp} to develop the nonlinear Rayleigh-Taylor instability theory for the problem, proving the existence of a sharp stability criterion given in terms of the surface tension coefficient, gravity, periodicity lengths, and $\rj$.

The free boundary problems corresponding to a single horizontally periodic layer of compressible viscous fluid with surface tension have been studied by several authors.  Jin \cite{jin} and Jin-Padula \cite{jin_padula} produced global-in-time solutions using Lagrangian coordinates, and Tanaka and Tani \cite{tanaka_tani} produced global solutions with temperature dependence.  However, to the best of our knowledge, even the local existence problem for two layers of compressible viscous fluids remains unsolved.

The two-layer problem is important because it allows for the development of the classical Rayleigh-Taylor instability \cite{3R,3T}, at least when the equilibrium has a heavier fluid on top and a lighter one below and there is a downward gravitational force.  In our companion paper \cite{JTW_GWP} we identify a stability criterion and prove the existence of global solutions that decay to equilibrium.  In our companion paper \cite{JTW_nrt} we show that the stability criterion is sharp, as in the incompressible case \cite{WT, WTK}, and that the Rayleigh-Taylor instability persists at the nonlinear level (the linear analysis was developed by Guo and Tice in \cite{GT_RT}). The Rayleigh-Taylor instability is a long time phenomenon; for the local-in-time theory developed in this paper it plays no essential role.

\subsection{Local existence}

In our companion paper \cite{JTW_GWP} we deal with questions of global existence and asymptotic stability of solutions to the perturbed system \eqref{geometric}.  The analysis there is carried out in a high-regularity functional framework that is indexed by an integer $N\ge 3$ related to the decay properties of solutions.   Consequently, we must first guarantee the local existence of solutions in this framework for every $N \ge 3$.

Our main result guarantees the existence of high-regularity solutions to \eqref{geometric} under smallness conditions on the initial free surface as well as the temporal interval.   Notice, though, that there are no smallness conditions  placed on the initial velocity or density.  We refer to Appendix \ref{sec_en_dis} for the definitions of the terms ${\EEE}$, $\E$, $\D$, $\hat{\E}^\sigma$, $\hat{\D}^\sigma$, and $\Lf$ (which all depend on $N$) appearing in the statement of the theorem.

\begin{thm}\label{local_existence_intro}
Let $N \ge 3$ be an integer. Assume that either $\sigma_\pm> 0$ or else $\sigma_\pm=0$. Suppose that $(u_0,q_0,\eta_0)$ satisfy ${\EEE}<\infty$ in addition to the compatibility conditions \eqref{ccs}, and that
\begin{equation}\label{q_0_assump}
\rho_\ast\le \rho_0:=\bar\rho+q_0+\p_3\bar\rho\theta_0\le \rho^\ast
\end{equation}
for two constants $0<\rho_\ast,\rho^\ast<\infty$. There exists a universal constant $\delta_\eta >0$ and a $T_{loc} = T_{loc}({\EEE})$ such that if
\begin{equation}
\Lf[\eta_0] \le \frac{\delta_\eta}{2} \text{ and } 0 < T \le T_{loc},
\end{equation}
then there exists a triple $(u,q,\eta)$ defined on the temporal interval $[0,T]$ satisfying the following three properties.  First, $(u,q,\eta)$ achieve the initial data at $t=0$.  Second, the triple uniquely solve \eqref{geometric}.  Third, the triple obey the estimates
\begin{multline}\label{le_intro_01}
 \sup_{0\le t \le T} \left( \E[u(t)] + \E[q(t)] + \hat{\E}^\sigma[\eta(t)] + \ns{\eta(t)}_{4N+1/2}  \right)
\\
 + \int_0^T \left(  \D[u(t)] + \ns{\rho J \dt^{2N+1} u(t)}_{ \Hd} +\D[q(t)]  + \hat{\D}^\sigma[\eta(t)] \right) dt \le P({\EEE})
\end{multline}
for a universal positive polynomial $P$ with $P(0)=0$.  Also,
\begin{equation}
\frac{1}{2}\rho_\ast\le \rho(x,t) :=\bar\rho(x)+q(x,t)+\p_3\bar{\rho}(x)\theta(x,t)\le \frac{3}{2}\rho^\ast \text{ for all } x \in \Omega \text{ and } t \in [0,T]
\end{equation}
and
\begin{equation}
\Lf[\eta](T) \le \delta_\eta.
\end{equation}
Moreover, the mapping $\Theta(\cdot,t)$ defined by \eqref{cotr} is a $C^{4N-2}$ diffeomorphism for each $t \in [0,T]$.
\end{thm}

\begin{remark}
 The diffeomorphism condition in Theorem \ref{local_existence_intro} guarantees that the solution to \eqref{geometric} also gives rise to a solution to \eqref{ns_euler} in the original coordinate system.
\end{remark}


\begin{remark}
The temporal existence interval of Theorem \ref{local_existence_intro}  with $\sigma_\pm>0$ is  independent of $\sigma_\pm$.   As such, we can use a standard limiting argument to produce a solution to \eqref{geometric} when $\sigma_+ >0$ and $\sigma_- =0$ or $\sigma_+ =0$ and $\sigma_- >0$ (the case $\sigma_\pm=0$ is already covered by the theorem).  While this produces a solution, it does not yield a satisfactory local well-posedness theory.  Indeed, if we are given data $(u_0,q_0,\eta_0)$ satisfying the compatibility conditions \eqref{ccs} with, say $\sigma_+ >0$ and $\sigma_- =0$, we would need to produce a sequence of approximate data $(u_0^\sigma,q_0^\sigma,\eta_0^\sigma)$ satisfying \eqref{ccs} with $\sigma_\pm >0$ that converge to $(u_0,q_0,\eta_0)$ as $\sigma_- \to 0$.  While this construction might be possible, the  compatibility conditions \eqref{ccs} are sufficiently complicated that we have not pursued it in this paper.

It is noteworthy, though, that if we send both $\sigma_\pm \to 0$, then the same sort of limiting argument can provide a sort of continuity result, connecting our result with $\sigma_\pm >0$ to our result with $\sigma_\pm=0$. See \cite{JTW_GWP} for more details.
\end{remark}

Theorem \ref{local_existence_intro} is proved in a somewhat more general form (the smallness conditions on the data are more general) in Theorems \ref{local_existence} and \ref{local_existence_no_ST}.  The former handles the case $\sigma_\pm >0$, while the latter handles the case $\sigma_\pm =0$.  The proof of Theorem \ref{local_existence} is more complicated than that of Theorem \ref{local_existence_no_ST} because of the regularity requirements for dealing with the surface tension terms in \eqref{geometric}.  Therefore, most of the paper is devoted to the proof of Theorem \ref{local_existence}.  We sketch below some of the main ideas of the proof.

{ \bf Picard iteration }

We construct solutions to \eqref{geometric} with $\sigma_\pm >0$ by way of a Picard iteration scheme, which is developed in Section \ref{sec_lwp_st}.  This scheme is  built on two sub-problems: the transport problem \eqref{xport_fundamental} for $q$, and the two-phase free boundary Lam\'{e} problem \eqref{lame_free_bndry} for $(u,\eta)$.  The sequence of Picard iterates $\{(u^n,q^n,\eta^n)\}_{n=0}^\infty$ is constructed in Theorem \ref{approx_solns} under some smallness assumptions on  $\eta_0$ and time. The theorem also provides estimates of various high-regularity norms in terms of polynomials of  ${\EEE}$.

The uniform bounds are sufficient for extracting weak and weak-$\ast$ limits from the sequence of Picard iterates.  Space-time compactness results then provide strong convergence results.  However, the nature of our iteration scheme does not allow us to immediately deduce from the strong subsequential limits that a solution to \eqref{geometric} exists.  Instead, we show that the sequence of iterates actually contracts in certain  low-regularity norms, as long as further smallness conditions are imposed on $\eta_0$ and the time interval.  This is the content of Theorem \ref{contraction_thm}.

Theorems \ref{approx_solns} and \ref{contraction_thm} are then combined in Theorem \ref{local_existence} to produce the desired high-regularity solutions to \eqref{geometric}.  Indeed, the low-regularity strong convergence and the high-regularity bounds combine to show that the sequence of iterates actually converges in a sufficiently high-regularity context to pass to the limit in the sequence (without extracting a subsequence) and produce a solution to \eqref{geometric} satisfying the bounds \eqref{le_intro_01}.  The high regularity of our solutions requires that we impose the compatibility conditions mentioned in Theorem \ref{local_existence_intro}.

{ \bf The transport problem for $q$ }

We study the problem \eqref{xport_fundamental} in Section \ref{sec_xport}; the pair $(u,\eta)$ are assumed to be given, and $q$ is solved for.  The key feature of the transport equation in a bounded domain is the behavior of the normal component of the transport velocity on the boundary: its vanishing allows for the construction of solutions without imposing boundary conditions.  This vanishing occurs for \eqref{xport_fundamental} precisely because the pair  $(u,\eta)$  satisfy  the equations $\dt \eta = u \cdot \n$ on $\Sigma$, $\jump{u}=0$ on $\Sigma_-$,  and $u_-=0$ on $\Sigma_b$.

Using this condition and the high regularity of $(u,\eta)$, we can produce a solution to \eqref{xport_fundamental} using the method of characteristics.  In a $C^k$-smoothness context we could then simply apply standard a priori estimates (which again depend critically on the vanishing of the normal transport velocity) to deduce the regularity estimates for $q$ desired in our iteration scheme.  However, due to the Sobolev regularity $(u,\eta)$, it is not a priori clear that those estimates can be rigorously applied to the solution given by characteristics.  To get around this technical obstacle, we show that the problem \eqref{xport_fundamental} can be ``lifted'' to a corresponding problem on the spatial domain $\mathrm{T}^2\times \mathbb{R}$, and that the restriction of the lifted solution agrees with the solution given by characteristics.  The advantage of working with the lifted problem is that it is amenable to solution by the Friedrichs mollification method, which employs delicate properties of mollification operators that are unavailable in bounded domains.  The Friedrichs method allows for the rigorous derivation of the desired a priori Sobolev estimates, and we deduce in Theorem \ref{xport_well-posed} that solutions to \eqref{xport_fundamental} exist and obey the estimates needed for our iteration scheme.

{ \bf The two-phase free boundary Lam\'{e} problem }

The sub-problem \eqref{lame_free_bndry} takes $\rho$ as a given function in addition to various forcing terms, and produces a solution pair $(u,\eta)$.  Notice that \eqref{lame_free_bndry} is not a linear problem: the $\a$ coefficients of the differential operators and the non-unit normal $\n$ are both determined by $\eta$.  We need this nonlinearity for two crucial reasons.  First, as mentioned above, the equation $\dt \eta = u \cdot \n$ is essential in using  $(u,\eta)$  to generate the transport velocity in \eqref{xport_fundamental}.  Second, it allows us to take  advantage of the regularity gain offered by the surface tension terms in \eqref{lame_free_bndry}.  Without it, our iteration scheme would fail to control $\eta$ at the highest level of regularity.

We produce a solution to  \eqref{lame_free_bndry} in Section \ref{sec_lame_free}.  It is simple to see \eqref{lame_free_bndry} as two coupled linear problems: a parabolic problem  for $u$ with $\eta$ given, and a transport problem for $\eta$ with $u$ given.  This perspective works well without surface tension, and indeed this is how we proceed in Theorem \ref{local_existence_no_ST}.  However, with surface tension the regularity demands of the $\sigma_\pm \Delta_\ast \eta_\pm$ terms are higher than what can be recovered from the transport equation for $\eta$, given the target regularity for $u$.  Our way around this difficulty is to recover a solution to \eqref{lame_free_bndry} as a limit as $\kappa \to 0$ of the $\kappa-$approximation problem \eqref{kappa_problem}.  This is accomplished in Theorems \ref{lame_local} and \ref{lame_exist}.  The former establishes the existence of solutions to \eqref{kappa_problem} on $\kappa-$independent intervals, while the latter passes to the limit $\kappa \to 0$ to solve \eqref{lame_free_bndry} and refines some of the estimates from the former.

{ \bf The $\kappa$-approximation problem }

Theorems \ref{lame_local} and \ref{lame_exist} are predicated on the analysis in Section \ref{sec_kappa_exist}, where solutions to \eqref{kappa_problem} are constructed, and Section \ref{sec_kappa_est}, where $\kappa-$independent a priori estimates for solutions to \eqref{kappa_problem} are derived.

Solutions to \eqref{kappa_problem} are produced in Theorem \ref{kappa_contract} by way of the contraction mapping principle in an appropriate space.  The contracting map is constructed by solving the linear sub-problems \eqref{kappa_heat} and  \eqref{kappa_lame}.  The problem \eqref{kappa_heat} is a heat equation; its well-posedness is standard, but we record (without proof) the precise form of the estimates of solutions in terms of $\kappa$ in Theorem \ref{heat_estimates} of Section \ref{sec_heat}.  The only subtlety in \eqref{kappa_heat} is the appearance of the term $\Xi$.  This is inserted as a ``corrector function'' in order to force the compatibility conditions for \eqref{kappa_problem} to match those of \eqref{geometric} for every $\kappa \in (0,1)$.  The problem \eqref{kappa_lame} is studied in Section \ref{sec_lame}, where solutions are produced via a Galerkin scheme and an iteration argument in Theorems \ref{lame_strong} and \ref{lame_high}.

The analysis in Section \ref{sec_kappa_est} develops a priori estimates for \eqref{kappa_problem} that are strong enough to allow us to bound the existence time from below in a $\kappa-$independent manner and to obtain $\kappa-$independent energy and dissipation estimates.  Here again the nonlinear structure of \eqref{kappa_problem} is essential in closing the estimates (see for example the handling of the $F^{3,\alpha}$ terms in Proposition \ref{k_hor_est}).  It is noteworthy that our main estimate, Theorem \ref{kappa_apriori}, actually fails to provide an estimate of $\ns{u}_{L^2_T H^{4N+1}}$, though this term is still guaranteed to be finite.  This is due to the fact that this estimate depends on being able to simultaneously estimate $\ns{\eta}_{L_T^\infty H^{4N+1/2}}$.  The structure of \eqref{kappa_problem} only permits $\kappa-$dependent estimates of this term, and so we fail to estimate $\ns{u}_{L^2_T H^{4N+1}}$ independently of $\kappa$.  It is only later, in Theorem \ref{lame_exist}, after passing to the limit $\kappa \to 0$ that we are able to estimate  $\ns{\eta}_{L_T^\infty H^{4N+1/2}}$ and thereby recover an estimate of $\ns{u}_{L^2_T H^{4N+1}}$.

{ \bf Smallness condition}

Throughout the paper we impose smallness conditions on $\eta_0$ and the size of the temporal interval on which solutions exist.  No smallness condition is imposed on $u_0$ or $q_0$.  The smallness of $\eta_0$ is used with a small-time condition to control various norms of $\eta(\cdot,t)$, which in turn provides control of the coefficients of the differential operators in \eqref{geometric} (see  Lemma \ref{eta_small}), a Korn inequality (see Proposition \ref{korn}), and elliptic regularity estimates (see Proposition  \ref{lame_elliptic}). The smallness of time is used as above and to guarantee that various absorption arguments work throughout the paper.


\subsection{Definitions and terminology} \label{def-ter}

We now mention some of the definitions, bits of notation, and conventions that we will use throughout the paper.

{ \bf Universal constants and polynomials}

Throughout the paper we will refer to generic constants as ``universal'' if they depend on $N$, $\Omega_\pm$,  the various parameters of the problem (e.g. $g$, $\mu_\pm$, $\sigma_\pm$), $\rho_\ast,\;\rho^\ast$ and the functions $\bar{\rho}_\pm$,  with the caveat that if the constant depends on $\sigma_\pm$, then it remains bounded above as either $\sigma_\pm$ tend to $0$.  For example this allows constants of the form $g\mu_+ + 3 \sigma_-^2 + \sigma_+$ but forbids constants of the form $3 + 1/\sigma_-$.  Whenever behavior of the latter type appears we will keep track of the dependence on $\sigma_\pm$ in our estimates.   We make this choice in order to handle the vanishing surface tension limit.

We will employ the notation $a \ls b$ to mean that $a \le C b$ for a universal constant $C>0$.  We also employ the convention of saying that a polynomial $P$ is a ``universal positive polynomial'' if $P(x) = \sum_{j=0}^m C_j x^j$ for some $m \ge 0$, where each $C_j >0$ is a universal constant.  Universal constants and polynomials are allowed to change from one inequality to the next as needed.

{ \bf Norms }

We write $H^k(\Omega_\pm)$ with $k\ge 0$ and and $H^s(\Sigma_\pm)$ with $s \in \Rn{}$ for the usual Sobolev spaces.   We will typically write $H^0 = L^2$.  If we write $f \in H^k(\Omega)$, the understanding is that $f$ represents the pair $f_\pm$ defined on $\Omega_\pm$ respectively, and that $f_\pm \in H^k(\Omega_\pm)$.  We employ the same convention on $\Sigma_\pm$.  Throughout the paper we will refer to the space $\H(\Omega)$ defined as follows:
\begin{equation}
 \H(\Omega) = \{ v \in H^1(\Omega) \; \vert \; \jump{v}=0 \text{ on } \Sigma_- \text{ and } v_- = 0 \text{ on } \Sigma_b\}.
\end{equation}

To avoid notational clutter, we will avoid writing $H^k(\Omega)$ or $H^k(\Sigma)$ in our norms and typically write only $\norm{\cdot}_{k}$, which we actually use to refer to  sums
\begin{equation}
 \ns{f}_k = \ns{f_+}_{H^k(\Omega_+)} + \ns{f_-}_{H^k(\Omega_-)}   \text{ or }  \ns{f}_k = \ns{f_+}_{H^k(\Sigma_+)} + \ns{f_-}_{H^k(\Sigma_-)}.
\end{equation}
Since we will do this for functions defined on both $\Omega$ and $\Sigma$, this presents some ambiguity.  We avoid this by adopting two conventions.  First, we assume that functions have natural spaces on which they ``live.''  For example, the functions $u$, $\rho$, $q$, and $\bar{\eta}$ live on $\Omega$, while $\eta$ lives on $\Sigma$.  As we proceed in our analysis, we will introduce various auxiliary functions; the spaces they live on will always be clear from the context.  Second, whenever the norm of a function is computed on a space different from the one in which it lives, we will explicitly write the space.  This typically arises when computing norms of traces onto $\Sigma_\pm$ of functions that live on $\Omega$.  We use  $\norm{\cdot}_{L^p_TX}$ to denote the norm of the space $L^p([0,T];X)$.

Occasionally we will need to refer to the product of a norm of $\eta$ and a constant that depends on $\pm$.  To denote this we will write
\begin{equation}
\gamma \ns{\eta}_{k} = \gamma_+ \ns{\eta_+}_{H^k(\Sigma_+)} + \gamma_- \ns{\eta_-}_{H^k(\Sigma_-)}.
\end{equation}

{ \bf Derivatives }

We write $\mathbb{N} = \{ 0,1,2,\dotsc\}$ for the collection of non-negative integers.  When using space-time differential multi-indices, we will write $\mathbb{N}^{1+m} = \{ \alpha = (\alpha_0,\alpha_1,\dotsc,\alpha_m) \}$ to emphasize that the $0-$index term is related to temporal derivatives.  For just spatial derivatives we write $\mathbb{N}^m$.  For $\alpha \in \mathbb{N}^{1+m}$ we write $\pal = \dt^{\alpha_0} \p_1^{\alpha_1}\cdots \p_m^{\alpha_m}.$ We define the parabolic counting of such multi-indices by writing $\abs{\alpha} = 2 \alpha_0 + \alpha_1 + \cdots + \alpha_m.$  We will write $\nab_{\ast}f$ for the horizontal gradient of $f$, i.e. $\nab_{\ast}f = \p_1 f e_1 + \p_2 f e_2$, while $\nab f$ will denote the usual full gradient.

\section{The two-phase Lam\'{e} problem}\label{sec_lame}

In this section we are concerned with solving the two-phase Lam\'{e} problem for the velocity field $u$:
\begin{equation}\label{lame}
 \begin{cases}
  \rho \dt u - \diva \Sa u = F^1 & \text{in }\Omega \\
  -\Sa u \n = F^2_+ &\text{on } \Sigma_+ \\
  -\jump{\Sa u } \n = -F^2_- &\text{on } \Sigma_- \\
  \jump{u} =0 &\text{on } \Sigma_- \\
  u_- = 0 &\text{on } \Sigma_b \\
  u(\cdot,0) = u_0 &\text{in } \Omega.
 \end{cases}
\end{equation}
Here we assume that $\eta$ is given and determines $\a$ and $\n$ via \eqref{ABJ_def} and \eqref{n_def}, respectively, and that $\rho$ is given as well.  We also assume that $\rho$ obeys the estimate
\begin{equation}\label{rho_assump_1}
\frac{1}{2} \rho_\ast\le \rho(x,t) \le \frac{3}{2} \rho^\ast \text{ for all } x \in \Omega \text{ and } t \in [0,T].
\end{equation}

In our analysis of \eqref{lame} we will need to make various assumption about energies associated to $\eta$ and $\rho$.  We record these now, first defining a low regularity term that will be used in controlling the coefficients of the equations in \eqref{lame}:
\begin{equation}
 \mathcal{K}[\eta,\rho](T) = \sup_{0\le t\le T} \left( \ns{\eta(t)}_{9/2} + \ns{\dt \eta(t)}_{7/2}  +  \ns{\dt^2 \eta(t)}_{5/2} + \ns{\rho(t)}_{3} + \ns{\dt \rho(t)}_3 \right).
\end{equation}
We also define some high regularity terms:
\begin{multline}
 \Kf[\eta,\rho](T) = \int_0^T \left( \ns{\eta(t)}_{4N+1/2} + \sum_{j=1}^{2N+1} \ns{\dt^j \eta(t)}_{4N-2j+3/2}  \right)dt \\
 + \sup_{0\le t\le T} \left( \ns{\eta(t)}_{4N} + \sum_{j=1}^{2N} \ns{\dt^j \eta(t)}_{4N-2j+1/2} \right)
 +  \sup_{0\le t\le T}  \sum_{j=0}^{2N} \ns{\dt^j \rho(t)}_{4N-2j} ,
\end{multline}
and
\begin{equation}
 \Ef[\eta,\rho](T) =    \sup_{0\le t\le T} \left(   \sum_{j=0}^{2N} \ns{\dt^j \eta(t)}_{4N-2j}  \right) +
 \sup_{0\le t\le T} \left(  \ns{\rho(t)}_{4N} +\sum_{j=1}^{2N} \ns{\dt^j \rho(t)}_{4N-2j+1} \right).
\end{equation}

We will also need to make certain assumptions on the forcing terms.  To describe these we first define
\begin{equation}\label{F2_def}
 \f_2(t) :=  \sum_{j=0}^{2N-1} \ns{\dt^j F^1(t)}_{4N-2j-1}    + \ns{\dt^{2N} F^1(t)}_{\Hd}   + \sum_{j=0}^{2N} \ns{\dt^j F^2(t)}_{4N-2j-1/2}
\end{equation}
and
\begin{equation}\label{Finf_def}
 \f_\infty(t) := \sum_{j=0}^{2N-1} \left[ \ns{\dt^j F^1(t)}_{4N-2j-2} + \ns{\dt^j F^2(t)}_{4N-2j-3/2} \right].
\end{equation}
We will assume throughout this section that the forcing terms obey the estimate
\begin{equation}\label{F12_assump}
 \sup_{0 \le t \le T} \f_\infty(t) + \int_0^T \f_2(t) dt < \infty.
\end{equation}

We define $\dt u(\cdot,0)$ via
\begin{equation}\label{u_data_1}
 (\rho \dt u) \vert_{t=0}  = (\diva \Sa u + F^1)\vert_{t=0},
\end{equation}
and then for $j=0,\dotsc,2N-1$ we inductively define
\begin{equation}\label{u_data_2}
 (\rho \dt^{j+1} u)\vert_{t=0} = - \sum_{k=0}^{j-1} C_{j}^k (\dt^{j-k} \rho \dt^{k+1} u)\vert_{t=0} + \dt^j (\diva \Sa u + F^1)\vert_{t=0}.
\end{equation}
We must assume that $u_0$, $F^1(\cdot,0)$, and $F^2(\cdot,0)$ satisfy the following $2N$ systems of compatibility conditions:
\begin{equation}\label{lame_ccs}
 \begin{cases}
  -\dt^j(\Sa u \n) \vert_{t=0} = \dt^j F^2_+\vert_{t=0} &\text{on } \Sigma_+ \\
  -\jump{\dt^j(\Sa u\n) } \vert_{t=0} = -\dt^j F^2_- \vert_{t=0} &\text{on } \Sigma_- \\
  \jump{\dt^j u}\vert_{t=0} =0 &\text{on } \Sigma_- \\
  \dt^j u_-\vert_{t=0} = 0 &\text{on } \Sigma_b
 \end{cases}
\end{equation}
for $j=0,\dotsc,2N-1$.

\subsection{Weak solutions}

Weak solutions will only arise in our analysis in the context of the highest-order time derivatives of solutions to \eqref{lame}, where they will arise as a byproduct of the regularity assumptions on the forcing terms.  As such, we will only discuss the meaning of weak solution and provide a simple uniqueness result, neglecting a construction of weak solutions.

We define the time-dependent inner-product on $\H$ according to
\begin{equation}
 \ipp{u}{v} = \int_\Omega J(t) \left( \frac{\mu}{2} \sg^0_{\a(t)} u : \sg^0 _{\a(t)} v  + \mu' \diverge_{\a(t)} u \diverge_{\a(t)} v  \right),
\end{equation}
where $J(t)>0$. Notice that for the sake of notational brevity we have neglected to include $t$ in   $\ipp{u}{v}$ even though the inner-product actually changes in time.  With this definition in hand we can give the meaning of weak solutions.

We say $u$ is a weak solution to \eqref{lame} if
\begin{equation}\label{weak_def_1}
 u \in L^\infty([0,T]; H^0) \cap L^2([0,T]; \H) \text{ and }\rho J \dt u \in L^2([0,T];  \Hd),
\end{equation}
$u(\cdot,0) = u_0$, and
\begin{equation}\label{weak_def_2}
 \br{\rho J \dt u,v}_{\ast} + \ipp{u}{v} = \br{J F^1,v}_{\ast} - \br{F^2,v}_{-1/2,\Sigma}
\end{equation}
for almost every $t \in [0,T]$ and for every $v \in \H$.  Here $\br{\cdot,\cdot}_\ast$ denotes the dual pairing for $\H$, while $\br{\cdot,\cdot}_{-1/2,\Sigma}$ denotes the dual-pairing on $H^{1/2}(\Sigma)$.

Next we show that weak solutions are unique.

\begin{prop}\label{weak_unique}
 Weak solutions to \eqref{lame} are unique.
\end{prop}
\begin{proof}
If $u$ is a weak solution, then we choose $v = u$ as a test function in \eqref{weak_def_2}.  Employing Proposition \ref{korn}, we may easily derive the differential inequality
\begin{equation}
 \ddt \int_\Omega \rho J \frac{\abs{u}^2}{2} + C \ns{u}_1 \ls  \int_\Omega \frac{\dt(\rho J)}{\rho J} \rho J \frac{\abs{u}^2}{2}  + \ns{J F^1}_{\ast} + \ns{F^2}_{-1/2}
\end{equation}
for some universal $C>0$.  Then Gronwall implies that
\begin{equation}\label{w_u_1}
 \int_\Omega \rho J \frac{\abs{u}^2}{2}(t) \le e^{\alpha t} \left( \int_\Omega \rho(0) J(0) \frac{\abs{u_0}^2}{2}  + \int_0^t \ns{J F^1}_{\ast} + \ns{F^2}_{-1/2}\right)
\end{equation}
for $\alpha = C + \norm{\dt(\rho J)/(\rho J) }_{L^\infty}$, with $C$ universal.

Then if $u_1, u_2$ are two weak solutions, then $u = u_1 - u_2$ is a weak solution with $u(0) =0$, $F^1=0$, and $F^3=0$.  The equality $u_1 =u_2$ then follows from \eqref{w_u_1} and the fact that $\min{\rho J} > 0$.
\end{proof}

\subsection{Strong solutions}

Now we turn to the construction of strong solutions to \eqref{lame}.  We say that $u$ is a strong solution to \eqref{lame} on $[0,T]$ if
\begin{equation}
\begin{cases}
 u \in L^\infty([0,T];H^2\cap \H) \cap L^2([0,T]; H^3)   \\
 \dt u \in  L^\infty([0,T]; H^0)\cap L^2 ([0,T];\H),
\end{cases}
\end{equation}
$u(\cdot, 0) = u_0$, and $u$ satisfies \eqref{lame} almost everywhere (with respect to Lebesgue measure on $\Omega$ and Hausdorff measure on $\Sigma_\pm$ and $\Sigma_b$).

Our next result establishes the existence of strong solutions.

\begin{thm}\label{lame_strong}
Suppose  that $\rho$ satisfies \eqref{rho_assump_1} and that $\Lf[\eta](T) < \delta_0$, where $\delta_0$ is given by Proposition \ref{lame_elliptic}.   Suppose that $u_0 \in H^2$, $F^1 \in L^\infty_T H^0  \cap L^2_T H^1$, $\dt F^1 \in L^2_T (\H)^\ast$,  $F^2 \in L^\infty_T H^{1/2} \cap L^2_T H^{3/2}$, $\dt F^2 \in L^2_T H^{-1/2}$, and that $u_0$ and $F^2(\cdot,0)$ satisfy the system of compatibility conditions
\begin{equation}\label{ls_01}
 \begin{cases}
 - \Sa u \n \vert_{t=0} = F^2_+ \vert_{t=0} &\text{on } \Sigma_+ \\
  -\jump{\Sa u  } \n \vert_{t=0} = -F^2_- \vert_{t=0} &\text{on } \Sigma_- \\
  \jump{ u}\vert_{t=0} =0 &\text{on } \Sigma_- \\
   u_-\vert_{t=0} = 0 &\text{on } \Sigma_b.
 \end{cases}
\end{equation}
Then there exists a unique $u$ that is a strong solution to \eqref{lame}.  Additionally, $\dt u$ is a weak solution to
\begin{equation}\label{ls_02}
\begin{cases}
   \rho \dt(\dt u) - \diva \Sa (\dt u) =  F^{1,1}  & \text{in }\Omega \\
  -\Sa (\dt u) \n = F^{2,1}_+ &\text{on } \Sigma_+ \\
  -\jump{\Sa (\dt u) } \n = -F^{2,1}_- &\text{on } \Sigma_- \\
  \jump{\dt u} =0 &\text{on } \Sigma_- \\
  \dt u_- = 0 &\text{on } \Sigma_b \\
  \dt u(\cdot,0) = [\rho^{-1}(\diva \Sa u + F^1)] \vert_{t=0}   &\text{in } \Omega,
\end{cases}
\end{equation}
where
\begin{equation}
\begin{split}
F^{1,1} &= \dt F^1  -\dt \rho \dt u  + \diverge_{\dt \a}\Sa u + \diva{\S_{\dt \a} u} \\
F^{2,1}_+& = \dt F^{2}_+ + \S_{\dt \a} u \n + \Sa u \dt \n \\
F^{2,1}_- &= \dt F^{2}_+ - \jump{\S_{\dt \a} u} \n - \jump{\Sa u} \dt \n.
\end{split}
\end{equation}
The solution obeys the estimate
\begin{multline}\label{ls_03}
 \ns{u}_{L^\infty_T H^2} + \ns{u}_{L^2_T H^3} +\ns{\dt u}_{L^\infty_T H^0} + \ns{\dt u}_{L^2_T H^1} + \ns{\rho J \dt^2 u}_{L^2_T \Hd} \\
\ls (1+ P(\mathcal{K}[\eta,\rho](T))) \exp(T(1+ P(\mathcal{K}[\eta,\rho](T))))
\left( \ns{u_0}_2 + \ns{F^1}_{L^\infty_T H^0}  \right. \\
\left. + \ns{F^1}_{L^2_T H^1} + \ns{\dt F^1}_{L^2_T (\H)^\ast}
+ \ns{F^2}_{L^\infty_T H^{1/2}}  + \ns{F^2}_{L^2_T H^{3/2}} + \ns{\dt F^2}_{L^2_T H^{-1/2}}\right),
\end{multline}
where $P$ is a positive universal polynomial such that $P(0)=0$.

\end{thm}

\begin{proof}

The proof is a fairly standard application of the Galerkin method.  For the sake of brevity we will provide only a  terse sketch.  For more details we refer to the incompressible case (where the analysis is actually somewhat more complicated): see for example \cite{GT_lwp}, \cite{WTK}, or \cite{TW}.

Step 1 -- The Galerkin approximation

Let $\{w_k\}_{k=1}^\infty$ be a basis of $\H$ and for $m \ge 1$ set $u^m(x,t) = a^m_k(t) w_k(x)$, where the $k$ summation runs over $k=1,\dotsc,m$.  Using the theory of linear ODEs we can choose the $a^m_k$ such that
\begin{equation}\label{ls_1}
 \ip{\rho J \dt u^m}{w_k}_0 + \ipp{u^m}{w_k} = \ip{J F^1}{w_k}_0 - \ip{F^2}{w_k}_{0,\Sigma}
\end{equation}
and $\ipp{u^m(0)}{w_k}  = \ipp{u_0}{w_k}$ for $k=1,\dotsc,m$. Here $\ip{\cdot}{\cdot}_0$ denotes the $L^2$ inner product on $\Omega$, while $\ip{\cdot}{\cdot}_{0,\Sigma}$ denotes the $L^2$ inner product on $\Sigma$. We then take linear combinations of \eqref{ls_1} so that $w_k$ may be replaced by $u^m$.  This results in the basic energy equality
\begin{equation}\label{ls_2}
 \frac{d}{dt} \int_\Omega \rho J \frac{\abs{u^m}^2}{2} + \ipp{u^m}{u^m} = \int_\Omega \dt(\rho J) \frac{\abs{u^m}^2}{2} + \ip{J F^1}{u^m}_0 - \ip{F^2}{u^m}_{0,\Sigma}.
\end{equation}
Notice that the assumption on  $\Lf[\eta](T)$ guarantees that that the smallness conditions of Lemma \ref{eta_small} and Proposition \ref{korn} are satisfied.  Korn's inequality, Proposition \ref{korn}, allows us to bound
\begin{equation}
 \ns{u^m}_{1} \ls \ipp{u^m}{u^m},
\end{equation}
while Lemma \ref{eta_small} and \eqref{rho_assump_1} provide the estimates
\begin{equation}
 \ns{u^m}_0 \ls \int_\Omega \rho J \frac{\abs{u^m}^2}{2} \ls \ns{u^m}_0 \text{ and } \ip{J F^1}{u^m}_0 \ls \norm{F^1}_0 \int_\Omega \rho J \frac{\abs{u^m}^2}{2},
\end{equation}
and trace theory provides the estimate
\begin{equation}\label{ls_3}
 - \ip{F^2}{u^m}_{0,\Sigma} \ls \norm{F^2}_0 \norm{u^m}_1.
\end{equation}
We combine \eqref{ls_2}--\eqref{ls_3} with Cauchy's inequality to absorb the $\norm{u^m}_1$ term onto the left; applying Gronwall's lemma to the resulting inequality results in the bound
\begin{equation}\label{ls_4}
 \ns{u^m}_{L^\infty_T H^0} + \ns{u^m}_{L^2_T H^1} \ls  \exp(T(1+ P(\mathcal{K}[\eta,\rho](T))) )
\left( \ns{u_0}_0 + \ns{F^1}_{L^2_T H^0}  + \ns{F^2}_{L^2_T H^{0}}  \right).
\end{equation}

Next we differentiate \eqref{ls_1} in time and use a similar argument to get estimates of the sum $\ns{\dt u^m}_{L^\infty_T H^0} + \ns{\dt u^m}_{L^2_T H^1}$.  Two complications arise.  The first is that differentiating \eqref{ls_1} introduces a number of commutators.  These may be dealt with via the estimates \eqref{ls_4} and the usual Sobolev embeddings.  The second is that we must estimate $\ns{\dt u^m(0)}_0$ in terms of $u_0$.  This can be accomplished by using \eqref{ls_1}, evaluated at $t=0$, in conjunction with the choice of initial condition for $u^m(0)$ and the compatibility conditions \eqref{ls_01}.  The resulting estimate bounds  $\ns{\dt u^m}_{L^\infty_T H^0} + \ns{\dt u^m}_{L^2_T H^1}$ by the right side of \eqref{ls_03}.

Step 2 -- Passing to the limit and higher regularity.

The estimates from Step 1 are uniform in $m$ and hence allow for the extraction of weak and weak-$\ast$ limits.  Passing to the limit yields a function such that $u, \dt u \in L^\infty_T H^0 \cap L^2 \H$ and
\begin{equation}\label{ls_5}
 \ip{\rho J \dt u}{w}_0 + \ipp{u}{w} = \ip{J F^1}{w}_0 - \ip{F^2}{w}_{0,\Sigma}
\end{equation}
for a.e. $t \in [0,T]$ and every $w \in \H$.  The $L^\infty_T H^0$ and $L^2_T H^1$ norms of  $u$ and $\dt u$ are controlled by the right side of \eqref{ls_03} by virtue of lower semicontinuity.  We then apply the elliptic regularity result for the Lam\'{e} problem (a variant of  Proposition \ref{lame_elliptic} adapted to the boundary conditions involving $\S$, which can be proved in a similar manner by appealing to Section 3 of \cite{WTK}) to deduce that $\ns{u}_{L^\infty_T H^2} + \ns{ u}_{L^2_T H^3}$ is also bounded by the right side of \eqref{ls_03}.  The only term remaining to estimate in \eqref{ls_03} is $\ns{\rho J \dt^2 u}_{L^2_T (\H)^\ast}$.  This can be accomplished by using \eqref{ls_5} and the existing estimates.  The equation \eqref{ls_02} is seen to hold by differentiating \eqref{ls_5} in time.
\end{proof}

\subsection{Higher regularity}

Our aim now is to show that the strong solutions of Theorem \ref{lame_strong} are of higher regularity when the forcing terms are more regular.  We first define the forcing terms that appear for the time differentiated versions of \eqref{lame}.  We set
\begin{equation}
 F^{1,0} = F^1, \; F^{2,0}_+ = F^2_+, \;\text{ and }F^{2,0}_- = F^2_-,
\end{equation}
and then for $j=1,\dotsc,2N$ we iteratively define
\begin{align}
 F^{1,j} &= \dt F^{1,j-1}  -\dt \rho \dt^j u  + \diverge_{\dt \a}\Sa \dt^{j-1} u + \diva{\S_{\dt \a} \dt^{j-1}u}
   \label{F1j_def} \\
 F^{2,j}_+ &= \dt F^{2,j-1}_+ + \S_{\dt \a} \dt^{j-1} u \n + \Sa \dt^{j-1} u \dt \n
   \label{F2j_def} \\
 F^{2,j}_- &= \dt F^{2,j-1}_+ - \jump{\S_{\dt \a} \dt^{j-1} u} \n - \jump{\Sa \dt^{j-1} u} \dt \n.
   \label{F3j_def}
\end{align}

The following lemma will be used in conjunction with an iteration argument to improve the regularity of strong solutions to \eqref{lame} when the forcing terms are of high regularity.  It is  a simple variant of  Lemma 3.2 of \cite{TW}, which is   a technical refinement of Lemma 4.5 of \cite{GT_lwp}.  As such, we omit the proof.

\begin{lem}\label{lame_forcing_est}
For $m=1,\dotsc,2N-1$ and $j=1,\dotsc,m$, we have the following estimates, where $P$ is a positive universal polynomial such that $P(0)=0$:
 \begin{multline}
  \ns{ F^{1,j}}_{L^2_T H^{2m-2j+1}}       +  \ns{ F^{2,j}}_{L^2_T H^{4N-2j+3/2}}
\ls (1+ P(\Kf[\eta,\rho](T))) \\
\times \left( \int_0^T \f_2(t)dt    + \ns{\dt^j u}_{L^2_T H^{2m-2j+1}}  +\sum_{\ell=0}^{j-1} \ns{\dt^\ell u}_{L^2_T H^{2m-2j+3}} + \ns{\dt^\ell u}_{L^\infty_T H^{2m-2j+2}}
 \right),
\end{multline}
\begin{multline}
 \ns{F^{1,j}}_{L^\infty_T H^{2m-2j}}  +  \ns{F^{2,j}}_{L^\infty_T H^{2m-2j+1/2}}
\\
\ls (1+ P(\Kf[\eta,\rho](T))) \left(  \sup_{0 \le t \le T} \f_\infty(t) + \ns{\dt^j u}_{L^\infty_T H^{2m-2j}}  + \sum_{\ell=0}^{j-1} \ns{\dt^\ell u}_{L^\infty_T H^{2m-2j+2}} \right),
\end{multline}
and  if we denote the functional $v \mapsto \ip{J F^{1,m}}{v}_0 - \ip{F^{3,m}}{v}_{\Sigma}$ in $\Hd$ by $\hat{F}^m$, then
\begin{multline}
 \ns{\dt \hat{F}^{m} }_{L^2_T \Hd}
\ls(1+ P(\Kf[\eta,\rho](T))) \left(  \int_0^T \f_2(t)dt   + \sum_{\ell=0}^{m-1} \ns{\dt^\ell u}_{L^\infty_T H^{2}} + \ns{\dt^\ell u}_{L^2_T H^{3}} \right. \\
\left. +\ns{\dt^m u}_{L^2_T H^{2}} + \ns{\dt^{m+1} u}_{L^2_T H^0}
\right).
\end{multline}

\end{lem}

We can now state our higher regularity result.

\begin{thm}\label{lame_high}
Assume that $u_0 \in H^{4N}$, that $F^1,F^2$ satisfy \eqref{F12_assump}, and that the compatibility conditions \eqref{lame_ccs} are satisfied.  Assume that $\rho$ satisfies \eqref{rho_assump_1}, and that $\Lf[\eta](T) < \delta_0$, where $\delta_0$ is given by Proposition \ref{lame_elliptic}.  Then there exists a unique $u$ satisfying $\dt^j u \in L^\infty([0,T];H^{4N-2j}) \cap L^2([0,T];H^{4N-2j+1})$ for $j=0,\dotsc,2N$ and $\rho J \dt^{2N+1}u \in L^2([0,T];\Hd)$ so that $\dt^j u$ solves
\begin{equation}\label{lh_01}
\begin{cases}
   \rho \dt(\dt^j u) - \diva \Sa (\dt^j u) =  F^{1,j}  & \text{in }\Omega \\
  -\Sa (\dt^j u) \n = F^{2,j}_+ &\text{on } \Sigma_+ \\
  -\jump{\Sa (\dt^j u) } \n = -F^{2,j}_- &\text{on } \Sigma_- \\
  \jump{\dt^j u} =0 &\text{on } \Sigma_- \\
  \dt^j u_- = 0 &\text{on } \Sigma_b \\
\end{cases}
\end{equation}
strongly for $j=0,\dotsc,2N-1$ and weakly for $j=2N$.  The solution obeys the estimates
\begin{multline}\label{lh_02}
 \sum_{j=0}^{2N}  \ns{\dt^j u}_{L^2_T H^{4N-2j+1}}  + \ns{\rho J \dt^{2N+1} u}_{L^2_T \Hd}
\\
 \ls (1+P(\Kf[\eta,\rho](T))) \exp\left((1+P(\Ef[\eta,\rho](T)))T \right)  \left( \sum_{j=0}^{2N} \ns{\dt^j u(\cdot,0)}_{4N-2j}   + \int_0^T \f_2(t) dt \right)
\end{multline}
and
\begin{multline}\label{lh_03}
 \sum_{j=0}^{2N}  \ns{\dt^j u}_{L^\infty_T H^{4N-2j}}
 \ls (1+P(\Kf[\eta,\rho](T))) \exp\left((1+P(\Ef[\eta,\rho](T)))T \right) \\
 \times \left( \sum_{j=0}^{2N} \ns{\dt^j u(\cdot,0)}_{4N-2j}  + \sup_{0\le t \le T} \f_\infty(t) + \int_0^T \f_2(t) dt \right),
\end{multline}
where $P$ is a universal positive polynomial such that $P(0)=0$.
\end{thm}
\begin{proof}
 The proof follows the same basic outline (though the arguments here are somewhat simpler), as that of Theorem 4.8 of \cite{GT_lwp}: we iteratively apply Theorem \ref{lame_strong},  Lemma \ref{lame_forcing_est} to estimate the forcing terms appearing for each $j$,  and the elliptic regularity result. We omit further details for the sake of brevity.
\end{proof}

\section{Heat estimates}\label{sec_heat}

Suppose that $0 < \kappa < 1$.  In this section we are concerned with the problem
\begin{equation}\label{heat}
 \begin{cases}
  \dt \eta - \kappa \Delta_\ast \eta = f & \text{in }\Sigma \times (0,T) \\
  \eta(\cdot,0) = \eta_0 & \text{in }\Sigma
 \end{cases}
\end{equation}
where $\eta_0 \in H^{4N+1}(\Sigma)$ and $f$ satisfies
\begin{equation}\label{eta_f_assump}
\int_0^T \sum_{j=0}^{2N} \ns{\dt^j f(t)}_{4N-2j } dt < \infty.
\end{equation}
A simple interpolation argument (see for example Lemma A.4 of \cite{GT_lwp}) allows us to deduce from \eqref{eta_f_assump} that $\dt^j f \in C^0([0,T];H^{4N-2j-1}(\Sigma))$ for $j=0,\dotsc,2N-1$.  We may use this in \eqref{heat} to inductively define $\dt^j \eta(\cdot,0)\in H^{4N-2j+1}(\Sigma)$ for $j=1,\dotsc,2N$

We now state a result on the well-posedness of \eqref{heat} and estimates of its solutions.  The proof is standard and thus omitted.

\begin{thm}\label{heat_estimates}
Suppose that $\eta_0 \in H^{4N+1}(\Sigma)$, that $f$ satisfies \eqref{eta_f_assump}, and that the data $\dt \eta(\cdot,0) \in H^{4N-2j+1}(\Sigma)$ are determined as above for $j=1,\dotsc,2N$.  The problem \eqref{heat} admits a unique solution that achieves the initial data $\dt^j\eta(\cdot,0)$ for $j=0,\dotsc,2N$ and that satisfies the estimate
\begin{multline}\label{he_01}
\sum_{j=0}^{2N} \ns{\dt^j \eta}_{L^\infty_T H^{4N-2j}} + \sum_{j=1}^{2N+1} \ns{\dt^j \eta}_{L^2_T H^{4N-2j+2}} \\
+ \kappa \left[ \sum_{j=0}^{2N} \ns{\dt^j \eta}_{L^\infty_T H^{4N-2j+1}} + \sum_{j=0}^{2N} \ns{\dt^j \eta}_{L^2_T H^{4N-2j+1}} \right] + \kappa^2 \ns{\eta}_{L^2_T H^{4N+2}}
\\
\ls e^T \sum_{j=0}^{2N} \ns{\dt^j \eta(\cdot,0)}_{4N-2j+1} + \int_0^T e^{T-t} \sum_{j=0}^{2N}  \ns{\dt^j f(t)}_{4N-2j } dt.
\end{multline}

\end{thm}

\section{The $\kappa-$approximation problem}\label{sec_kappa_exist}

Our goal in this section is to produce a solution pair $(u,\eta)$ to the problem
\begin{equation}\label{kappa_problem}
\begin{cases}
  \rho \dt u - \diva \Sa u = F^1 & \text{in }\Omega \\
  \dt \eta - \kappa \Delta_\ast \eta = u \cdot \n  - \kappa \Xi &\text{on }\Sigma \\
  -\Sa u \n = - \sigma_+ \Delta_\ast \eta  \n +  F^2_+   &\text{on } \Sigma_+ \\
  -\jump{\Sa u } \n =  \sigma_- \Delta_\ast \eta \n - F^2_- &\text{on } \Sigma_- \\
  \jump{u} =0 &\text{on } \Sigma_- \\
  u_- = 0 &\text{on } \Sigma_b \\
  u(\cdot,0) = u_0, \eta(\cdot, 0) = \eta_0.
 \end{cases}
\end{equation}
Here we assume that $u_0 \in H^{4N}(\Omega)$, $\eta_0 \in H^{4N+1}(\Sigma)$, $F^1$, $F^2$ satisfy \eqref{F12_assump}.  Notice that the problem \eqref{kappa_problem} is nonlinear: the terms $\a$ and $\n$ are viewed as being generated by $\eta$ itself.  We will again assume that $\rho$ is a given function on the given time interval  $[0,T_\ast]$  and satisfies \eqref{rho_assump_1} (with $T$ replaced by $T_\ast$) in addition to the bound
\begin{equation}\label{rho_assump_2}
 \Ef[\rho](T_*) :=  \sup_{0 \le t \le T_\ast} \Ef_\infty[\rho(t)] := \sup_{0\le t \le T_\ast} \left(\ns{\rho(t)}_{4N} + \sum_{j=1}^{2N} \ns{\dt^j \rho(t)}_{4N-2j+1} \right) \ls 1+ P({\EEE}).
\end{equation}

We now show how to construct the data $(\dt^j u(\cdot,0),\dt^j\eta(\cdot,0))$ for $j=1,\dotsc,2N$ from the pair $(u_0,\eta_0)$ in a manner that is consistent with the compatibility conditions for the problem \eqref{geometric}. Given $(\dt^j u(\cdot,0),\dt^j\eta(\cdot,0))$ for some $j=0,\dotsc,2N-1$ we first use the first equation in \eqref{kappa_problem} to solve for $\dt^{j+1} u(\cdot,0)$.  Then we set
\begin{equation}\label{eta_data}
\dt^{j+1} \eta(\cdot,0) = \dt^j(u \cdot \n)\vert_{t=0}.
\end{equation}
In this way we iteratively define all the time-differentiated data for $j=1,\dotsc,2N$.  However, the only way this can be consistent with the natural compatibility conditions needed to study \eqref{kappa_problem} is if $\Xi:\Sigma \times [0,\infty)$ is chosen so that
\begin{equation}\label{Xi_prop}
 \begin{split}
  \Xi(\cdot,0) &= \Delta_\ast \eta_0 \in H^{4N-1}(\Sigma), \text{ and }\\
  \dt^j \Xi(\cdot,0) &= \left. \Delta_\ast \left(\dt^{j-1}(u\cdot \n) \right) \right\vert_{t=0}  \in H^{4N-2j-1/2}(\Sigma) \text{ for }j=1,\dotsc,2N.
 \end{split}
\end{equation}
 According to Proposition \ref{extension_Xi}, there exists a $\Xi$ satisfying \eqref{Xi_prop} as well as the bounds
\begin{multline}\label{Xi_bound}
 \sum_{j=0}^{2N-1} \sup_{t>0} \ns{\dt^j \Xi(t)}_{4N-2j-1} + \sum_{j=0}^{2N} \int_0^\infty \ns{\dt^j \Xi(t)}_{ 4N-2j}dt
\\
\ls \ns{\Delta_\ast \eta_0}_{4N-1} + \sum_{j=1}^{2N} \ns{\Delta_\ast \left(\dt^{j-1}(u\cdot \n) \right) (0)}_{4N-2j-1}
\end{multline}

In order to produce a high regularity solution to \eqref{kappa_problem} we must assume that the $u_0$, $\eta_0$, and the forcing terms $F^1$, $F^2$ satisfy the compatibility conditions
\begin{equation}\label{kappa_ccs}
\begin{cases}
  -\dt^j(\Sa u \n) \vert_{t=0} = \dt^j\left(    - \sigma_+ \Delta_\ast \eta  \n +  F^2_+  \right)  \vert_{t=0} &\text{on } \Sigma_+ \\
  -\jump{\dt^j(\Sa u\n) } \vert_{t=0} = \dt^j \left(  \sigma_- \Delta_\ast \eta \n - F^2_- \right) \vert_{t=0} &\text{on } \Sigma_- \\
  \jump{\dt^j u}\vert_{t=0} =0 &\text{on } \Sigma_- \\
  \dt^j u_-\vert_{t=0} = 0 &\text{on } \Sigma_b
 \end{cases}
\end{equation}
for $j=0,\dotsc,2N-1$.

For $0<T$ we define
\begin{equation}
\begin{split}
 \j_2[u](T) &= \int_0^T \left( \sum_{j=0}^{2N} \ns{\dt^j u(t)}_{4N-2j+1} + \ns{\rho J \dt^{2N+1} u(t)}_{\Hd} \right) dt \\
 \j_\infty[u](T) &= \sup_{0 \le t \le T} \sum_{j=0}^{2N} \ns{\dt^j u(t)}_{4N-2j}.
\end{split}
\end{equation}
Similarly, we write
\begin{equation}
   \i_2[\eta](T) = \int_0^T \sum_{j=0}^{2N+1} \ns{\dt^j \eta (t) }_{4N-2j+2} dt \text{ and }
   \i_\infty[\eta](T) = \sup_{0\le t \le T} \sum_{j=0}^{2N} \ns{\dt^j \eta(t)}_{4N-2j+1}.
\end{equation}
We then define
\begin{equation}\label{ab_def}
 \j[u](T) = \j_2[u](T) + \j_\infty[u](T) \text{ and } \i[\eta](T) = \i_2[\eta](T) + \i_\infty[\eta](T).
\end{equation}
Finally, we also define
\begin{equation}\label{ij_def}
 \Jf[u](T) = \sup_{0\le t \le T} \sum_{j=0}^{2N-1} \ns{\dt^j u(t)}_{4N-2j-1}  \text{ and } \If[\eta](T) = \sup_{0\le t \le T} \sum_{j=0}^{2N-1} \ns{\dt^j \eta(t)}_{4N-2j-1}.
\end{equation}

\subsection{Constructing the mapping}

We will employ a fixed point argument in order  to produce a solution to the nonlinear $\kappa-$approximation problem \eqref{kappa_problem}.  We thus begin with the development of an appropriate metric space in which to work.

Notice that  if $\j_2[v](T) + \i_2[\zeta](T)< \infty$, then the maps $[0,T]\ni t \mapsto \dt^j v(\cdot,t) \in H^{4N-2j}(\Omega)$ and $[0,T]\ni t \mapsto \dt^j \zeta(\cdot,t) \in H^{4N-2j}(\Sigma)$ are absolutely continuous for $j=0,\dotsc,2N$.
In particular, $\dt^j v(\cdot,0)$ and $\dt^j \zeta(\cdot,0)$ are well-defined for $j=0,\dotsc,2N$.  This then allows us to  define the metric space, for $0 < T \le T_\ast$,  $M_1,M_2>0$ and  $r\ge 2$ an integer,
\begin{multline}\label{metric_space_def}
 \XT  = \{ (v,\zeta) \;\vert\; \j[v](T) \le M_1 \kappa^{-r},  \i[\zeta](T)  \le M_2 \kappa^{-2},  \\
\text{ and } \dt^j v(\cdot, 0) = \dt^j u(\cdot,0), \dt^j \zeta(\cdot,0) = \dt^j \eta(\cdot,0)  \text{ for }j=0,\dots,2N\}.
\end{multline}
The space $\XT$ is complete due to the fact that $\{ (v,\zeta) \;\vert\; \j[v](T) + \i[\zeta](T)< \infty\}$ is a Banach space.

We now define a mapping $\M : \XT \to \XT$ as $\M(v,\zeta) = (u,\eta)$, where $\eta$ and $u$ are determined through the following two steps.  In these steps we will assume that $0<T<1$ and that $\Lf[\eta_0] < \delta_0/2$, where $\delta_0$ is given by Proposition \ref{lame_elliptic}.

\emph{Step 1  -- The $\eta$ equation}

 Note that
\begin{equation}\label{s1_1}
 \sum_{j=0}^{2N}  \int_0^T \ns{\dt^j (v\cdot \n[\zeta])(t)}_{4N-2j} dt \ls (1+ \i_\infty[\zeta](T)) \sum_{j=0}^{2N} \int_0^T \ns{\dt^j v(t)}_{4N-2j+1/2}dt.
\end{equation}
The usual Sobolev interpolation allows us to estimate
\begin{equation}
 \ns{\dt^j v(t)}_{4N-2j+1/2} \ls \norm{\dt^j v(t)}_{4N-2j} \norm{\dt^j v(t)}_{4N-2j+1}
\end{equation}
when $j =0,\dotsc,2N$, and hence
\begin{multline}\label{s1_1_5}
 \sum_{j=0}^{2N} \int_0^T \ns{\dt^j v(t)}_{4N-2j+1/2}dt \ls  \sum_{j=0}^{2N} \int_0^T \norm{\dt^j v(t)}_{4N-2j} \norm{\dt^j v(t)}_{4N-2j+1}dt \\
\ls \left(\sum_{j=0}^{2N} \int_0^T  \ns{\dt^j v(t)}_{4N-2j}  dt \right)^{1/2} \left(\sum_{j=0}^{2N} \int_0^T \ns{\dt^j v(t)}_{4N-2j+1} dt \right)^{1/2} \\
\ls \sqrt{T} \sqrt{\j_\infty[v](T)} \sqrt{\j_2[v](T)}.
\end{multline}
Combining the  estimates \eqref{s1_1} and \eqref{s1_1_5}, we find that
\begin{equation}\label{s1_1_6}
  \sum_{j=0}^{2N}  \int_0^T \ns{\dt^j (v\cdot \n[\zeta])(t)}_{4N-2j} dt \ls (1+ \i_\infty[\zeta](T)) \sqrt{T} \sqrt{\j_\infty[v](T)} \sqrt{\j_2[v](T)}.
\end{equation}

Estimate \eqref{s1_1_6} and \eqref{Xi_bound} allow us to use  Theorem \ref{heat_estimates} with $f = v\cdot \n[\zeta] - \kappa \Xi$ and  $\Xi$  defined as in \eqref{Xi_prop}--\eqref{Xi_bound} to produce a unique solution $\eta$ to
\begin{equation}\label{kappa_heat}
 \begin{cases}
 \dt \eta - \kappa \Delta_\ast \eta = v \cdot \n[\zeta] - \kappa \Xi & \text{in }\Sigma \times (0,T) \\
  \eta(\cdot,0) = \eta_0 & \text{in }\Sigma.
 \end{cases}
\end{equation}
The theorem guarantees that $\dt^{j+1} \eta(\cdot,0) = \dt^j (v \cdot \n)\vert_{t=0} = \dt^j (u\cdot \n)\vert_{t=0}$ for $j=0,\dotsc,2N-1$, with the understanding that this last quantity is the initial data computed previously in \eqref{eta_data}, and that $\eta$ obeys the estimate (utilizing \eqref{s1_1_6})
\begin{multline}\label{s1_2}
 \kappa^2 \left(\i_2[\eta](T) + \i_\infty[\eta](T) \right)\ls e^T \sum_{j=0}^{2N} \ns{\dt^j \eta(\cdot,0)}_{4N-2j+1}
\\
+ \int_0^T e^{T-t} \sum_{j=0}^{2N}  \ns{\dt^j (v \cdot \n[\zeta] -\kappa \Xi)(t)}_{4N-2j } dt
\\
\ls e^T \left( \ns{\eta_0}_{4N+1} + \sum_{j=0}^{2N-1} \ns{\dt^j (u \cdot \n)\vert_{t=0}}_{4N-2j-1} \right)
+   (1+ \i_\infty[\zeta](T)) \sqrt{T} \sqrt{\j_\infty[v](T)} \sqrt{\j_2[v](T)}
\\
\ls \ns{\eta_0}_{4N+1} + \sum_{j=0}^{2N-1} \ns{\dt^j (u \cdot \n)\vert_{t=0}}_{4N-2j-1} +  \left(1+\frac{M_2}{\kappa^2}\right)  \frac{M_1}{\kappa^r} \sqrt{T} ,
\end{multline}
where in the last line we have used the bounds built into the definition of $\XT$ and the assumption  that $T<1$.

From \eqref{s1_2} we see that if $M_2$ is sufficiently large (depending on the data, $\sigma$, etc) and  $T$ is sufficiently small (depending on $\kappa$ and $M_1,M_2$),  then
\begin{equation}\label{s1_3}
 \i[\eta](T) \le \frac{M_2}{\kappa^2},
\end{equation}
which is the requirement for $\eta$ to be part of a pair belonging in $\XT$.  On the other hand, from Lemma \ref{time_interp}
\begin{equation}\label{s1_4}
\Lf[\eta](T) \le \Lf[\eta_0] + \int_0^T  \ih_\infty[\eta(t)]dt \le \Lf[\eta_0] + T \i_\infty[\eta](T) \le \frac{\delta_0}{2} + \frac{TM_2}{\kappa^2} < \delta_0
\end{equation}
if $T$ is further restricted.

\emph{Step 2  -- The $u$ equation}

Now we seek to solve the problem
\begin{equation}\label{kappa_lame}
\begin{cases}
  \rho \dt u - \diverge_{\a[\eta]} \S_{\a[\eta]} u = F^1 & \text{in }\Omega \\
  -\S_{\a[\eta]}  u \n[\eta] =   -\sigma_+ \Delta_\ast \eta \n[\eta] +  F^2_+ &\text{on } \Sigma_+ \\
  -\jump{\S_{\a[\eta]}  u } \n[\eta] = \sigma_- \Delta_\ast \eta \n[\eta] -F^2_- &\text{on } \Sigma_- \\
  \jump{u} =0 &\text{on } \Sigma_- \\
  u_- = 0 &\text{on } \Sigma_b \\
  u(\cdot,0) = u_0,
 \end{cases}
\end{equation}
where $\eta$ is the function constructed in Step 1 and $\n[\eta], \a[\eta]$ are determined in terms of $\eta$ as in \eqref{ABJ_def} and \eqref{n_def}.  Let us write $\tilde{F}^1 = F^1$, $\tilde{F}^2_+ =-\sigma_+ \Delta_\ast \eta\n[\eta] +  F^2_+$, and $\tilde{F}^2_- = \sigma_- \Delta_\ast \eta\n[\eta] -F^2_-$.  Then  we write $\tilde{\f}_2(t)$ and $\tilde{\f}_\infty(t)$ as in \eqref{F2_def} and \eqref{Finf_def}, except with $\tilde{F}^1$ and $\tilde{F}^2$ in place of $F^1,F^2$.  Then clearly
\begin{equation}\label{s2_1}
\begin{split}
 \tilde{\f}_\infty(t) &\ls \f_\infty(t) +
 \left(1+\sum_{j=0}^{2N} \ns{\dt^j \eta(t)}_{4N-2j+1} \right) \sum_{j=0}^{2N} \ns{\dt^j \eta(t)}_{4N-2j+1} ,    \\
 \tilde{\f}_2(t) &\ls \f_2(t) + (1+\sigma^2) \left(1 + \sum_{j=0}^{2N} \ns{\dt^j \eta(t)}_{4N-2j+1} \right) \sum_{j=0}^{2N}\ns{\dt^j \eta(t)}_{4N-2j+3/2}.
\end{split}
\end{equation}

A simple interpolation argument, in conjunction with the Cauchy-Schwarz inequality, provides us with the estimate
\begin{multline}\label{s2_1_5}
  \int_0^T \sum_{j=0}^{2N}\ns{\dt^j \eta(t)}_{4N-2j+3/2} dt \ls  \int_0^T \sum_{j=0}^{2N}\norm{\dt^j \eta(t)}_{4N-2j+1} \norm{\dt^j \eta(t)}_{4N-2j+2}dt \\
\ls \sqrt{T} \sqrt{ \i_\infty[\eta](T)}  \sqrt{\i_2[\eta](T) } \ls \sqrt{T} \left( \i_2[\eta](T) + \i_\infty[\eta](T)\right).
\end{multline}
We then deduce from \eqref{s2_1} and \eqref{s2_1_5} that
\begin{equation}\label{s2_2_1}
  \sup_{0\le t \le T} \tilde{\f}_\infty(t)  \ls \sup_{0\le t \le T} \f_\infty(t) + (1+\i_\infty[\eta](T))  \i_\infty[\eta](T)
\end{equation}
and
\begin{multline}\label{s2_2_2}
\int_0^T \tilde{\f}_2(t) dt \ls    \int_0^T \f_2(t)dt  +(1+\sigma^2)(1+\i_\infty[\eta](T))  \sqrt{T} \left( \i_2[\eta](T) + \i_\infty[\eta](T)\right)   .
\end{multline}

Now, owing to the bounds \eqref{s2_2_1}--\eqref{s2_2_2}, the compatibility conditions \eqref{kappa_ccs}, and the estimate \eqref{s1_4}, we may apply Theorem \ref{lame_high} with $\tilde{F}^1$ and $\tilde{F}^2$ in place of $F^1,F^2$.  This yields a unique solution $u$ to \eqref{kappa_lame} satisfying (after combining \eqref{lh_02}--\eqref{lh_03} with \eqref{s2_2_1}--\eqref{s2_2_2})
\begin{multline}\label{s2_3}
\j_2[u](T)  \ls (1+ P(\i[\eta](T)   +\Ef[\rho](T) )) \exp\left(T(1+ P(\i_\infty[\eta](T) +\Ef[\rho](T))) \right) \\
\times \left( \sum_{j=0}^{2N} \ns{\dt^j u(\cdot,0)}_{4N-2j} +  \int_0^T \f_2(t)dt
+(1+\sigma^2)\sqrt{T}  \i[\eta](T)     \right),
\end{multline}
and
\begin{multline}\label{s2_4}
\j_\infty[u](T) \ls (1+ P(\i[\eta](T)   +\Ef[\rho](T)))\exp\left(T(1+P(\i_\infty[\eta](T) +\Ef[\rho](T))) \right)
\\
\times
\left( \sum_{j=0}^{2N} \ns{\dt^j u(\cdot,0)}_{4N-2j}  + \sup_{0\le t \le T} \f_\infty(t) + \i_\infty[\eta](T) +  \int_0^T \f_2(t)dt  +(1+\sigma^2)\sqrt{T}  \i[\eta](T)  \right)
\end{multline}
for some universal positive polynomial $P(\cdot)$ with $P(0)=0$.

Recall that we assume the bound \eqref{rho_assump_2}.  Then by restricting $T$ further and employing the estimate \eqref{s1_3}, we may deduce from \eqref{s2_3} and \eqref{s2_4} that there exists a natural number $\ell$ such that
\begin{equation}\label{s2_5}
\j[u](T)  \ls \left(  \frac{1+P({\EEE})+M_2^\ell}{\kappa^\ell}    \right)
\left( \sum_{j=0}^{2N} \ns{\dt^j u(\cdot,0)}_{4N-2j}   + \sup_{0\le t \le T} \f_\infty(t) + \int_0^T \f_2(t)dt
+1     \right).
\end{equation}
If $M_1$ is sufficiently large and $r$ is set to $\ell$, then
\begin{equation}\label{s2_6}
 \j[u](T) \le \frac{M_1}{\kappa^r},
\end{equation}
which means that $u$ satisfies the estimate required to belong to a pair in $\XT$.

\emph{Step 3 -- Inclusion in $\XT$.}

In particular, \eqref{s1_3} and \eqref{s2_6} imply that if $M_1,M_2>0$ are taken to be sufficiently large, and $0<T<1$ is taken to be sufficiently small, then $(u,\eta) \in \XT$.    This allows us to set $\M(v,\zeta) = (u,\eta) \in \XT$, and so the mapping $\M: \XT \to \XT$ is well-defined.

\subsection{Fixed point}

Now we show that the mapping constructed above has a fixed point.

\begin{thm}\label{kappa_contract}
Assume that $\rho$ is given on the time interval $[0,T_\ast]$ and satisfies \eqref{rho_assump_1} and \eqref{rho_assump_2}.  Let $\delta_0$ be given by Proposition \ref{lame_elliptic}, and assume that $\delta \in (0,\delta_0)$.  Suppose that  that $\Lf[\eta_0] \le \delta/2$, where $\Lf$ is given by \eqref{l0_def}.  There exist $M_1,M_2 >0$, $r \in \mathbb{N}$, and  $T_0 = T_0(\kappa,\sigma,\delta,\TE) \in (0,T_\ast]$ such that if $0 < T \le T_0$ then $\M : \XT \to \XT$ is a contraction, and therefore admits a unique fixed point $(u,\eta) = \M(u,\eta)$.  In particular, there exists a unique pair $(u,\eta)$ satisfying
\begin{equation}\label{kc_01}
\j[u](T) \le M_1 \kappa^{-r} \text{ and } \i[\eta](T) \le M_2 \kappa^{-2}
\end{equation}
that solve \eqref{kappa_problem} and achieve the initial data $\dt^j u(\cdot,0)$ and $\dt^j \eta(\cdot,0)$ for $j=0,\dotsc,2N$.  Moreover,
\begin{equation}\label{kc_02}
 \Lf[\eta](T) \le    \delta \text{ and } \If[\eta](T) + \Jf[u](T) \ls \tilde{\E}^0[u_0,\eta_0] + 1,
\end{equation}
where $\tilde{\E}^0[u_0,\eta_0]$ is as defined in \eqref{TE_def2}.

\end{thm}

\begin{proof}

Let $(v_i,\zeta_i) \in \XT$, $(u_i,\eta_i) = \M(v_i,\zeta_i) \in \XT$ for $i=1,2$.  Then $u:= u_1 - u_2$ and $\eta := \eta_1 -\eta_2$ satisfy
\begin{equation}\label{kc_1}
\begin{cases}
  \rho \dt u - \diverge_{\a[\eta_1]} \S_{\a[\eta_1]} u = H^1 & \text{in }\Omega \\
  -\S_{\a[\eta_1]}  u \n[\eta_1] =  -\sigma_+ \Delta_\ast \eta \n[\eta_1] + H^2_+  &\text{on } \Sigma_+ \\
  -\jump{\S_{\a[\eta_1]}  u } \n[\eta_1] =  \sigma_- \Delta_\ast \eta \n[\eta_1] + H^2_- &\text{on } \Sigma_- \\
  \jump{u} =0 &\text{on } \Sigma_- \\
  u_- = 0 &\text{on } \Sigma_b \\
  u(\cdot,0) = 0
 \end{cases}
\end{equation}
and
\begin{equation}\label{kc_2}
 \begin{cases}
  \dt \eta - \kappa \Delta_\ast \eta = (v_1 - v_2)\cdot \n[\zeta_1] + v_2 \cdot (\n[\zeta_1] - \n[\zeta_2]) \\
  \eta(\cdot,0) = 0,
 \end{cases}
\end{equation}
where
\begin{equation}
\begin{split}
H^1 &= \diverge_{\a[\eta_1] - \a[\eta_2]} \S_{\a[\eta_2]} u_2 + \diverge_{\a[\eta_1]} \S_{\a[\eta_1] - \a[\eta_2]} u_2  \\
H^2_+ &=   -\sigma_+ \Delta_\ast \eta_2 (\n[\eta_1] - \n[\eta_2])  -\S_{\a[\eta_1]}u_2 (\n[\eta_1] - \n[\eta_2])
 - \S_{\a[\eta_1] - \a[\eta_2]} u_2 \n[\eta_2] \\
H^2_- &=  \sigma_- \Delta_\ast \eta_2 (\n[\eta_1] - \n[\eta_2])  -\jump{\S_{\a[\eta_1]}u_2} \cdot (\n[\eta_1] - \n[\eta_2]) - \jump{\S_{\a[\eta_1] - \a[\eta_2]} u_2 } \n[\eta_2].
\end{split}
\end{equation}
Notice that the $\Xi$ terms cancel in \eqref{kc_2} since they are both computed from the same initial data.

From \eqref{s1_1_6} in Step 1 above, we know that
\begin{multline}
   \sum_{j=0}^{2N}  \int_0^T \ns{\dt^j ((v_1-v_2)\cdot \n[\zeta_1])(t)}_{4N-2j} dt
\\
\ls (1+ \i_\infty[\zeta_1](T)) \sqrt{T} \sqrt{\j_\infty[v_1-v_2](T)} \sqrt{\j_2[v_1-v_2](T)}
\\
\ls \left(1 + \frac{M_2}{\kappa^2}  \right) \sqrt{T} \j[v_1-v_2](T).
\end{multline}
A similar argument shows that
\begin{multline}
   \sum_{j=0}^{2N}  \int_0^T \ns{\dt^j (v_2\cdot (\n[\zeta_1]) -\n[\zeta_2])(t)}_{4N-2j} dt
\\
\ls \i_\infty[\zeta_1-\zeta_2](T) \sqrt{T} \sqrt{\j_\infty[v_2](T)} \sqrt{\j_2[v_2](T)}
\ls \frac{M_1}{\kappa^r} \sqrt{T} \i_\infty[\zeta_1-\zeta_2](T).
\end{multline}
Then, as in \eqref{s1_2}, we have the estimate
\begin{equation}\label{kc_3}
 \i_2[\eta](T) + \i_\infty[\eta](T) \ls  \left(\frac{1}{\kappa^2} + \frac{M_2}{\kappa^4}  \right) \sqrt{T} \j[v_1-v_2](T)
+ \frac{M_1}{\kappa^{r+2}} \sqrt{T} \i_\infty[\zeta_1-\zeta_2](T).
\end{equation}

From \eqref{s2_3} in Step 2 above, we know that
\begin{multline}\label{kc_4}
 \j_2[u](T) \ls (1+ P(\i[\eta](T)   +\Ef[\rho](T))) \exp\left(T(1+P(\i_\infty[\eta](T) +\Ef[\rho](T))) \right) \\
\times \left(  \int_0^T \mathcal{H}_2(t)dt
+(1+\sigma^2)\sqrt{T}  \i[\eta](T)     \right)
\end{multline}
for some universal positive polynomial $P$ with $P(0)=0$, where we have written
\begin{equation}
 \mathcal{H}_2(t) =  \sum_{j=0}^{2N-1} \ns{\dt^j H^1(t)}_{4N-2j-1}    + \ns{\dt^{2N} H^1(t)}_{\Hd}   + \sum_{j=0}^{2N} \ns{\dt^j H^2(t)}_{4N-2j-1/2}.
\end{equation}
Similarly,  \eqref{s2_3} in Step 2 yields
\begin{multline}\label{kc_5}
\j_\infty[u](T) \ls  ( 1+P(\i[\eta](T)    +\Ef[\rho](T) ))\exp\left(T(1+P(\i_\infty[\eta](T) +\Ef[\rho](T))) \right)
\\
\times
\left( \sup_{0\le t \le T} \mathcal{H}_\infty(t)  +  \int_0^T \mathcal{H}_2(t)dt + \i_\infty[\eta](T) +(1+\sigma^2)\sqrt{T}  \i[\eta](T)  \right),
\end{multline}
where we have written
\begin{equation}\
 \mathcal{H}_\infty(t) := \sum_{j=0}^{2N-1} \left[ \ns{\dt^j H^1(t)}_{4N-2j-2} + \ns{\dt^j H^2(t)}_{4N-2j-3/2} \right].
\end{equation}

We now seek to estimate the $\mathcal{H}$ terms appearing on the right side of \eqref{kc_4} and \eqref{kc_5}.  Standard nonlinear estimates lead us to the bounds
\begin{multline}\label{kc_7}
 \int_0^T \mathcal{H}_2(t)dt \ls  \i_\infty[\eta](T)  (1+P(\i_\infty[\eta_1](T) +\i_\infty[\eta_2](T) ) ) \j_2[u_2](T)
+  \i_\infty[\eta](T) \i_2[\eta_2](T)
\\ \ls
P\left(\frac{M_2}{\kappa^r} +\frac{M_1 M_2}{\kappa^{r+2}} + \frac{M_1}{\kappa^2}\right)\i_\infty[\eta](T)
\end{multline}
and
\begin{multline}\label{kc_8}
 \sup_{0\le t \le T} \mathcal{H}_\infty(t) \ls \i_\infty[\eta](T)  (1+\i_\infty[\eta_1](T) +\i_\infty[\eta_2](T)  ) \j_\infty[u_2](T) +  \i_\infty[\eta](T) \i_\infty[\eta_2](T)
\\
\ls P\left(\frac{M_2}{\kappa^r} +\frac{M_1 M_2}{\kappa^{r+2}} + \frac{M_1}{\kappa^2}\right)\i_\infty[\eta](T).
\end{multline}

Now we sum the estimates \eqref{kc_4} and \eqref{kc_5} and then combine the resulting estimate with \eqref{kc_7} and \eqref{kc_8} to deduce that
\begin{multline}\label{kc_9}
 \j[u](T) \ls \left( 1+ P\left(\frac{M_2}{\kappa^r} \right)\right) \exp\left(T \left(1+ P\left(\frac{M_2}{\kappa^2}\right) \right) \right)
\\
\times \left[
 P \left(\frac{M_2}{\kappa^r} +\frac{M_1 M_2}{\kappa^{r+2}} + \frac{M_1}{\kappa^2}\right)\i_\infty[\eta](T)
+(1+\sigma^2)\sqrt{T}  \i[\eta](T)
   \right]
\end{multline}
for some universal positive polynomial with $P(0)=0$.  Combining  \eqref{kc_3} and \eqref{kc_9} then yields the estimate
\begin{multline}\label{kc_10}
  \j[u](T) +  \i[\eta](T)   \ls    \left( 1+ P\left(\frac{M_2}{\kappa^r} \right)\right) \exp\left(T \left(1+ P\left(\frac{M_2}{\kappa^2}\right)\right) \right) \\
  \times
\left(  P\left( \frac{M_2}{\kappa^r} +\frac{M_1 M_2}{\kappa^{r+2}} + \frac{M_1}{\kappa^2} \right)  + (1+\sigma^2)\sqrt{T} \right)
\\
\times
\left[
 \left(\frac{1}{\kappa^2} + \frac{M_2}{\kappa^4}  \right) \sqrt{T} \j[v_1-v_2](T)
+ \frac{M_1}{\kappa^{r+2}} \sqrt{T} \i_\infty[\zeta_1-\zeta_2](T)
\right].
\end{multline}
Finally, from \eqref{kc_10} we see that if we further restrict $T$ in terms of $M_1,$ $M_2,$ $\ell,$ $\sigma,$ and $\kappa$, then
\begin{equation}
\j[u_1-u_2](T) +  \i[\eta_1-\eta_2](T)  \le \hal \left( \j[v_1-v_2](T) +  \i[\zeta_1-\zeta_2](T) \right),
\end{equation}
which implies that the mapping $\M : \XT \to \XT$ is a contraction.

The existence of a fixed point satisfying \eqref{kc_01} is then an easy consequence.  The estimates \eqref{kc_02} follow from \eqref{kc_01}, Lemma \ref{time_interp}, and standard estimates of the data.
\end{proof}

\section{Estimates for the $\kappa-$problem \eqref{kappa_problem}} \label{sec_kappa_est}

Our goal in this section is to derive $\kappa-$independent estimates for the problem \eqref{kappa_problem}.  These will eventually allow us to pass to the limit $\kappa \to 0$.  In this section and the next we must work with a slightly weaker form of the dissipation for $u$, which is defined as
\begin{equation}\label{dissipation_weak_u}
 \check{\D}[u] =  \ns{u}_{4N} + \ns{\nab_{\ast,0}^{4N-1} u}_{2} + \sum_{j=1}^{2N} \ns{\dt^j u }_{4N-2j+1}.
\end{equation}
Here and at several points in this section we employ the notational convention
\begin{equation}\label{horiz_sum}
 \ns{\nab_{\ast,0}^{m} u}  = \sum_{\substack{ \alpha \in \mathbb{N}^2 \\ \abs{\alpha} \le m }} \ns{\pal u}
\end{equation}
for any $m\ge 0$ and any norm $\norm{\cdot}$.

Recall that $\Xi$ is the function defined in \eqref{Xi_prop}--\eqref{Xi_bound} in terms of the data.  We will need to refer to the following functional:
\begin{equation}\label{Xi_fnal}
 \zf := \sum_{j=0}^{2N-1} \sup_{t>0} \ns{\dt^j \Xi(t)}_{4N-2j-1} + \sum_{j=0}^{2N} \int_0^\infty \ns{\dt^j \Xi(t)}_{ 4N-2j}dt.
\end{equation}
It's easy to see from \eqref{Xi_bound} that
\begin{equation}\label{Xi_data_est}
  \zf \ls \frac{1}{\sigma} \ns{\sqrt{\sigma} \nab_\ast \eta_0}_{4N} + P(\TE[u_0,\eta_0])
\end{equation}
for a positive universal polynomial $P$ such that $P(0)=0$.

\subsection{Preliminaries}

 Rather than work directly with the solutions from Theorem \ref{kappa_contract} we will prove our estimates in a somewhat more general context.  We assume the following for some parameter $\delta >0$ and time interval $[0,T_\ast]$.
\begingroup
\allowdisplaybreaks
\begin{align}
  &\text{The parameter } \delta>0 \text{ satisfies } \delta \le \delta_0 \text{ where }\delta_0 \text{ is given by Proposition } \ref{lame_elliptic}. \label{kap_assump_1}\\
  &\text{The initial data } \dt^j u(\cdot,0) \text{ and }\dt^j \eta(\cdot,0) \text{ for }j=0,\dotsc,2N \text{ are given as before and}  \\
& \quad\text{  satisfy the compatibility conditions  } \eqref{kappa_ccs}. \nonumber \\
&  \text{The data satisfy }\Lf[\eta_0] \le \delta/2, \text{ where }\Lf[\eta_0] \text{ is given by } \eqref{l0_def}. \\
&  \text{A  solution }(u,\eta) \text{ to }\eqref{kappa_problem} \text{ exists on the interval } [0,T_\ast] \text{ and achieves the initial data}. \\
&  \text{The solution satisfies the estimate }\j[u](T_\ast) + \i[\eta](T_\ast) < \infty, \text{ where } \j[u] \text{ and } \i[\eta] \nonumber \\
&\quad\text{are as defined in }\eqref{ab_def}. \label{kap_assump_2} \\
&  \text{We have the estimates }   \Lf[\eta](T_\ast) < \delta \text{ and }\If[\eta](T_\ast) + \Jf[u](T_\ast) \ls P(\TE) + 1  \nonumber \\
&\quad \text{for a universal positive polynomial }P\text{ such that }P(0)=0. \text{ Here } \If[\eta], \Jf[u]  \nonumber \\
&\quad\text{are defined by } \eqref{ij_def}. \\
& \text{The forcing terms satisfy } \sup_{0\le t\le T_\ast} \f_\infty(t) + \int_0^{T_\ast} \f_2(t)dt \ls 1+ P(\TE) \nonumber \\
& \quad \text{for a universal positive polynomial }P\text{ such that }P(0)=0.  \label{kap_assump_4}\\
& \text{The function } \rho \text{ is defined on }[0,T_\ast] \text{ and satisfies } \eqref{rho_assump_1} \text{ and } \eqref{rho_assump_2}. \label{kap_assump_3}
\end{align}
\endgroup

The assumption \eqref{kap_assump_2}  guarantees that $(u,\eta)$ are regular enough for us to apply $\pal$ to \eqref{kappa_problem} for $\alpha \in \mathbb{N}^{1+2}$ with $\abs{\alpha} \le 4N$.  This results in

\begin{equation}\label{kappa_deriv}
\begin{cases}
  \rho \dt \pal u - \diva \Sa \pal u = \pal F^1 + F^{1,\alpha} & \text{in }\Omega \\
  \dt \pal \eta - \kappa \Delta_\ast \pal \eta = (\pal u) \cdot \n  - \kappa \pal \Xi + F^{3,\alpha} &\text{on }\Sigma \\
  -\Sa \pal u \n = - \sigma_+ \Delta_\ast \pal\eta  \n +  \pal F^2_+ + F^{2,\alpha}_+  &\text{on } \Sigma_+ \\
  -\jump{\Sa \pal u } \n = \sigma_- \Delta_\ast \pal \eta  \n - \pal F^2_- - F^{2,\alpha}_- &\text{on } \Sigma_- \\
  \jump{\pal u} =0 &\text{on } \Sigma_- \\
  \pal u_- = 0 &\text{on } \Sigma_b,
 \end{cases}
\end{equation}
where for $i=1,2,3$,
\begin{multline}\label{kF1_def}
 F^{1,\alpha}_i = -\sum_{\beta < \alpha} C_{\alpha,\beta} \p^{\alpha-\beta} \rho \dt \p^\beta u_i + \sum_{\beta < \alpha} C_{\alpha,\beta} \p^{\alpha-\beta} \a_{\ell m} \p^\beta \p_m(\mu \a_{\ell k} \p_k u_i)
\\
+ \sum_{\beta < \alpha} C_{\alpha,\beta} \left( \mu' + \frac{\mu}{3}\right) \left( \a_{ik} \p_k (\p^{\alpha-\beta} \a_{\ell m} \p^\beta \p_m u_\ell ) + \p^{\alpha-\beta} \a_{ik} \p^{\beta} \p_k(\a_{\ell m} \p_m u_\ell)         \right),
\end{multline}
\begin{multline}
 F^{2,\alpha}_{+,i} = \sum_{\beta < \alpha} C_{\alpha,\beta} \left( \mu \p^{\alpha-\beta} (\n_\ell \a_{ik}) \p^\beta \p_k u_\ell + \mu \p^{\alpha-\beta}(\n_\ell \a_{\ell k} )\p^\beta \p_k u_i  \right)
\\
+    \sum_{\beta < \alpha} C_{\alpha,\beta} \left(\left( \mu' - \frac{2\mu}{3}\right) \p^{\alpha-\beta}(\n_i \a_{\ell k}) \p^\beta \p_k u_\ell  - \sigma_+ \p^{\alpha-\beta}\n_i \p^\beta  \Delta_\ast \eta     \right)
\end{multline}
and
\begin{multline}
 -F^{2,\alpha}_{-,i} =   \sum_{\beta < \alpha} C_{\alpha,\beta} \left( \ \p^{\alpha-\beta} (\n_\ell \a_{ik}) \jump{\mu \p^\beta \p_k u_\ell} +  \p^{\alpha-\beta}(\n_\ell \a_{\ell k} ) \jump{\mu \p^\beta \p_k u_i } \right)
\\
+    \sum_{\beta < \alpha} C_{\alpha,\beta} \left(\p^{\alpha-\beta}(\n_i \a_{\ell k}) \jump{\left( \mu' - \frac{2\mu}{3}\right)  \p^\beta \p_k u_\ell} +\sigma_- \p^{\alpha-\beta}\n_i \p^\beta \Delta_\ast \eta    \right).
\end{multline}
Also,
\begin{equation}\label{kF3_def}
 F^{3,\alpha} = \sum_{\beta < \alpha} C_{\alpha,\beta} \p^{\alpha-\beta} \n \p^\beta u.
\end{equation}

\subsection{The estimates}

We begin with a basic energy identity.

\begin{lem}\label{kappa_en_ident}
Assume \eqref{kap_assump_1}--\eqref{kap_assump_3}.  Let $\alpha \in \mathbb{N}^{1+2}$ satisfy $\abs{\alpha} \le 4N$.  Then
\begin{multline}\label{kei_01}
 \ddt \left(  \int_\Omega \frac{\rho J}{2} \abs{\pal u}^2 + \int_{\Sigma} \frac{1}{2} \abs{\pal \eta}^2 + \frac{\sigma}{2} \abs{\nab_\ast \pal \eta}^2   \right)
+ \int_\Omega \frac{\mu J}{2} \abs{\sgz_{\a} \pal u}^2 + J \mu' \abs{\diva \pal u}^2
\\
+ \kappa \int_{\Sigma}  \abs{\nab_\ast \pal \eta}^2 + \sigma \abs{\Delta_\ast \pal \eta}^2
= \int_\Omega \frac{\dt(\rho J)}{\rho J} \frac{\rho J}{2} \abs{\pal u}^2  + J \pal F^1 \cdot \pal u - \int_\Sigma \pal F^2 \cdot \pal u \\
+ \int_\Omega J F^{1,\alpha} \cdot \pal u - \int_\Sigma F^{2,\alpha} \cdot \pal u +  \int_{\Sigma} \pal \eta \pal u \cdot \n
+  \int_{\Sigma} (\kappa \pal \Xi - F^{3,\alpha} ) (- \pal  \eta + \sigma \Delta_\ast \pal  \eta).
\end{multline}
\end{lem}
\begin{proof}
We multiply the first equality in \eqref{kappa_deriv} by $J\pal u$ and integrate over $\Omega$.  After integrating by parts and using the boundary conditions, we find that
\begin{multline}\label{kei_1}
 \ddt \left(  \int_\Omega \frac{\rho J}{2} \abs{\pal u}^2 + \int_{\Sigma}  \frac{\sigma}{2} \abs{\nab_\ast \pal \eta}^2  \right)
+ \int_\Omega \frac{\mu J}{2} \abs{\sgz_{\a} \pal u}^2 + J \mu' \abs{\diva \pal u}^2
+ \kappa \int_{\Sigma} \sigma \abs{\Delta_\ast \pal\eta}^2
\\
= \int_\Omega \frac{\dt(\rho J)}{\rho J} \frac{\rho J}{2} \abs{\pal u}^2  + J \pal F^1 \cdot \pal u - \int_\Sigma \pal F^2 \cdot \pal u
\\+ \int_\Omega J F^{1,\alpha} \cdot \pal u - \int_\Sigma F^{2,\alpha} \cdot \pal u
+  \int_{\Sigma} (\kappa \pal \Xi - F^{3,\alpha} )  \sigma  \Delta_\ast \pal  \eta.
\end{multline}
Also, we multiply the second equality in \eqref{kappa_deriv}  by $\pal \eta$ and integrate over $\Sigma$.   After again integrating by parts, we find that
\begin{equation}\label{kei_2}
 \ddt \int_{\Sigma} \hal \abs{\pal \eta}^2 + \kappa \int_{\Sigma} \abs{\nab_\ast \pal \eta}^2 = \int_{\Sigma} \pal \eta (\pal u \cdot \n) + (-\kappa \pal \Xi    + F^{3,\alpha})   \pal \eta.
\end{equation}
Summing \eqref{kei_1} and \eqref{kei_2} yields \eqref{kei_01}.
\end{proof}

Our next result provides some estimates of the  forcing terms that appear in the equations \eqref{kappa_deriv} as a result of commutators with $\p^\alpha$.  The proof is similar to that of Lemma \ref{lame_forcing_est} and the $H$ forcing estimates of Theorem \ref{kappa_contract}, and is similarly omitted.  We recall that $\Ef_\infty[\rho]$ is defined in \eqref{rho_assump_2}.

\begin{lem}\label{k_F_ests}
Let $\alpha \in \mathbb{N}^{1+2}$ satisfy $\abs{\alpha} \le 4N$ and let $F^{i,\alpha}$ be given by \eqref{kF1_def}--\eqref{kF3_def}.  Then there exists a polynomial $P$ with universal positive coefficients and $P(0)=0$ such that
\begin{equation}\label{kfE_01}
 \ns{F^{1,\alpha}}_0 + \ns{F^{2,\alpha}}_{-1/2}
 \ls   ( P(\E^0[\eta]) + \Ef_\infty[\rho]) (\E[u] + \E^0[\eta]) + P(\Lf[\eta]) \check{\D}[u].
\end{equation}
Additionally,
\begin{equation}\label{kfE_02}
  \ns{(F^{3,\alpha} + (\pal \nab_\ast \eta) \cdot u ) }_{0} \ls \E^0[\eta] \E[u] + \Lf[\eta]\check{\D}[u],
\end{equation}
and
\begin{equation}\label{kfE_03}
 \ns{\nab_\ast (F^{3,\alpha} + (\pal \nab_\ast \eta) \cdot u ) }_0 \ls  \sigma^{-1}\E^\sigma[\eta] \E[u] + \Lf[\eta]\check{\D}[u].
\end{equation}


\end{lem}

With this lemma in hand, we can turn the energy identity of Lemma \ref{kappa_en_ident} into a useful estimate.  Note that in this proposition we employ the notation defined in \eqref{horiz_sum}.

\begin{prop}\label{k_hor_est}
Assume \eqref{kap_assump_1}--\eqref{kap_assump_3}.  Then for $0 \le T \le T_\ast$ we have the estimate
\begin{multline}\label{khe_01}
  \sum_{j=0}^{2N} \ns{\dt^j \nab_{\ast,0}^{4N-2j} u}_{L^\infty_T H^0} + \ns{\dt^j \nab_{\ast,0}^{4N-2j} u}_{L^2_T H^1} + \ns{\dt^j  \eta}_{L^\infty_T H^{4N-2j}} + \sigma \ns{\dt^j \nab_\ast  \eta}_{L^\infty_T H^{4N-2j}} \\
+ \kappa \sum_{j=0}^{2N} \ns{\dt^j \nab_\ast \eta}_{L^2_T H^{4N-2j}} + \sigma \ns{\dt^j \Delta_\ast \eta}_{L^2_T H^{4N-2j}}
\ls e^{\gamma T} \TE[u_0,\eta_0]
\\
+ e^{\gamma T} \int_0^T \left(  \f_2(t)  +  ( P(\E^0[\eta]) + \Ef_\infty[\rho]) (\E[u] + \E^0[\eta]) + P(\Lf[\eta]) \check{\D}[u]
  \right) dt,
\end{multline}
where $ \gamma = C \left( 1 +   \If[\eta](T) + \Ef[\rho](T) + \Jf[u](T) )  \right)$, and $P$ is a  polynomial with positive universal coefficients satisfying $P(0)=0$.

\end{prop}

\begin{proof}

Let $\alpha \in \mathbb{N}^{1+2}$ satisfy $\abs{\alpha} \le 4N$.  Taking \eqref{kei_01} of Lemma \ref{kappa_en_ident} as our starting point, we seek to estimate each term on the right hand side and to estimate non-time-derivative term on the left.  In our subsequent analysis we will rewrite $F^{3,\alpha} = (F^{3,\alpha} + (\pal \nab_\ast \eta)\cdot u) - (\pal \nab_\ast \eta)\cdot u$.

First note that we may use Proposition \ref{korn} to bound
\begin{equation}\label{khe_1}
\ns{\pal u}_{1} \ls  \int_\Omega \frac{\mu J}{2} \abs{\sgz_{\a} \pal u}^2 + J \mu' \abs{\diva \pal u}^2.
\end{equation}
This will allow us to absorb most of the $\pal u$ terms appearing on the right side of \eqref{kei_01}.   Using the definition of $\f_2$, given in \eqref{F2_def}, in conjunction with the estimates \eqref{kfE_01} and \eqref{kfE_02} of Lemma \ref{k_F_ests} and the bound $0 < \kappa <1$, we may  bound
\begin{multline}\label{khe_3_5}
 \ns{\pal F^1}_0 + \ns{\pal F^2}_0  + \ns{F^{1,\alpha}}_0 + \ns{F^{2,\alpha}}_{-1/2}
+ \ns{(F^{3,\alpha} + (\pal \nab_\ast \eta) \cdot u ) }_{0} \ + \kappa^2 \ns{\pal \Xi}_0
\\
\ls  \f_2(t)
+  ( P(\E^0[\eta]) + \Ef_\infty[\rho]) (\E[u] + \E^0[\eta]) + P(\Lf[\eta]) \check{\D}[u]
+ \kappa \sum_{j=0}^{2N} \ns{\dt^j \Xi}_{4N-2j},
\end{multline}
where $P$ is a universal positive polynomial with $P(0)=0$.  Next, we use Cauchy's inequality, Proposition \ref{korn}, trace theory, and the estimates $\ns{J}_{L^\infty} + \ns{\n}_{L^\infty} \ls 1$ (which follow from Lemma \ref{eta_small} and the bounds on $\Lf[\eta](T)$)  and \eqref{khe_3_5} to bound, for any $0 < \ep_1 < 1$,
\begin{multline}\label{khe_2}
\int_\Omega J \pal F^1 \cdot \pal u - \int_\Sigma \pal F^2 \cdot \pal u + \int_\Omega J F^{1,\alpha} \cdot \pal u - \int_\Sigma F^{2,\alpha} \cdot \pal u \\
+  \int_{\Sigma} \pal \eta (\pal u \cdot \n)
+  \int_{\Sigma} (\kappa \pal \Xi - (F^{3,\alpha} + (\pal \nab_\ast \eta) \cdot u ))  \pal  \eta
\\
\ls \ep_1 \ns{\pal u}_1
+\frac{1}{\ep_1} \left( \ns{ \pal \eta}_0 +  \ns{\pal F^1}_0 + \ns{\pal F^2}_0  + \ns{F^{1,\alpha}}_0 + \ns{F^{2,\alpha}}_{-1/2}   \right)
 \\
+ \ns{F^{3,\alpha} + (\pal \nab_\ast \eta) \cdot u }_0 + \kappa^2 \ns{\pal \Xi}_0
\ls \ep_1 \ns{\pal u}_1
+\frac{1}{\ep_1} \ns{ \pal \eta}_0
\\
 +\frac{1}{\ep_1}  \left(
\f_2(t)  + ( P(\E^0[\eta]) + \Ef_\infty[\rho]) (\E[u] + \E^0[\eta]) + P(\Lf[\eta]) \check{\D}[u]
 + \kappa\sum_{j=0}^{2N} \ns{\dt^j \Xi}_{4N-2j}
\right).
\end{multline}

For the remaining $\Xi$ terms on the right side of \eqref{kei_01} we estimate
\begin{equation}\label{khe_3}
  \int_{\Sigma} \kappa \pal \Xi  \sigma \Delta_\ast \pal  \eta \le \hal \int_\Sigma \kappa \sigma \abs{\Delta_\ast \pal \eta}^2 +
  \frac{\kappa\sigma}{2}\sum_{j=0}^{2N} \ns{\dt^j \Xi}_{4N-2j}
\end{equation}
with the aim of absorbing the $\Delta_\ast \pal \eta$ term onto the left of \eqref{kei_01}.   Next we handle the $(F^{3,\alpha} + (\pal \nab_\ast \eta)\cdot u)  ( \sigma \Delta_\ast \pal  \eta)$ terms on the right side of \eqref{kei_01}.   We integrate by parts and then use \eqref{kfE_03} of Lemma \ref{k_F_ests} to see that
\begin{multline}\label{khe_4}
 \int_\Sigma -(F^{3,\alpha} + (\pal \nab_\ast \eta)\cdot u)  \sigma \Delta_\ast \pal  \eta =   \int_\Sigma \nab_\ast (F^{3,\alpha} + (\pal \nab_\ast \eta)\cdot u) \cdot  \sigma \nab_\ast \pal  \eta
\\
\le \frac{\sigma}{2} \ns{\nab_\ast (F^{3,\alpha} + (\pal \nab_\ast \eta)\cdot u)}_0 +  \int_\Sigma \frac{\sigma}{2} \abs{\nab_\ast \pal \eta}^2 \ls \E^0[\eta] \E[u] + \Lf[\eta]\check{\D}[u]+ \int_\Sigma \frac{\sigma}{2} \abs{\nab_\ast \pal \eta}^2.
\end{multline}

It remains to handle the term
\begin{equation}\label{khe_5}
 \int_{\Sigma}  (\pal \nab_\ast \eta)\cdot u  ( - \pal  \eta + \sigma  \Delta_\ast \pal  \eta) .
\end{equation}
We have that
\begin{multline}\label{khe_6}
 -\int_{\Sigma }  (\pal \nab_\ast \eta)\cdot u   \pal  \eta
=- \int_{\Sigma }    \frac{1}{2}  u\cdot \nab_\ast \abs{\pal  \eta}^2
= \int_{\Sigma }   \frac{1}{2}  \diverge_\ast u    \abs{\pal  \eta}^2
\\
\ls \norm{\diverge_\ast u}_{L^\infty} \int_\Sigma \frac{1}{2} \abs{\pal \eta}^2 \ls  \Jf[u](T) \int_\Sigma  \frac{1}{2} \abs{\pal \eta}^2.
\end{multline}
Similarly,
\begin{multline}\label{khe_7}
 \int_{\Sigma}  (\pal \nab_\ast \eta)\cdot u  \sigma \Delta_\ast \pal  \eta = -  \sum_{i,j=1}^2\int_{\Sigma}  \sigma \p_i \pal \eta  (\p_i u_j \p_j \pal \eta + u_j \p_i \p_j \pal \eta)
\\ =  - \int_{\Sigma} \sigma u\cdot \nab_\ast \frac{\abs{\nab_\ast \pal \eta}^2 }{2} +  \sum_{i,j=1}^2  \sigma \p_i \pal \eta  \p_i u_j \p_j \pal \eta
\\
\ls \norm{\nab_\ast u}_{L^\infty} \int_\Sigma \frac{\sigma}{2} \abs{\nab_\ast \pal \eta}^2 \ls \Jf[u](T)  \ \int_\Sigma \frac{\sigma}{2} \abs{\nab_\ast \pal \eta}^2.
\end{multline}
Combining \eqref{khe_5}--\eqref{khe_7}, we find that
\begin{equation}\label{khe_8}
 \int_{\Sigma} (\pal \nab_\ast \eta)\cdot u  (- \pal  \eta + \sigma \Delta_\ast \pal  \eta)
\ls  \Jf[u](T)  \ \int_\Sigma  \frac{1}{2} \abs{\pal \eta}^2 + \frac{\sigma}{2} \abs{ \nab_\ast \pal \eta}^2.
\end{equation}

Now we employ the estimates \eqref{khe_1}, \eqref{khe_2}, \eqref{khe_3}, \eqref{khe_4}, and \eqref{khe_8}  in \eqref{kei_01}.  By choosing $\ep_1$ small enough (but universal), we may absorb the $\ep_1 \ns{\pal u}_1$ term in \eqref{khe_2} onto the left.  This results in the differential inequality
\begin{multline}\label{khe_9}
 \ddt \left(  \int_\Omega \frac{\rho J}{2} \abs{\pal u}^2 + \int_{\Sigma} \frac{1}{2} \abs{\pal \eta}^2 + \frac{\sigma}{2} \abs{\nab_\ast \pal \eta}^2   \right)
+ C \ns{\pal u}_{1}
+ \frac{\kappa}{2} \int_{\Sigma}   \abs{\nab_\ast \pal \eta}^2 + \sigma \abs{\Delta_\ast \pal \eta}^2
\\
\ls \left(1 + \norm{\dt(\rho J)(\rho J)^{-1}}_{L^\infty} + \Jf[u](T) \right) \left( \int_\Omega \frac{\rho J}{2} \abs{\pal u}^2 + \int_\Sigma \frac{1}{2} \abs{\pal \eta}^2 +\frac{\sigma}{2} \abs{\nab_\ast \pal \eta}^2 \right)
\\
+\f_2(t)  +  ( P(\E^0[\eta]) + \Ef_\infty[\rho]) (\E[u] + \E^0[\eta]) + P(\Lf[\eta]) \check{\D}[u]
 +\kappa \sum_{j=0}^{2N} \ns{\dt^j \Xi}_{4N-2j}.
\end{multline}
Notice that
\begin{equation}
\sup_{0\le t \le T}  \norm{\dt(\rho J)(\rho J)^{-1}}_{L^\infty} \ls \If[\eta](T) + \Ef[\rho](T).
\end{equation}
We may then apply Gronwall's inequality to \eqref{khe_9} and sum over $\alpha$ to deduce the bound
\begin{multline}\label{khe_10}
  \sum_{j=0}^{2N} \ns{\dt^j \nab_{\ast,0}^{4N-2j} u}_{L^\infty_T H^0} + \ns{\dt^j \nab_{\ast,0}^{4N-2j} u}_{L^2_T H^1} + \ns{\dt^j  \eta}_{L^\infty_T H^{4N-2j}} + \sigma \ns{\dt^j \nab_\ast  \eta}_{L^\infty_T H^{4N-2j}} \\
+ \kappa \sum_{j=0}^{2N} \ns{\dt^j \eta}_{L^2_T H^{4N-2j}} + \sigma \ns{\dt^j \nab_\ast \eta}_{L^2_T H^{4N-2j}}
\ls e^{\gamma T}\TE[u_0,\eta_0]
\\
+ e^{\gamma T} \int_0^T \left(  \f_2(t)  +  ( P(\E^0[\eta]) + \Ef_\infty[\rho]) (\E[u] + \E^0[\eta]) + P(\Lf[\eta]) \check{\D}[u] +\kappa \sum_{j=0}^{2N} \ns{\dt^j \Xi}_{4N-2j}  \right) dt,
\end{multline}
where $ \gamma = C \left( 1 +   \If[\eta](T) + \Ef[\rho](T) + \Jf[u](T) )  \right). $  Then \eqref{khe_01} follows from \eqref{khe_10} and \eqref{Xi_bound}.
\end{proof}

Next we employ various elliptic estimates to gain control of all derivatives of $(u,\eta)$.

\begin{prop}\label{kappa_improved}
Assume \eqref{kap_assump_1}--\eqref{kap_assump_3}.  Then for $0 \le T \le T_\ast$ we have the estimate
\begin{multline}\label{ki_01}
  \sum_{j=0}^{2N} \ns{\dt^j  u}_{L^\infty_T H^{4N-2j}} +   \ns{u}_{L^2_T H^{4N}} +  \ns{\nab_{\ast,0}^{4N-1} u }_{L^2_T H^{2}} +\sum_{j=1}^{2N} \ns{\dt^j u}_{L^2_T H^{4N-2j+1}}
  \\
  +  \ns{\rho J \dt^{2N+1} u}_{L^2_T \Hd} +
 \sum_{j=0}^{2N} \ns{\dt^j  \eta}_{L^\infty_T H^{4N-2j}} +  \sum_{j=0}^{2N} \sigma \ns{\dt^j \nab_\ast  \eta}_{L^\infty_T H^{4N-2j}} \\
  + \sigma^2 \ns{\eta}_{L^2_T H^{4N+3/2}}
  + \ns{\dt \eta}_{L^2_T H^{4N-1/2}} + \sum_{j=2}^{2N+1} \ns{\dt^j \eta}_{L^2_T H^{4N-2j+2}}
\\+ \kappa \sum_{j=0}^{2N} \ns{\dt^j \nab_\ast \eta}_{L^2_T H^{4N-2j}}
+ \sigma \ns{\dt^j \Delta_\ast \eta}_{L^2_T H^{4N-2j}}
\ls e^{\gamma T}\TE[u_0,\eta_0]
+\sup_{0\le t \le T} \f_\infty(t)  \\
+ T \sup_{0 \le t \le T} \E^0[\eta(t)] + e^{\gamma T} \int_0^T \left(  \f_2(t)  +  ( P(\E^0[\eta]) + P(\Ef_\infty[\rho])) (\E[u] + \E^0[\eta]) + P(\Lf[\eta]) \D[u]
  \right) dt \\
+\int_0^T   \E[u] \frac{1}{\sigma} \E^\sigma[\eta] dt + \kappa^2 \zf,
\end{multline}
where $P$ is a polynomial with positive universal coefficients such that $P(0)=0$ and $\gamma = C \left( 1 +   \If[\eta](T)+ \Ef[\rho](T) + \Jf[u](T) )  \right).$
\end{prop}

\begin{proof}
The argument is very similar to one used in \cite{JTW_GWP}, so we will only provide a sketch.  Let us write $\z$ to denote a term of the same form as the right hand side of \eqref{khe_01}.  From line to line we will let the polynomials vary as well as the universal constant $C>0$ appearing in $\gamma$.

First we use trace estimates to bound
\begin{equation}\label{ki_20}
 \sum_{j=0}^{2N} \ns{\dt^j u}_{L^2_T H^{4N-2j+1/2}(\Sigma)} \ls \sum_{j=0}^{2N} \ns{\dt^j\nab_{\ast,0}^{4N-2j} u}_{L^2_T H^1} \ls \z.
\end{equation}
Then we use elliptic regularity for the Lam\'{e} system with Dirichlet boundary conditions, Proposition \ref{lame_elliptic}, to deduce that
\begin{equation}\label{ki_11}
 \ns{u}_{L^2_T H^{4N}} + \sum_{j=1}^{2N}   \ns{\dt^j u}_{L^2_T H^{4N-2j+1}} \ls \z.
\end{equation}
Notice here that we have only used the estimate \eqref{ld_01} in order to avoid the introduction of the term $\ns{\eta}_{4N+1/2}$.  This causes no difficulty in getting $4N-2j+1$ estimates when  $j=1,\dotsc,2N$ but prevents us from obtaining an estimate of $\ns{u}_{4N+1}$ at this point.  Instead we sum the elliptic estimate for $\p^\alpha u$ obtained from \eqref{kappa_deriv} with $\alpha \in \mathbb{N}^2$ and $\abs{\alpha}\le 4N-1$ to obtain the estimate
\begin{equation}\label{ki_19}
 \ns{\nab_{\ast,0}^{4N-1} u}_2 \ls \z .
\end{equation}

Now we derive the $L^\infty$ in time estimate for $u$ and its time derivatives.  Lemma \ref{time_interp} provides us with  $L^\infty$ in time estimates at lower regularity:
\begin{equation}
 \sum_{j=1}^{2N-1} \ns{\dt^j u}_{L^\infty_T H^{4N-2j}} \ls \TE[u_0] +  \sum_{j=1}^{2N} \ns{\dt^j u}_{L^2_T H^{4N-2j+1}},
\end{equation}
and hence \eqref{ki_11} and the bound of $\ns{\dt^{2N} u }_{L^\infty_T H^0}$ provided by \eqref{khe_01}  imply that
\begin{equation}\label{ki_5}
 \sum_{j=1}^{2N} \ns{\dt^j u}_{L^\infty_T H^{4N-2j}} \ls \z.
\end{equation}
Similarly, Lemma \ref{time_interp} and \eqref{ki_20} imply that
\begin{equation}\label{ki_21}
   \ns{u}_{L^\infty_T H^{4N-1/2}(\Sigma)} \ls \TE[u_0] + \sum_{j=0}^{1} \ns{\dt^j u}_{L^2_T H^{4N-2j+1/2}(\Sigma)} \ls \z.
\end{equation}
We may then use \eqref{ki_21} and  the elliptic estimate of Proposition \ref{lame_elliptic} to bound
\begin{multline}\label{ki_22}
 \ns{u}_{L^\infty_T H^{4N}} \ls \ns{\rho \dt u}_{L^\infty_T H^{4N-2}} + \ns{F^1}_{L^\infty_T H^{4N-2}} + \ns{u}_{L^\infty_T H^{4N-1/2}(\Sigma)} \\
 \ls \Ef[\rho](T)   \ns{ \dt u}_{L^\infty_T H^{4N-2}} +\sup_{0\le t \le T} \f_\infty(t) +  \ns{u}_{L^\infty_T H^{4N-1/2}(\Sigma)}  \\
 \ls (1+\Ef[\rho](T)) \z +\sup_{0\le t \le T} \f_\infty(t)  \ls \z +\sup_{0\le t \le T} \f_\infty(t).
\end{multline}
Note that in the last inequality we have used the bound $\Ef[\rho](T) \z \ls \z$, which holds since we may increase the constant $C>0$ appearing in $\gamma$.

Next we derive the $L^2$ in time estimates for $\eta$ and its derivatives.  We first use the dynamic boundary condition and \eqref{ki_19} to estimate
\begin{multline}\label{ki_1}
\sigma^2  \ns{ \eta}_{L^2_T H^{4N+3/2}} \ls \sigma^2 \ns{\eta}_{L^2_T H^0} + \sigma^2  \ns{\Delta_\ast  \eta}_{L^2_T H^{4N-1/2}} \\
 \ls T \sup_{0\le t \le T} \E^0[\eta(t)] + \z   + \int_0^T \left((1+P(\E^0[\eta]) \ns{\nab_\ast \eta}_{4N-1/2}  \E[u]\right) dt \\
\ls  T \sup_{0\le t \le T} \E^0[\eta(t)] + \z   + \int_0^T \left((1+P(\E^0[\eta]) \frac{1}{\sigma} \E^\sigma[\eta]  \E[u]\right) dt.
\end{multline}
Then we use the parabolic modification of the kinematic boundary condition to get improved $L^2$ in time estimates for $\dt^j \eta$ for $j\ge 1$.  Standard heat equation estimates yield the bounds
\begin{equation}\label{ki_12}
 \ns{\dt \eta}_{L^2_T H^{4N-1}} + \sum_{j=2}^{2N+1} \ns{\dt^j \eta}_{L^2_T H^{4N-2j+2}} \ls \z + \int_0^T   \E[u] \frac{1}{\sigma} \E^\sigma[\eta] dt + \kappa^2 \zf.
\end{equation}

We then combine the estimates \eqref{khe_01} of Proposition \ref{k_hor_est} with \eqref{ki_11}, \eqref{ki_19}, \eqref{ki_5}, \eqref{ki_22}, \eqref{ki_1}, and \eqref{ki_12}  to deduce all estimates in \eqref{ki_01}  except that of $\ns{\rho J \dt^{2N+1} u}_{L^2_T \Hd}$.  To recover this estimate we simply appeal to \eqref{kappa_deriv} with $\pa = \dt^{2N}$.  A standard duality argument then allows us to estimate $\ns{\rho J \dt^{2N+1} u}_{L^2_T \Hd}$ in terms of all of the existing terms in \eqref{ki_01}; this yields \eqref{ki_01} in the form written.
\end{proof}

Finally, we combine Propositions \ref{k_hor_est} and \ref{kappa_improved}  in order to obtain $\kappa-$independent estimates.

\begin{thm}\label{kappa_apriori}
Assume \eqref{kap_assump_1}--\eqref{kap_assump_3}.  Assume also that $0<\kappa < \min\{1,\sigma_+,\sigma_-\}$.   There exists a universal constant $0 < \delta_1 < \delta_0 /2$ (where $\delta_0$ is from Proposition \ref{lame_elliptic}) and  a $0 < T_1 = T_1(\TE,\sigma) \le 1$  such that if $0<\delta < \delta_1$ in \eqref{kap_assump_1}--\eqref{kap_assump_3} and $0\le T \le \min\{T_1,T_\ast\}$, then
\begin{multline}\label{ka_01}
 \sup_{0\le t \le T} (\E[u(t)] + \E^\sigma[\eta(t)]) + \int_0^T (\check{\D}[u(t)] +\ns{\rho J \dt^{2N+1} u(t)}_{ \Hd} +\D^\sigma[\eta(t)] )dt
\\
+ \kappa  \sum_{j=0}^{2N} \int_0^T \ns{\dt^j \nab_\ast \eta(t)}_{4N-2j}  +\sigma \ns{\dt^j \Delta_\ast \eta(t) }_{4N-2j}  dt
\ls P(\TE[u_0,\eta_0]) +\sup_{0\le t \le T} \f_\infty(t)  + \int_0^T  \f_2(t) dt
\end{multline}
for a positive universal polynomial $P$ such that $P(0)=0$.
\end{thm}
\begin{proof}
We first notice that
\begin{equation}
 \gamma = C \left( 1 +   \If[\eta](T)+ \Ef[\rho](T) + \Jf[u](T) )  \right) \le C(1 + P(\TE)),
\end{equation}
and so we can make $T_2$ small in terms of $\TE$, $\sigma$, and a universal constant so that $e^{\gamma T} \le 2$, $\sqrt{T} P(\Ef[\rho](T)) \le 1$, and $\sqrt{T}  \le \sigma$ whenever $T \le \min\{T_2,T_\ast\}$.  Then Proposition \ref{kappa_improved} yields the estimate
\begin{multline}\label{ka_1}
 \sup_{0\le t \le T} (\E[u(t)] + \E^\sigma[\eta(t)]) + \int_0^T (\check{\D}[u(t)] +\ns{\rho J \dt^{2N+1} u(t)}_{ \Hd} +\D^\sigma[\eta(t)] )dt
\\
+ \kappa  \sum_{j=0}^{2N} \int_0^T \ns{\dt^j \nab_\ast \eta(t)}_{4N-2j}  + \sigma \ns{\dt^j \Delta_\ast \eta(t) }_{4N-2j}  dt
\le C\TE[u_0,\eta_0] + C \sup_{0\le t\le T} \f_\infty(t)
\\
+ C\int_0^T \f_2(t) dt
+ C\kappa^2 \zf + \sqrt{T} P\left(  \sup_{0\le t \le T} (\E[u(t)] + \E^\sigma[\eta(t)]  \right) +   C\sup_{0\le t \le T} \Lf[\eta(t)] \int_0^T \check{\D}[u(t)]dt
\end{multline}
for every $T \le \min\{T_2, T_\ast\}$, where   $P$ is a universal positive polynomial such that $P(0)=0$, and $C\ge 1$ is a universal constant.

Since
\begin{equation}
 \sup_{0\le t \le T} \Lf[\eta(t)] \le  \delta \le \delta_1,
\end{equation}
we may choose a universal $\delta_1>0$ such that $C \delta_1 =1/2$.  We may then absorb the last term on the right side of \eqref{ka_1} onto the left.  This yields the estimate
\begin{multline}\label{ka_2}
 \sup_{0\le t \le T} (\E[u(t)] + \E^\sigma[\eta(t)]) + \hal \int_0^T (\check{\D}[u(t)] +\ns{\rho J \dt^{2N+1} u(t)}_{ \Hd} +\D^\sigma[\eta(t)] )dt
\\
+ \kappa  \sum_{j=0}^{2N} \int_0^T \ns{\dt^j \nab_\ast \eta(t)}_{4N-2j}  + \sigma\ns{\dt^j \Delta_\ast \eta(t) }_{4N-2j}  dt
\\
\le C\TE[u_0,\eta_0] + C \sup_{0\le t\le T} \f_\infty(t) + C\int_0^T \f_2(t) dt + C \kappa^2 \zf
\\
+ \sqrt{T} P\left(  \sup_{0\le t \le T} (\E[u(t)] + \E^\sigma[\eta(t)]) \right)
\end{multline}
for every $0 \le T \le \min\{T_2,T_\ast\}$.

We may view \eqref{ka_2} abstractly as an inequality of the form
\begin{equation}\label{ka_3}
 X(T) \le C Z(T) + \sqrt{T} P(X(T))
\end{equation}
for $X,Z:[0,\min\{T_2,T_\ast\}] \to [0, \infty)$ continuous functions such that $Z$ is non-decreasing and $X(0) \le C Z(0)$.  By continuity we know that either $X(T) \le 2C Z(T)$ for all $0\le T \le \min\{T_2,T_\ast\}$, or else there exists a first time $0 < T_3 < \min\{T_2,T_\ast\}$ such that $X(T_3) = 2C Z(T_3).$  Plugging this equality into \eqref{ka_3} implies that
\begin{equation}
 2C Z(T_3) = X(T_3) \le C Z(T_3) + \sqrt{T_3} P( X(T_3)) = C Z(T_3) + \sqrt{T_3} P( 2C Z(T_3) ),
\end{equation}
and hence
\begin{equation}
 C Z(T_3) \le \sqrt{T_3}  P( 2C Z(T_3) ).
\end{equation}
From this we deduce that
\begin{equation}
 \sqrt{T_3} \ge \hal  \frac{2 C Z(T_3)}{ P( 2C Z(T_3) )} \ge \hal \min_{z \in [0,2C Z_{max}]} \frac{z}{P(z)},
\end{equation}
where
\begin{equation}
Z_{max} = \max_{0 \le T \le \min\{T_2,T_\ast\}} Z(T) = Z(\min\{T_2,T_\ast\}) \le  P(\TE) + \tilde{C}.
\end{equation}
The last estimate follows because of assumption \eqref{kap_assump_4}, from which the universal constant $\tilde{C} >0$ comes, and \eqref{Xi_data_est} combined with the bound $\kappa < \min\{1,\sigma_+,\sigma_-\}$.  Since $P(0)=0$, we then find that
\begin{equation}
 \sqrt{T_3} \ge  \hal \min_{z \in [0,2C(P({\TE}) + \tilde{C} ) ]} \frac{z}{P(z)} =: \sqrt{T_4} > 0.
\end{equation}
This leads us to define
\begin{equation}
 T_1 =T_1(\TE,\sigma) = \min\{T_2,T_4 \}
\end{equation}
so that if $0 \le T \le \min\{T_1,T_\ast\}$ we have the estimate
\begin{equation}
 X(T) \le 2C Z(T).
\end{equation}
Removing the abstraction, this implies that
\begin{multline}\label{ka_4}
 \sup_{0\le t \le T} (\E[u(t)] + \E^\sigma[\eta(t)]) + \hal \int_0^T (\check{\D}[u(t)] +\ns{\rho J \dt^{2N+1} u(t)}_{ \Hd} +\D^\sigma[\eta(t)] )dt
\\
+ \kappa  \sum_{j=0}^{2N} \int_0^T \ns{\dt^j \nab_\ast \eta(t)}_{4N-2j}  + \sigma \ns{\dt^j \Delta_\ast \eta(t) }_{4N-2j}  dt
\\
\le 2C\TE[u_0,\eta_0] + 2C \sup_{0\le t\le T} \f_\infty(t) + 2C\int_0^T \f_2(t) dt + 2C\kappa^2 \zf
\end{multline}
for $0 \le T \le \min\{T_1,T_\ast\}$.  This and \eqref{Xi_data_est} yield \eqref{ka_01}.
\end{proof}

 \section{The two-phase free boundary Lam\'{e} problem }\label{sec_lame_free}

Our goal in this section is to produce a solution to
\begin{equation}\label{lame_free_bndry}
\begin{cases}
  \rho \dt u - \diva \Sa u = F^1 & \text{in }\Omega\\
  \dt \eta  = u \cdot \n   &\text{on }\Sigma \\
  -\Sa u \n = - \sigma_+ \Delta_\ast \eta \n +  F^2_+   &\text{on } \Sigma_+ \\
  -\jump{\Sa u } \n =   \sigma_- \Delta_\ast \eta  \n - F^2_- &\text{on } \Sigma_- \\
  \jump{u} =0 &\text{on } \Sigma_- \\
  u_- = 0 &\text{on } \Sigma_b \\
  u(\cdot,0) = u_0, \eta(\cdot, 0) = \eta_0
 \end{cases}
\end{equation}
by passing to the limit $\kappa \to 0$ in the $\kappa-$approximation problem \eqref{kappa_problem}.

Our strategy is as follows.  First, we will combine the local existence result of Theorem \ref{kappa_contract} with the a priori estimates of Theorem \ref{kappa_apriori} to produce solutions on a time interval that is independent of $\kappa$ and that satisfy $\kappa-$independent estimates of the form \eqref{ka_01}.  Second, we will pass to the limit $\kappa \to 0$ to recover a solution to \eqref{lame_free_bndry}.

\subsection{Local existence on $\kappa-$independent intervals}

Our goal now is to combine Theorems \ref{kappa_contract} and \ref{kappa_apriori}.  We first describe some assumptions on the data and forcing terms that will be needed.

We say the data and forcing satisfy $\mathfrak{P}(\delta)$ on the time interval $[0,T_\ast]$ if the following hold.
\begingroup
\allowdisplaybreaks
\begin{align}
  &\text{The initial data } \dt^j u(\cdot,0) \text{ and }\dt^j \eta(\cdot,0) \text{ for }j=0,\dotsc,2N \text{ are given as before and}  \label{lame_assump_1}\\
& \quad\text{  satisfy the compatibility conditions } \eqref{kappa_ccs}. \nonumber  \\
&  \text{The data satisfy }\Lf[\eta_0] < \delta/2, \text{ where }\Lf[\eta_0] \text{ is given by }\eqref{l0_def}. \label{lame_assump_2}\\
 & \text{The forcing terms satisfy } \sup_{0\le t\le T_\ast} \f_\infty(t) + \int_0^{T_\ast} \f_2(t)dt \le P(\TE) \nonumber \\
& \quad \text{for a universal positive polynomial }P\text{ such that }P(0)=0.  \label{lame_assump_3}\\
& \text{The function } \rho \text{ is defined on }[0,T_\ast] \text{ and satisfies } \eqref{rho_assump_1} \text{ and } \eqref{rho_assump_2}. \label{lame_assump_4}
\end{align}
\endgroup

\begin{thm}\label{lame_local}
Assume that $0<\kappa < \min\{1,\sigma_+,\sigma_-\}$.  Let $0 < \delta_1$ be the universal constant and $0 < T_1 = T_1(\TE,\sigma)$ be from Theorem \ref{kappa_apriori}. Assume that the data and forcing satisfy $\mathfrak{P}(\delta_1)$ on the time interval $[0,T_\ast]$.  Then there exists $0 < T_2 = T_2(\TE,\sigma) \le T_1(\TE,\sigma)$  such that if $0 < T \le \min\{T_\ast,T_2\}$, then  the following hold.
\begin{enumerate}
 \item A unique  solution $(u,\eta)$  to \eqref{kappa_problem}  exists on the interval $[0,T]$  and achieves the initial data.
 \item The solution satisfies the estimate
\begin{equation}\label{lam_loc_01}
\j[u](T) + \i[\eta](T) < \infty,
\end{equation}
 where $\j[u]$ and  $\i[\eta]$ are as defined in \eqref{ab_def}.
 \item We have the estimate
\begin{multline} \label{lam_loc_02}
 \sup_{0\le t \le T} (\E[u(t)] + \E^\sigma[\eta(t)]) + \int_0^{T}(\check{\D}[u(t)] +\ns{\rho J \dt^{2N+1} u(t)}_{ \Hd} +\D^\sigma[\eta(t)] )dt
\\
+ \kappa  \sum_{j=0}^{2N} \int_0^{T} \ns{\dt^j \nab_\ast \eta(t)}_{4N-2j}  + \sigma\ns{\dt^j \Delta_\ast \eta(t) }_{4N-2j}  dt
\ls P(\TE[u_0,\eta_0]) +\sup_{0\le t \le T} \f_\infty(t)  + \int_0^{T}  \f_2(t) dt
\end{multline}
for a universal positive polynomial $P$ such that $P(0)=0$.  Here we recall the notation $\check{\D}$ defined in \eqref{dissipation_weak_u}.

 \item We also have the estimates
\begin{equation}\label{lam_loc_03}
 \Lf[\eta](T)  < \delta_1
\end{equation}
and
\begin{equation}\label{lam_loc_04}
 \If[\eta](T) + \Jf[u](T) \ls P(\TE) + 1
\end{equation}
for a universal positive polynomial $P$ such that $P(0)=0$, where $\Lf$ is defined by \eqref{l_def} and  $\If$ and $\Jf$ are defined by \eqref{ij_def}.

 \item One of the following is true.  Either $T_2 = T_1$, or else there exists a universal positive polynomial $P$ satisfying $P(0)=0$ and a universal constant $C>0$ such that
\begin{equation}\label{lam_loc_05}
\frac{\sigma \delta_1}{ C (1+ P({\EEE}))} \le T_2.
\end{equation}

\end{enumerate}
\end{thm}

\begin{proof}
For $r >0$ let $\mathfrak{S}(r)$ denote the proposition that the following three conditions hold.  First, a unique solution to \eqref{kappa_problem} exists on $[0,r]$ and achieves the initial data.  Second, the solution satisfies \eqref{lam_loc_01} with $T$ replaced by $r$.  Third, the solution satisfies \eqref{lam_loc_03} and \eqref{lam_loc_04} with $T$ replaced by $r$.  Define the set
\begin{equation}
 \mathfrak{R} = \{ r \in [0,\min\{T_\ast,T_1\}] \;\vert \;  \mathfrak{S}(r) \text{ holds} \}.
\end{equation}
Theorem \ref{kappa_apriori} guarantees that $\delta_1 \le \delta_0/2$, where $\delta_0>0$ is defined by Proposition \ref{lame_elliptic}.  We may then apply Theorem \ref{kappa_contract} to see that $\mathfrak{R} \neq \varnothing$.  Let $T_{\mathfrak{R}} = \sup \mathfrak{R} \in (0, T_\ast]$.

If $T_{\mathfrak{R}} = \min\{T_\ast,T_1\}$ then we set $T_2 = T_1$, and we are done.  Indeed, the hypotheses of Theorem \ref{kappa_apriori} are satisfied, and so \eqref{lam_loc_02} follows.   We may assume then that $T_{\mathfrak{R}}  < \min\{T_\ast,T_1\}$ throughout the rest of the proof.

If $T_{\mathfrak{R}}= \max \mathfrak{R}$, then a standard continuation argument, employing Theorem \ref{kappa_contract} to extend the solutions, yields a contradiction, and so we deduce that $T_{\mathfrak{R}} \notin \mathfrak{R}$.  This means that $\mathfrak{S}(T_{\mathfrak{R}})$ fails.  We claim that in fact it is only the third condition in $\mathfrak{S}(T_{\mathfrak{R}})$ that can fail, i.e. the first two conditions remain true at $T_{\mathfrak{R}}$.

We know that $\mathfrak{S}(T_{\mathfrak{R}} -\ep)$ is true for $\ep$ sufficiently small.  The a priori estimates of Theorem \ref{kappa_apriori} are then valid and provide $\ep-$independent  control of $\E(T_{\mathfrak{R}} -\ep)$ in terms of $\TE$.  We may then use Theorem \ref{kappa_contract} to extend the solutions to $T_3 = T_{\mathfrak{R}} - \ep + T_0(\kappa,\sigma,\delta,\TE)$.  When $\ep$ is sufficiently small we have that $T_3 > T_{\mathfrak{R}}$, and so the first condition of $\mathfrak{S}(T_{\mathfrak{R}})$ must hold.  Additionally, the functional setting of Theorem \ref{kappa_contract} guarantees that the second condition of $\mathfrak{S}(T_{\mathfrak{R}})$ must also hold.  We deduce then that it is the third condition of $\mathfrak{S}(T_{\mathfrak{R}})$ that fails at $T_{\mathfrak{R}}$, proving the claim.

From Lemma \ref{time_interp}, Theorem \ref{kappa_apriori}, and \eqref{lame_assump_3} we may estimate
\begin{equation}
\If[\eta](T_{\mathfrak{R}}) + \Jf[u](T_{\mathfrak{R}}) \ls \TE[u_0,\eta_0] + \int_0^{T_{\mathfrak{R}}}  (\E[u(t)] + \E^0[\eta(t)] )dt \ls  \TE[u_0,\eta_0] + \int_0^{T_{\mathfrak{R}}}  P(\TE) dt,
\end{equation}
and since $T_{\mathfrak{R}} \le T_1 \le 1$, we then have that  $\If[\eta](T_2) + \Jf[u](T_2) \ls P(\TE) + 1.$ This means that estimate \eqref{lam_loc_04} remains true at $T_{\mathfrak{R}}$, so it is actually estimate \eqref{lam_loc_03} that fails at time $T_{\mathfrak{R}}$.  Arguing similarly, we  deduce that
\begin{multline}
 \delta_1 =  \Lf[\eta(T_{\mathfrak{R}})] = \ns{\eta(T_{\mathfrak{R}})}_{4N-1/2} \le \Lf[\eta_0] + C\int_0^{T_{\mathfrak{R}}} (\ns{\eta(t)}_{4N} + \ns{\dt \eta(t)}_{4N-1}) dt \\
\le \frac{\delta_1}{2} + \frac{C}{\sigma} \int_0^{T_{\mathfrak{R}}} \E^\sigma[\eta(t)] dt \le \frac{\delta_1}{2} + \frac{C T_{\mathfrak{R}} }{\sigma}(1 + P({\EEE})),
\end{multline}
and hence that
\begin{equation}\label{lam_loc_1}
\frac{\sigma \delta_1}{2 C (1+ P({\EEE}))} \le T_{\mathfrak{R}}.
\end{equation}
To conclude the proof we then set $T_2 = T_{\mathfrak{R}} /2$, apply Theorem \ref{kappa_apriori}, and use \eqref{lam_loc_1} to produce \eqref{lam_loc_05}.
\end{proof}

\subsection{Sending $\kappa \to 0$}

Our aim now is to use Theorem \ref{lame_local} to send $\kappa \to 0$ in  \eqref{kappa_problem} in order to produce a solution to \eqref{lame_free_bndry}.

\begin{thm}\label{lame_exist}
Let $\delta_1 >0$ be the universal constant  from Theorem \ref{kappa_apriori}, and assume that the data and forcing satisfy $\mathfrak{P}(\delta_1)$ on the time interval $[0,T_\ast].$  There exists a $0 < T_3 = T_3(\TE)$ such that if  $0 < T \le \min\{T_\ast,T_3\}$, then  the following hold.
\begin{enumerate}
 \item A unique  solution $(u,\eta)$  to \eqref{lame_free_bndry}  exists on the interval $[0,T]$  and achieves the initial data.

 \item We have the estimate
\begin{multline} \label{lamex_01}
 \sup_{0\le t \le T} (\E[u(t)] + \E^\sigma[\eta(t)]) + \int_0^{T}(\D[u(t)] +\ns{\rho J \dt^{2N+1} u(t)}_{ \Hd} +\D^\sigma[\eta(t)] )dt
\\
\ls P(\TE[u_0,\eta_0]) +\sup_{0\le t \le T} \f_\infty(t)  + \int_0^{T}  \f_2(t) dt
\end{multline}
for a universal positive polynomial $P$ such that $P(0)=0$.

 \item We have the estimates
\begin{equation}\label{lamex_02}
 \Lf[\eta](T)  < \delta_1
\text{ and }
 \If[\eta](T) + \Jf[u](T) \ls P(\TE) + 1
\end{equation}
for a universal positive polynomial $P$ such that $P(0)=0$, where $\Lf[\eta]$ is defined by \eqref{l_def} and  $\If[\eta]$ and $\Jf[u]$ are defined by \eqref{ab_def}.

 \item We  have the estimate
\begin{multline}\label{lamex_04}
 \sup_{0\le t \le T} \sum_{j=1}^{2N} \ns{\dt^j \eta(t) }_{4N-2j+3/2} + \int_0^T \left( \ns{\dt \eta}_{4N-1/2} +  \sum_{j=2}^{2N+1} \ns{\dt^j \eta(t) }_{4N-2j+5/2} \right)dt\\
\ls P(\TE) +\sup_{0\le t \le T} \f_\infty(t)  + \int_0^{T}  \f_2(t) dt
\end{multline}
for a universal positive polynomial $P$ such that $P(0)=0$.

 \item We also have the estimate
\begin{multline}\label{lamex_05}
 \sup_{0\le t \le T} \ns{\eta(t)}_{4N+1/2} \le \exp\left(C T \int_0^T \ns{u(t)}_{H^{4N+1/2}(\Sigma)} dt \right) \\
 \times \left(\ns{\eta_0}_{4N+1/2} + T \int_0^T  \ns{u(t)}_{H^{4N+1/2}(\Sigma)} dt \right).
\end{multline}

\end{enumerate}
\end{thm}
\begin{proof}

We divide the proof into two steps.  In the first we will initially prove the theorem with the weaker condition that $T_3 = T_3(\TE,\sigma)$, i.e. that $T_3$ is allowed to depend on $\sigma$ as well as the data, but not on $\kappa$.  In the second we will use the first step and some auxiliary arguments to remove the $\sigma$ dependence.

Step 1 -- $T_3 = T_3(\TE,\sigma)$.

For each $0 < \kappa  < \min\{1,\sigma_+,\sigma_-\}$ Theorem \ref{lame_local} provides us with a pair $(u_\kappa,\eta_\kappa)$ solving \eqref{kappa_problem} on $(0,T_2)$ and satisfying the conclusions of the theorem.  We shall consider $\kappa$ to index a sequence of values chosen in $(0,\min\{1,\sigma_+,\sigma_-\})$ that decrease to $0$.  The $\kappa-$independent estimates of \eqref{lam_loc_02} show that the sequence $\{(u_\kappa,\eta_\kappa)\}_\kappa$ is bounded uniformly in the function space determined by the first line of \eqref{lam_loc_02} (i.e. the left hand side of the inequality with $\kappa=0$).  These uniform bounds allow us to argue as in  Theorem 6.3 of \cite{GT_lwp}, using weak and weak-$\ast$ compactness arguments together with interpolation and lower semi-continuity arguments, to extract a subsequence (still denoted by $\kappa$) such that
\begin{equation}\label{lamex_1}
\begin{split}
 \dt^j u_\kappa &\to \dt^j u \text{ in } C^0([0,T];H^{4N-2j-2}(\Omega)) \text{ for }j=0,1,2 \\
 \dt^j \eta_\kappa &\to \dt^j \eta \text{ in } C^0([0,T];H^{4N-2j-1}(\Sigma)) \text{ for }j=0,1,2,
 \end{split}
\end{equation}
where $(u,\eta)$ satisfies \eqref{lam_loc_02} and achieves the initial data.  Note here that the convergence \eqref{lamex_1} can be improved to a larger range of $j$ and to higher regularity for the various $\eta$ terms.  We state \eqref{lamex_1} as is because it is more than sufficient to pass to the limit in \eqref{kappa_problem} and deduce that $(u,\eta)$ are a strong solution to \eqref{lame_free_bndry}.
The estimates \eqref{lamex_02}  follow from \eqref{lam_loc_03} and \eqref{lam_loc_04} using similar lower semi-continuity arguments.

It remains to derive the improved estimates.  This entails improving \eqref{lam_loc_02} to \eqref{lamex_01} by obtaining an estimate for $\ns{u}_{L^2_T H^{4N+1}}$ and also proving  \eqref{lamex_04}.  To accomplish the first we will need \eqref{lamex_05}.   This estimate is proved in Step 1 of Theorem 5.4 of \cite{GT_lwp} by using estimates developed in \cite{danchin} for solutions to the kinematic transport equation \eqref{lame_xport}.

With \eqref{lamex_05} in hand, we use \eqref{lam_loc_02} and \eqref{lame_assump_3} in conjunction with the simple estimate
\begin{equation}
 \ns{u}_{H^{4N+1/2}(\Sigma)} \ls \ns{\nab_{\ast,0}^{4N-1} u}_{H^{3/2}(\Sigma)} \ls \ns{\nab_{\ast,0}^{4N-1} u}_{2}
\end{equation}
in order to bound
\begin{equation}
  \sup_{0\le t \le T} \ns{\eta(t)}_{4N+1/2} \ls \exp\left(C T  P(\TE) \right)
  \left(\TE + T  P(\TE) \right).
\end{equation}
Then if $T_3 \le T_2$ is taken sufficiently small, we may bound
\begin{equation}
 \sup_{0\le t \le T} \ns{\eta(t)}_{4N+1/2} \ls P(\TE).
\end{equation}
We then appeal to the elliptic estimate \eqref{ld_02} of Proposition \ref{lame_elliptic} to bound
\begin{multline}
 \int_0^T \ns{u(t)}_{4N+1} dt \ls \int_0^T \left(  \ns{\rho \dt u(t)}_{4N-1} + \ns{F^1(t)}_{4N-1} + \ns{u(t)}_{H^{4N+1/2}(\Sigma)} \right) dt \\
 +  \sup_{0\le t \le T} \ns{\eta(t)}_{4N+1/2} \int_0^T \left( \ns{\rho \dt u(t)}_{2} + \ns{F^1(t)}_{2} + \ns{u(t)}_{H^{7/2}(\Sigma)} \right) dt.
\end{multline}
Hence \eqref{rho_assump_1}, \eqref{rho_assump_2}, and \eqref{lam_loc_02} allow us to bound
\begin{multline}\label{lamex_3}
  \int_0^T \ns{u(t)}_{4N+1} dt \ls \int_0^T \left(\check{D}[u(t)] + \f_2(t) \right) dt + T(1+P(\TE) ) \sup_{0 \le t \le T} \left( \E[u(t)] + \f_\infty(t) \right)
 \\ \ls P(\TE) +\sup_{0\le t \le T} \f_\infty(t)  + \int_0^{T}  \f_2(t) dt
\end{multline}
once we further restrict $T_3$ so that $T_3 (1+P(\TE) )\le 1$.  Summing \eqref{lamex_3} and \eqref{lam_loc_02} then yields \eqref{lamex_01}.

Finally,  we prove the improved estimates \eqref{lamex_04}.  For this we will use the fact that $(u,\eta)$ satisfy the kinematic equation
\begin{equation}\label{lame_xport}
 \dt \eta = u_3 +  \nabla_\ast \eta\cdot u\text{ on }\Sigma.
\end{equation}
We may use this equality as in Theorem 5.4 of \cite{GT_lwp}, using the usual estimates of products in Sobolev spaces, to deduce that
\begin{multline}\label{lamex_2}
  \sup_{0\le t \le T} \sum_{j=1}^{2N} \ns{\dt^j \eta(t) }_{4N-2j+3/2} +   \int_0^T \left( \ns{\dt \eta}_{4N-1/2} +  \sum_{j=2}^{2N+1} \ns{\dt^j \eta(t) }_{4N-2j+5/2} \right)dt \\
\le P\left(  \sup_{0\le t \le T} (\E[u(t)] + \E^\sigma[\eta(t)]) + \int_0^{T}(\D[u(t)]  +\D^\sigma[\eta(t)] )dt   \right)
\end{multline}
for some universal positive polynomial with $P(0)=0$.  Then from \eqref{lamex_2}, \eqref{lamex_01}, and \eqref{lame_assump_3} we deduce that \eqref{lamex_04} holds.

Step 2 -- Improvement to $T_3 = T_3(\TE)$

With the Theorem in hand for a $T_3 = T_3(\TE,\sigma)$, we can employ a continuation argument in order to remove the dependence on $\sigma$.  The argument is similar to the one used in Theorem \ref{lame_local}, so in the interest of brevity we will only point out the key point.  This lies in the fact that our result from Step 1 requires $\mathfrak{P}(\delta_1)$ to hold on $[0,T_\ast]$, which in turn demands that $\ns{\eta_0}_{4N-1/2} < \delta_1/2$.  The estimate \eqref{lamex_04} provides us with an estimate of $\sup_{0 \le t \le T} \ns{ \dt \eta(t)}_{4N-1/2}$, which when coupled to the fundamental theorem of calculus, allows us to estimate $\mathfrak{L}[\eta](T)$.  Using this, we can prove that the theorem remains true on a time interval $[0, \min\{T_\ast, T_3\}]$, where either $T_3 = T_\ast$ or else $T_3$ is when $\ns{\eta(T_3)}_{4N-1/2}  \ge \frac{\delta_1}{2}.$  We may then argue as in \eqref{lam_loc_1}, using the estimate for $\ns{\dt \eta}_{4N-1/2}$ in place of Lemma \ref{time_interp}, to show that $T_3$ is bounded below by a positive function of $\TE$ that is independent of $\sigma$.  Hence $T_3 = T_3(\TE)$.
\end{proof}

\begin{remark}
 The improved $\eta$ estimates of \eqref{lamex_04} are the key to eliminating the dependence of the existence interval on $\sigma$.  These in turn depend on the structure of the equation $\dt \eta = u \cdot \n$ in place of the $\kappa$ approximation.  We thus see the importance of passing to the limit $\kappa \to 0$ before attempting to remove the dependence on $\sigma$.
\end{remark}

\section{The transport problem}\label{sec_xport}


Our goal in this section is to study the transport problem
\begin{equation}\label{xport_fundamental}
\begin{cases}
  \dt q  +(\a^T u - K \dt \theta e_3) \cdot \nab q + \diva(u) q = f & \text{in } \Omega \times (0,T) \\
 q(\cdot, 0) = q_0 &\text{in }\Omega,
\end{cases}
\end{equation}
where $u$, $\eta$ (and hence $\a$, etc), and $f$ are given.  To simplify the structure of \eqref{xport_fundamental} we will initially study the more general problem
\begin{equation}\label{xport_proto}
 \dt q + v \cdot \nab q +  c q = f,
\end{equation}
where $v$, $c$, and $f$ are given.

The key to the analysis of a transport problem \eqref{xport_proto} on a domain with boundaries is the behavior of $v$ at the boundary.  For the $v$ arising in \eqref{xport_fundamental} with $(u,\eta)$ satisfying  \eqref{lame_free_bndry} we have the following crucial identities: on $\Sigma_\pm$ we have
\begin{equation}\label{xport_v1}
v \cdot e_3 =  (\a^T u - K \dt \theta e_3)\cdot e_3 = K(u \cdot J\a e_3 - \dt \eta) = K (u_\pm \cdot\n_\pm - \dt \eta_\pm) =0,
\end{equation}
while on $\Sigma_b$ we have
\begin{equation}\label{xport_v2}
 v \cdot e_3 =  (\a^T u - K \dt \theta e_3)\cdot e_3 = K(u_- \cdot J \a e_3 - \tilde{b}_1 \dt \bar{\eta}_+ - \tilde{b}_2 \dt \bar{\eta}_1) =0
\end{equation}
since $\tilde{b}_1 = \tilde{b}_2 =0$  and $u_- =0$ on $\Sigma_b$. Note that the condition $u_- =0$ on $\Sigma_b$ could  be weakened to $u_{3,-}=0$ on $\Sigma_b$ since $J \a e_3=e_3$ on $\Sigma_b$.  The upshot of these identities is that we may assume that the vector field $v$ satisfies
\begin{equation}\label{xport_v}
 v \cdot e_3 =0 \text{ on } \p \Omega.
\end{equation}

Notice that we do not couple the transport equation to boundary conditions, and in fact the equations in $\Omega_+$ and $\Omega_-$ are decoupled from one another.  This allows us to solve the equations in each domain separately.  To this end we will let $\Upsilon = \Omega_+$ or $\Upsilon = \Omega_-$ and discuss the transport equation in $\Upsilon$.   Let us now define various functionals that will appear in our analysis of the transport problem.

For a function $g$  we define
\begin{equation}\label{rqe_def}
 \Qf_e[g] = \sum_{j=1}^{2N} \ns{\dt^j g}_{4N-2j+1} \text{ and } \Rf_e[g] = \sum_{j=0}^{2N-1} \ns{\dt^j g}_{4N-2j-1}.
\end{equation}
and
\begin{equation}\label{rqd_def}
 \Qf_d[g] = \ns{\dt g}_{4N-1} + \sum_{j=2}^{2N+1} \ns{\dt^j g}_{4N-2j+2} \text{ and } \Rf_d[g] = \ns{g}_{4N-1} +  \sum_{j=1}^{2N} \ns{\dt^j g}_{4N-2j}.
\end{equation}

We will assume that $v$, $c$, and $f$ satisfy
\begin{equation}\label{xport_reg}
 \sup_{0 \le t \le T} \left( \Rf_e[v(t)] + \Rf_e[c(t)] + \Rf_e[f(t)]  \right) + \int_0^T \left( \Rf_d[v(t)] + \Rf_d[c(t)] + \Rf_d[f(t)] \right) dt <\infty.
\end{equation}

\subsection{Solution by characteristics}

Consider the transport problem
\begin{equation}\label{xport_eqn}
 \begin{cases}
  \dt q + v \cdot \nab q + c q = f & \text{in } \Upsilon \times (0,T) \\
  q(\cdot,0) = q_0 &\text{in }\Upsilon.
 \end{cases}
\end{equation}
Here we assume that $v$, $c$, and $f$ satisfy \eqref{xport_reg}.  In particular the usual Sobolev embeddings require that $v$, $c$, and $f$ are all $C^1(\Upsilon \times [0,T])$.

To produce a solution to \eqref{xport_eqn} we use the method of characteristics.  We let $\zeta_t(x)$ denote the solution to the ODE
\begin{equation}\label{xport_char_1}
\begin{cases}
 \dt \zeta_t(x) = v(\zeta_t(x),t)  \\
 \zeta_0(x) = x
\end{cases}
\end{equation}
for $x \in \Upsilon$ and $t \in [0,T]$.  The identity \eqref{xport_v} is essential here, since it guarantees that for each $t \in [0,T]$, $\zeta_t : \Upsilon \to \Upsilon$ is a diffeomorphism.   Let us define the map $\omega : [0,T]^2 \times \Upsilon \to \Upsilon$ via
\begin{equation}\label{xport_char_2}
 \omega(s,t,x) = \zeta_s(\zeta^{-1}_t(x)).
\end{equation}

Then the solution to \eqref{xport_eqn} is
\begin{multline}\label{xport_soln}
 q(x,t) = q_0(\omega(0,t,x)) \exp\left[ - \int_0^t c(\omega(s,t,x),s)ds \right] \\
+ \int_0^t f(\omega(s,t,x),s) \exp\left[ - \int_s^t c(\omega(r,t,x),r)dr \right]ds.
\end{multline}
While this explicit form of the solution is nice, it is not convenient for making higher regularity estimates.  In particular it is not immediately obvious from \eqref{xport_soln} that $q$ belongs to the space defined by \eqref{energy_def_q} and \eqref{dissipation_def_q}, and it is not clear that we can justify applying $\p^\alpha$ to \eqref{xport_eqn} and performing a priori estimates.  The usual solution to this difficulty is the Friedrichs mollification method: we first solve a mollified version of \eqref{xport_eqn} so that the solution is smooth enough to justify the a priori estimates, then we derive the a priori estimates in a manner independent of the mollification parameter, and then finally we pass to the limit.  This works well when $\p \Upsilon=\varnothing$, but the mollification procedure runs into technical obstructions when $\p \Upsilon \neq \varnothing$, which is the case here.

Two options then present themselves.  The first is  to modify the mollification procedure in a manner that makes sense in our $\Upsilon$ but does not destroy the structure of the a priori estimates.  The second is to transfer the problem \eqref{xport_eqn} to a new problem on a set without boundary in such a way that the estimates from Friedrichs' method carry over to \eqref{xport_eqn}.  We have chosen to go for the second option since Friedrichs' method is so well-known.  Our goal then is to justify the transfer of the problem and the estimates.

\subsection{Transfer}

Consider the transport problem \eqref{xport_eqn} where  we only assume for now that $q_0 \in L^2(\Upsilon)$ and $f \in L^2([0,T];L^2(\Upsilon))$.  We say that $q \in L^2([0,T];L^2(\Upsilon))$ is a weak solution to \eqref{xport_eqn} if
\begin{equation}\label{xport_weak}
 \int_0^T \int_\Upsilon -q \left( \dt \varphi + \diverge(v \varphi ) - c\varphi\right) = \int_0^T \int_\Upsilon \varphi f + \int_\Upsilon q_0 \varphi(\cdot,0)
\end{equation}
for all $\varphi \in C_c^1(\Upsilon \times [0,T))$.  The identity \eqref{xport_weak} is clearly satisfied by any regular solution to \eqref{xport_eqn}, and in particular by the solution given by \eqref{xport_soln}.

We have the following simple lemma on the uniqueness of weak solutions, which we state without proof.

\begin{lem}\label{xport_unique_char}
 The following are equivalent.
\begin{enumerate}
 \item For every $f \in L^2([0,T];L^2(\Upsilon))$ and $q_0 \in L^2(\Upsilon)$, there exists at most one function $q \in L^2([0,T];L^2(\Upsilon))$ that is a weak solution to \eqref{xport_eqn}.
 \item If $q \in L^2([0,T];L^2(\Upsilon))$ satisfies
\begin{equation}\label{xuc_01}
  \int_0^T \int_\Upsilon -q \left( \dt \varphi + \diverge(v \varphi ) - c\varphi\right) =0
\end{equation}
for every $\varphi \in C_c^1(\Upsilon \times [0,T))$, then $q=0$.
\end{enumerate}
\end{lem}

Our next goal is to verify that the second item of Lemma \ref{xport_unique_char} holds, which means that weak solutions to \eqref{xport_eqn} are unique.  To this end we first study the adjoint problem determined by what appears in parentheses in \eqref{xuc_01}.

Let $\psi \in C_c^\infty(\Upsilon \times (0,T))$.  We want to find $\varphi \in C_c^1(\Upsilon \times [0,T))$ satisfying the adjoint problem
\begin{equation}\label{xport_adjoint}
\begin{cases}
  \dt \varphi + \diverge(v \varphi) - c \varphi = \psi &\text{in }\Upsilon \times (0,T) \\
  \varphi(\cdot,T) = 0 &\text{in }\Upsilon.
\end{cases}
\end{equation}
Note that this is a terminal value problem; we seek to solve this so that $\varphi$ can be taken to be compactly supported in $\Upsilon \times [0,T)$.

To solve the adjoint problem \eqref{xport_adjoint} we  use $\zeta_t$ given by \eqref{xport_char_1} to reduce to an ODE along the characteristics.  Recall the function $\omega :[0,T]^2 \times \Upsilon \to \Upsilon$ given by  \eqref{xport_char_2}; we may use it to derive  the solution
\begin{equation}
\varphi(x,t) = - \int_t^T \psi(\omega(s,t,x),s) \exp\left[ \int_t^r \left( \diverge{v}(\omega(r,t,x),r)  - c(\omega(r,t,x),r) \right)dr   \right] ds.
\end{equation}
From this formula and the inclusion $\psi \in C_c^\infty(\Upsilon \times (0,T))$ it is easy to see that $\varphi \in C_c^1(\Upsilon \times [0,T))$, as desired.  We have thus proved the following lemma.

\begin{lem}\label{xport_adjoint_soln}
 Let $\psi \in C_c^\infty(\Upsilon \times (0,T))$.  Then there exists $\varphi \in C_c^1(\Upsilon \times [0,T))$ solving \eqref{xport_adjoint}.
\end{lem}

With Lemma \ref{xport_adjoint_soln} in hand we can prove that weak solutions to \eqref{xport_eqn} are unique.

\begin{prop}\label{xport_weak_unique}
 Let $f \in L^2([0,T];L^2(\Upsilon))$ and $q_0 \in L^2(\Upsilon)$.  Then there exists at most one function $q \in L^2([0,T];L^2(\Upsilon))$ that is a weak solution to \eqref{xport_eqn} in the sense of \eqref{xport_weak}.
\end{prop}
\begin{proof}
 To prove uniqueness we will show that the second item of Lemma \ref{xport_unique_char} holds.  Suppose that $q \in L^2([0,T];L^2(\Upsilon))$ satisfies \eqref{xuc_01} for every $\varphi \in C_c^1(\Upsilon \times [0,T))$.  For any  $\psi \in C_c^\infty(\Upsilon \times (0,T))$ we may use Lemma \ref{xport_adjoint_soln} to find $\varphi \in C_c^1(\Upsilon \times [0,T))$ solving \eqref{xport_adjoint}.  Using this $\varphi$ in \eqref{xuc_01} yields the equality
\begin{equation}
 \int_0^T \int_\Upsilon q \psi =0.
\end{equation}
Since $\psi$ was arbitrary, we deduce that $q =0$.  Hence the second item of Lemma \ref{xport_unique_char} holds.
\end{proof}

Next we define the extended problem, which is easier to handle with Friedrichs' method.  Let $\Gamma = \mathrm{T}^2\times \mathbb{R}$ denote the extended domain in which we will pose the extended problem.  Let $E$ denote a Sobolev extension operator such that $E : H^m(\Upsilon) \to H^m(\Gamma)$ for every $m=0,\dotsc,M$, where $M >0$  is large enough to surpass every regularity index in the energy and dissipation defined in \eqref{energy_def_q} and \eqref{dissipation_def_q}.  We then define
\begin{equation}
 \bar{v} = E v, \; \bar{c} = Ec, \;\bar{q}_0 = Eq_0, \text{ and } \bar{f} = E f.
\end{equation}

We say that $\bar{q} \in L^2([0,T];L^2(\Gamma))$ is a weak solution to
\begin{equation}\label{xport_extended}
 \begin{cases}
  \dt \bar{q} + \bar{v} \cdot \nab \bar{q} + \bar{c} \bar{q} = \bar{f} & \text{in } \Gamma \times (0,T) \\
  \bar{q}(\cdot,0) = \bar{q}_0 &\text{in }\Gamma
 \end{cases}
\end{equation}
if
\begin{equation}\label{xport_weak_extended}
 \int_0^T \int_\Gamma -\bar{q} \left( \dt \varphi + \diverge(\bar{v} \varphi ) - \bar{c}\varphi\right) = \int_0^T \int_\Gamma \varphi \bar{f} + \int_\Gamma \bar{q}_0 \varphi(\cdot,0)
\end{equation}
for all $\varphi \in C_c^1(\Gamma \times [0,T))$.

\begin{lem}\label{xport_transfer}
 Suppose that $\bar{q} \in L^2([0,T];L^2(\Gamma))$ is a weak solution to \eqref{xport_extended} in the sense of \eqref{xport_weak_extended}.  Let $q$ denote the restriction of $\bar{q}$ to $\Upsilon \times (0,T)$.  Then $q$ is the unique  weak solution to \eqref{xport_eqn} in the sense of \eqref{xport_weak}.
\end{lem}
\begin{proof}
 Since \eqref{xport_weak_extended} holds for all $\varphi \in C_c^1(\Gamma \times [0,T))$ it must also hold for all $\varphi \in C_c^1(\Upsilon \times [0,T))$.  For such $\varphi$ the equality \eqref{xport_weak_extended} is identical to \eqref{xport_weak} because $\bar{v},$ $\bar{c}$,$\bar{q}_0$, and $\bar{f}$ are  extensions of $v$, $c$, $\bar{q}_0$, and $f$ to $\Gamma$.  Hence $q$ is a weak solution to to \eqref{xport_eqn} in the sense of \eqref{xport_weak}.  Uniqueness follows from Proposition \ref{xport_weak_unique}.
\end{proof}

The upshot of Lemma \ref{xport_transfer} is that if we can produce a solution to the extended problem \eqref{xport_extended} that obeys the estimates we seek for the original problem \eqref{xport_eqn}, then we know that those estimates are also valid for the solution \eqref{xport_soln} given by characteristics on $\Upsilon$.  This leads us to study the extended problem \eqref{xport_extended}.

\subsection{The extended problem \eqref{xport_extended} }

\subsubsection{The mollified problem }

We begin by defining  space and time mollification operators.  Since the horizontal directions in $\Gamma$ are periodic, it is convenient to decompose our spatial mollifiers into horizontal and  vertical parts.    Let $\vartheta \in C_c^\infty(\Rn{})$ be a standard mollifier.  For $n \in \mathbb{N}$ and $i=1,2$, let $F_n^i \in C^\infty(2\pi L_i \mathbb{T})$ denote the Fej\'{e}r kernel.  Then we define the spatial mollifier via
\begin{equation}
 K_\ep^{sp} g(x) =   \int_\Gamma    g(x-y) F_{\lfloor 1/\ep\rfloor}^1(y_1) F_{\lfloor 1/\ep\rfloor}^2(y_2) \frac{1}{\ep} \vartheta\left(\frac{y_3}{\ep}\right)  dy,
\end{equation}
where $\lfloor z \rfloor$ denotes the integer part of $z$.  To define the temporal mollification operator we must first define a temporal extension.  For a function $g$ defined on $\Gamma \times (0,T)$ we first extend to $\tilde{g}$  defined on $\Gamma \times \Rn{}$ via
\begin{equation}
 \tilde{g}(x,t) =
\begin{cases}
 g(x,t) & \text{if }t \in [0,T] \\
 0 & \text{if }t \notin [0,T].
\end{cases}
\end{equation}
Then the temporal mollification is
\begin{equation}
 K_\ep^{te} g(x,t) = \int_{\Rn{}} \tilde{g}(x,t-s) \frac{1}{\ep}\vartheta\left(\frac{s}{\ep}\right) ds.
\end{equation}
The operators $K_\ep^{sp}$ and $K_\ep^{te}$ satisfy all the usual properties of a mollification operators.

The mollified problem studied in Friedrichs' method is
\begin{equation}\label{xport_mollified}
 \begin{cases}
 \dt \bar{q}_\ep + K_\ep^{sp}[ (K_\ep^{te}\bar{v}) \cdot \nab (K_\ep^{sp} \bar{q}_\ep)  ] + (K_\ep^{te} \bar{c}) \bar{q}_\ep = K_\ep^{te} \bar{f} &\text{in } \Gamma \times (0,\infty) \\
 \bar{q}_\ep(\cdot,0) =\bar{q}_0 &\text{in }  \Gamma.
 \end{cases}
\end{equation}
Notice that $K_\ep^{te}\bar{v}$, $K_\ep^{te}\bar{c}$ and $K_\ep^{te}\bar{f}$ belong to $C^\infty(\Rn{};H^{4N}(\Gamma))$.  Then the theory of linear ODEs in Banach spaces  provides us with a unique solution $\bar{q}_\ep \in C^\infty([0,\infty);H^{4N}(\Gamma))$ to \eqref{xport_mollified}.

Our next goal is to produce $\ep-$independent estimates to the solutions to \eqref{xport_mollified}.  To this end we note that for $\alpha \in \mathbb{N}^{3}$ we may apply $\p^\alpha$ to \eqref{xport_mollified} to find that $\p^\alpha \bar{q}_\ep$ solves
\begin{equation}\label{xport_mollified_diff}
 \begin{cases}
 \dt \p^\alpha \bar{q}_\ep + K_\ep^{sp}[ (K_\ep^{te}\bar{v}) \cdot \nab (K_\ep^{sp} \p^\alpha \bar{q}_\ep)  ] + (K_\ep^{te} \bar{c}) \p^\alpha\bar{q}_\ep = f^\alpha_\ep &\text{in } \Gamma \times (0,\infty) \\
 \p^\alpha\bar{q}_\ep(\cdot,0) =\p^\alpha \bar{q}_0 &\text{in }  \Gamma,
 \end{cases}
\end{equation}
where
\begin{equation}
 f^\alpha_\ep = \p^\alpha K_\ep^{te} \bar{f} - \sum_{0 < \beta \le \alpha }C_{\alpha, \beta} \left(  K_\ep^{sp}[ (K_\ep^{te} \p^\beta \bar{v}) \cdot \nab (K_\ep^{sp} \p^{\alpha-\beta} \bar{q}_\ep)  ]  + (K_\ep^{te} \p^\beta \bar{c}) \p^{\alpha-\beta}\bar{q}_\ep  \right).
\end{equation}

\subsubsection{Estimates }

By multiplying the first equation in  \eqref{xport_mollified_diff} by $\p^\alpha \bar{q}_\ep$ and integrating by parts over $\Gamma$ we may derive the basic energy identity
\begin{equation}
 \frac{d}{dt} \int_{\Gamma} \hal \abs{\p^\alpha \bar{q}_\ep}^2 + \int_\Gamma (K_\ep^{te}\bar{c}) \abs{\p^\alpha \bar{q}_\ep}^2- (K_\ep^{te} \diverge \bar{v} )\abs{K_\ep^{sp} \p^\alpha \bar{q}_\ep}^2 = \int_\Gamma f_\ep^\alpha \p^\alpha \bar{q}_\ep.
\end{equation}
We may use standard Sobolev embeddings and properties of mollifiers to estimate
\begin{equation}
 \int_\Gamma f_\ep^\alpha \p^\alpha \bar{q}_\ep \ls \norm{\bar{f}}_{4N} \norm{\bar{q}_\ep}_{4N} + \left( \norm{\bar{v}}_{4N} + \norm{\bar{c}}_{4N} \right) \norm{\bar{q}_\ep}_{4N}^2.
\end{equation}
Then
\begin{equation}
 \frac{d}{dt} \norm{\bar{q}_\ep}_{4N}^2 \ls \norm{\bar{f}}_{4N} \norm{\bar{q}_\ep}_{4N} + \left( \norm{\bar{v}}_{4N} + \norm{\bar{c}}_{4N} \right) \norm{\bar{q}_\ep}_{4N}^2,
\end{equation}
which leads us to the fundamental estimate
\begin{equation}\label{xport_main_1}
 \sup_{0\le t \le T} \norm{\bar{q}_\ep(t)}_{4N} \ls  \exp\left( C\int_0^T \left( \norm{\bar{v}(t)}_{4N} + \norm{\bar{c}(t)}_{4N} \right)dt \right)  \left(\norm{\bar{q}_0}_{4N} + \int_0^T \norm{\bar{f}(t)}_{4N} dt \right).
\end{equation}
Here the constants on the right-hand side do not depend on $\ep$.

In order to record other estimates we  recall the functionals $\Qf_e$, $\Rf_e$ given by \eqref{rqe_def} and $\Qf_d$, $\Rf_d$ given by \eqref{rqd_def}.  Here the norms are understood to be computed over $\Gamma$.  With the estimate \eqref{xport_main_1} in hand we may use the equation \eqref{xport_mollified} to directly estimate $\norm{\dt \bar{q}_\ep}_{4N-1}^2$:
\begin{equation}
 \norm{\dt \bar{q}_\ep}_{4N-1}^2 \ls \ns{\bar{f}}_{4N-1} + \left(  \ns{\bar{v}}_{4N-1} + \ns{\bar{c}}_{4N-1} \right)\ns{\bar{q}_\ep}_{4N}.
\end{equation}
We may iteratively apply $\dt$ to \eqref{xport_mollified} to estimate higher-order temporal derivatives; this leads us to the estimate
\begin{equation}
 \Qf_e[\bar{q}_\ep] \ls (1+ P(\Rf_e[\bar{v}] + \Rf_e[\bar{c}])) \Rf_e[\bar{f}] + P(\Rf_e[\bar{v}] + \Rf_e[\bar{c}]) \ns{\bar{q}_\ep}_{4N}
\end{equation}
for some universal positive polynomial $P$ with $P(0)=0$.   Hence
\begin{multline}\label{xport_main_2}
\sup_{0\le t\le T}  \E[\bar{q}_\ep(t)] \ls  \sup_{0\le t\le T} \left(1+ P(\Rf_e[\bar{v}(t)] + \Rf_e[\bar{c}(t)])) \Rf_e[\bar{f}(t)] \right) \\
+ \left(\sup_{0\le t\le T}  \left(1+P(\Rf_e[\bar{v}(t)] + \Rf_e[\bar{c}(t)]) \right) \right) \\
\times \exp\left( C T \int_0^T \left( \ns{\bar{v}(t)}_{4N} + \ns{\bar{c}(t)}_{4N} \right)dt \right)  \left(\ns{\bar{q}_0}_{4N} + T\int_0^T \ns{\bar{f}(t)}_{4N} dt \right).
\end{multline}

Similarly, we may derive the bound
\begin{equation}
 \Qf_d[\bar{q}_\ep] \ls (1+ P(\Rf_e[\bar{v}] + \Rf_e[\bar{c}])) \Rf_d[\bar{f}] + P(\Rf_e[\bar{v}] + \Rf_e[\bar{c}]) \left( \ns{\bar{q}_\ep}_{4N} + (\Rf_d[\bar{v}] + \Rf_d[\bar{c}])\E[\bar{q}_\ep] \right).
\end{equation}
This leads us to the estimate
\begin{multline}
 \int_0^T \D[\bar{q}_\ep(t)]dt \le \sup_{0\le t\le T}  \left(1+ P(\Rf_e[\bar{v}(t)] + \Rf_e[\bar{c}(t)]) \right) \int_0^T \Rf_d[\bar{f}(t)]dt \\
+ \sup_{0\le t\le T} \left( 1+ P(\Rf_e[\bar{v}(t)] + \Rf_e[\bar{c}(t)]) \right)  \int_0^T \ns{\bar{q}_\ep(t)}_{4N} dt  \\
+  \left(\sup_{0\le t\le T}   P(\Rf_e[\bar{v}(t)] + \Rf_e[\bar{c}(t)]) \right) \left( \sup_{0\le t \le T} \E[\bar{q}_\ep(t)] \right) \int_0^T \left( \Rf_d[\bar{v}(t)] + \Rf_d[\bar{c}(t)]\right) dt.
\end{multline}
Combining with \eqref{xport_main_1} and \eqref{xport_main_2} then leads to the bound
\begin{multline}\label{xport_main_3}
  \int_0^T \D[\bar{q}_\ep(t)]dt  \le  \sup_{0\le t\le T}  \left(1+ P(\Rf_e[\bar{v}(t)] + \Rf_e[\bar{c}(t)]) \right) \int_0^T \Rf_d[\bar{f}(t)]dt \\
+  \sup_{0\le t\le T}  \left( P(\Rf_e[\bar{v}(t)] + \Rf_e[\bar{c}(t)]) \right)  \left( \int_0^T \left( \Rf_d[\bar{v}(t)] + \Rf_d[\bar{c}(t)]\right) dt \right)  \sup_{0\le t\le T}  \left( \Rf_e[\bar{f}(t)]dt \right) \\
+ \left[ T  \sup_{0\le t\le T}  \left(1+ P(\Rf_e[\bar{v}(t)] + \Rf_e[\bar{c}(t)]) \right) \right. \\
\left. +  \sup_{0\le t\le T}  \left( P(\Rf_e[\bar{v}(t)] + \Rf_e[\bar{c}(t)]) \right) \int_0^T \left( \Rf_d[\bar{v}(t)] + \Rf_d[\bar{c}(t)]\right) dt \right]
\\
\times    \exp\left( CT \int_0^T \left( \ns{\bar{v}(t)}_{4N} + \ns{\bar{c}(t)}_{4N} \right)dt \right)  \left(\ns{\bar{q}_0}_{4N} + T\int_0^T \ns{\bar{f}(t)}_{4N} dt \right).
\end{multline}

\subsubsection{Passing to the limit }

The estimates \eqref{xport_main_2} and \eqref{xport_main_3} provide us with estimates of
\begin{equation}
 \sup_{0\le t \le T} \E[\bar{q}_\ep(t)] + \int_0^T \D[\bar{q}_\ep(t)]dt
\end{equation}
that are independent of $\ep$.  We may then extract a subsequence of $\ep$ values such that $\bar{q}_\ep$ converges to some $\bar{q}$, which by virtue of lower semicontinuity, must also obey the estimates \eqref{xport_main_2} and \eqref{xport_main_3}.  Multiplying \eqref{xport_mollified} by $\varphi \in C_c^1(\Gamma \times [0,T))$ and integrating by parts leads to the identity
\begin{equation}
  \int_0^T \int_\Gamma -\bar{q}_\ep \left( \dt \varphi + K_\ep^{sp} \diverge((K_\ep^{te}\bar{v}) \varphi ) - (K_\ep^{te}\bar{c})\varphi\right) = \int_0^T \int_\Gamma \varphi K_\ep^{te}\bar{f} + \int_\Gamma \bar{q}_0 \varphi(\cdot,0).
\end{equation}
We may then send $\ep \to 0$  to deduce that $\bar{q}$ is weak solution to \eqref{xport_extended} in the sense of \eqref{xport_weak_extended}.  This proves the following proposition.

\begin{prop}\label{xport_ext_exist}
 There exists a weak solution $\bar{q}$ to \eqref{xport_extended} obeying the estimates \eqref{xport_main_2} and \eqref{xport_main_3}.
\end{prop}

\subsection{Estimates for the solution to \eqref{xport_fundamental} }

We are now ready to return to the problem \eqref{xport_fundamental}.

We begin with estimates of the $v$ and $c$ terms that arise in \eqref{xport_fundamental}.  They are simple variants of previous nonlinear estimates (for example Lemmas \ref{lame_forcing_est} and \ref{k_F_ests}, Theorem \ref{kappa_contract}, and Proposition \ref{kappa_improved}), so we omit the proof.

\begin{prop}\label{xport_nonlins}
Let $v$ and $c$ be given on $\Omega$ as in \eqref{xport_fundamental} and let $\bar{v}$, $\bar{c}$ denote their bounded extensions from $\Upsilon = \Omega_\pm$ to $\Gamma$.  Then we have the following estimates, where $P$ is a universal positive polynomial with $P(0)=0$:
\begin{equation}
 \Rf_e[\bar{v}] + \Rf_e[\bar{c}] \le P( \E^0[\eta]  ) \left( \E[u] + \hat{\E}^0[\eta]    \right),
\end{equation}
\begin{equation}\label{xp_n_01}
 \Rf_d[\bar{v}] + \Rf_d[\bar{c}] \le P( \hat{\E}^0[\eta]  ) \left( \D[u] + \hat{\E}^0[\eta] + \hat{\D}^0[\eta]    \right),
\end{equation}
and
\begin{equation}\label{xp_n_02}
  \ns{\bar{v}}_{4N} + \ns{\bar{c}}_{4N} \le P( \E[u] + \hat{\E}^0[\eta]  ) \left(  \D[u] + \hat{\D}^0[\eta]  + \ns{\eta}_{4N+1/2} \right).
\end{equation}
\end{prop}

\begin{remark}
 The term $\hat{\E}^0[\eta]$ is added onto the right side of \eqref{xp_n_01} only because $\hat{\D}^0[\eta]$ provides no control of $\eta$ itself.  The $\eta$ term in $\hat{\E}^0[\eta]$ is enough to make up for this deficit.  The term $\ns{\eta}_{4N+1/2}$ is added on the right side of \eqref{xp_n_02} for a similar reason, but in this case the regularity demands require more than $\hat{\E}^0[\eta]$.
\end{remark}

Now we prove that the solution produced by the method of characteristics in \eqref{xport_soln} obeys various useful estimates.

\begin{thm}\label{xport_well-posed}
 Suppose that $u$ and  $\eta$ are given and satisfy
\begin{equation}
 \sup_{0 \le t \le T} \left( \E[u(t)] + \hat{\E}^0[\eta(t)]  \right) + \int_0^T \left(\D[u(t)] + \hat{\D}^0[\eta(t)]\right)dt < \infty
\end{equation}
and
\begin{equation}\label{xwp_03}
\dt \eta = u \cdot \n \text{ on } \Sigma, \jump{u}=0 \text{ on } \Sigma_-,  \text{ and } u_- =0 \text{ on } \Sigma_b.
\end{equation}
Let $q$ be given by \eqref{xport_soln}.  Then $q$ is the unique solution to \eqref{xport_fundamental}.  Moreover, the solution obeys the following estimates for some universal positive polynomial $P$ with $P(0)=0$:
\begin{multline}\label{xwp_01}
\sup_{0\le t\le T}  \E[q(t)] \ls  \sup_{0\le t\le T} \left(1+ P(\E[u(t)] + \hat{\E}^0[\eta(t)]) \right) \sup_{0\le t\le T}\left( \Rf_e[f(t)] \right) \\
+ \left(\sup_{0\le t\le T}  \left(1 + P(\E[u(t)] + \hat{\E}^0[\eta(t)]) \right) \right) \Xi(T)
\left(\ns{q_0}_{4N} + T\int_0^T \ns{f(t)}_{4N} dt \right),
\end{multline}
and
\begin{multline}\label{xwp_02}
 \int_0^T \D[q(t)]dt  \le  \sup_{0\le t\le T} \left(1+ P(\E[u(t)] + \hat{\E}^0[\eta(t)]) \right) \int_0^T \Rf_d[f(t)]dt \\
+  \sup_{0\le t\le T} \left( P(\E[u(t)] + \hat{\E}^0[\eta(t)]) \right) \left( \int_0^T \left( \D[u(t)] + \hat{D}^0[\eta(t)] + \hat{\E}^0[\eta(t)] \right) dt \right)  \sup_{0\le t\le T}  \left( \Rf_e[f(t)]dt \right) \\
+ \Xi(T) \left(\ns{q_0}_{4N} + T\int_0^T \ns{f(t)}_{4N} dt \right) \left[ T \sup_{0\le t\le T} \left(1+ P(\E[u(t)] + \hat{\E}^0[\eta(t)]) \right) \right. \\
\left. +  \sup_{0\le t\le T} \left( P(\E[u(t)] + \hat{\E}^0[\eta(t)]) \right) \int_0^T \left( \D[u(t)] + \hat{D}^0[\eta(t)] +\hat{\E}^0[\eta(t)] \right) dt \right],
\end{multline}
where we have written
\begin{equation}
 \Xi(T) :=  \exp\left( C T \int_0^T P( \E[u(t)] + \hat{\E}^0[\eta(t)]  ) \left(  \D[u(t)] + \hat{\D}^0[\eta(t)]  + \ns{\eta(t)}_{4N+1/2} \right) dt \right).
\end{equation}

\end{thm}
\begin{proof}
The equations in \eqref{xwp_03} guarantee that the computations in \eqref{xport_v1} and \eqref{xport_v2} are valid, and so $v$ satisfies \eqref{xport_v}.  First consider $\Upsilon = \Omega_+$.   Proposition \ref{xport_ext_exist} yields a weak solution $\bar{q}$ to \eqref{xport_extended} on $\Gamma$  obeying the estimates \eqref{xport_main_2} and \eqref{xport_main_3} on all of $\Gamma$.  We call $q$ the restriction of $\bar{q}$ to $\Upsilon$; Lemma \ref{xport_transfer} then guarantees that $q$ is the unique weak solution to \eqref{xport_eqn} on $\Upsilon$, but Proposition \ref{weak_unique} guarantees that $q$ coincides with the solution produced by characteristics in \eqref{xport_soln}.   The estimates \eqref{xwp_01} and \eqref{xwp_02} follow easily from \eqref{xport_main_2} and \eqref{xport_main_3} and Proposition \ref{xport_nonlins}.   A similar argument works for $\Upsilon = \Omega_-$.
\end{proof}

\subsection{Some more useful estimates for \eqref{xport_eqn} }

To conclude our analysis of the transport problem we record another a priori estimate for solutions to \eqref{xport_eqn}  that will be useful in the next section.

\begin{prop}\label{xport_low_reg_est}
 Suppose that $q$ is a solution to \eqref{xport_eqn} satisfying
\begin{equation}
  \sup_{0 \le t \le T} \E[q(t)] + \int_0^T \D[q(t)]dt < \infty.
\end{equation}
Assume also that
\begin{equation}
 \gamma = \sup_{0 \le t \le T} \left( \norm{c(t)}_{C^k} + \norm{v(t)}_{C^k}  \right) < \infty
\end{equation}
for some $1 \le k \le 4N$.  Then there exists a universal constant $C>0$ such that
\begin{equation}\label{xlre_01}
 \norm{q(t)}_k \le e^{C \gamma t} \norm{q_0}_k + \int_0^t e^{C \gamma (t-s)} \norm{f(s)}_k ds
\end{equation}
for $t \in [0,T]$.  In particular, if $q_0 =0$ then
\begin{equation}\label{xlre_02}
 \sup_{0\le t \le T} \ns{q(t)}_k \le T e^{2 C \gamma T} \int_0^T \ns{f(s)}_k ds.
\end{equation}
\end{prop}

\begin{proof}
Let $\alpha \in \mathbb{N}^3$ with $\abs{\alpha } \le k$.  Applying $\p^\alpha$ to \eqref{xport_eqn} leads to the equation
\begin{equation}
\begin{cases}
\dt \p^\alpha q + v \cdot \nab \p^\alpha q + c \p^\alpha q = \p^\alpha f - f^\alpha, \\
\p^\alpha q(t=0) = \p^\alpha q_0
\end{cases}
\end{equation}
where
\begin{equation}
 f^\alpha = \sum_{\beta < \alpha} C_{\alpha,\beta} \left( \p^{\alpha- \beta} v \cdot \nab \p^\beta q +  \p^{\alpha- \beta} c    \p^\beta q \right).
\end{equation}
Because of the condition \eqref{xport_v} we may multiply by $\p^\alpha q$ and integrate to deduce the standard energy identity
\begin{equation}
 \frac{d}{dt} \frac{\ns{\p^\alpha q}_0 }{2} = \int_\Upsilon (\diverge v - 2c) \frac{\abs{\p^\alpha q}}{2} + \int_\Upsilon (\p^\alpha f - f^\alpha)\p^\alpha q.
\end{equation}
From this  and the structure of $f^\alpha$ we may easily deduce the estimate
\begin{equation}
 \frac{d}{dt} \frac{\ns{\p^\alpha q}_0 }{2} \le \left( \norm{c}_{C^0} + \hal \norm{v}_{C^1} \right) \ns{\p^\alpha q}_0 + C\left( \norm{c}_{C^k} +  \norm{v}_{C^k}\right) \norm{q}_k \norm{\p^\alpha q}_0 + \int_\Upsilon \p^\alpha f \p^\alpha q
\end{equation}
for some universal $C>0$.
Summing this inequality over all $\abs{\alpha} \le k$, we find (since $k \ge 1$) that
\begin{equation}\label{xlre_1}
 \frac{d}{dt} \frac{ \ns{q}_{k} }{2} \le C \gamma \ns{q}_k + \norm{f}_k \norm{q}_k,
\end{equation}
where $C>0$ is another universal constant.  The differential inequality   \eqref{xlre_1} and a standard Gronwall argument then imply \eqref{xlre_01}.  The bound \eqref{xlre_02} follows from \eqref{xlre_01} and the Cauchy-Schwarz inequality:
\begin{multline}
 \left(\int_0^t e^{C \gamma (t-s)} \norm{f(s)}_k ds \right)^2 \le e^{2C \gamma t} \left(\int_0^t \ns{f(s)}_k ds\right) \left( \int_0^t e^{-2C \gamma s} ds  \right) \\
 \le t e^{2C \gamma t} \int_0^t \ns{f(s)}_k ds \le T e^{2C \gamma T} \int_0^T \ns{f(s)}_k ds.
\end{multline}
\end{proof}

\section{Local well-posedness of \eqref{geometric} with $\sigma_\pm >0$} \label{sec_lwp_st}

\subsection{Data: construction and estimates}

Our  goal now is to deal with the initial data.   We assume initially that
\begin{equation}
 u_0 \in H^{4N}, \eta_0 \in H^{4N+1/2}, \sqrt{\sigma} \nab_\ast \eta_0 \in H^{4N}, \text{ and } q_0 \in H^{4N}.
\end{equation}
We may then construct the data for the higher temporal derivatives as in Appendix \ref{app_ccs}.  This process leads us to the following estimate.

\begin{prop}\label{data_estimate}
 Suppose that $\Lf[\eta_0] \le \delta_1/2$, where  $\Lf[\eta_0]$ is given by \eqref{l0_def} and $\delta_1>0$ is from Theorem \ref{kappa_apriori}.   Then
\begin{equation}
 \sum_{j=1}^{2N} \ns{\dt^j u(0)}_{4N-2j} +  \ns{\dt^j q(0)}_{4N-2j+1}+  \ns{\dt^j \eta(0)}_{4N-2j+3/2}  \ls \ns{u_0}_{4N} + \ns{\eta_0}_{4N+1/2} + \ns{q_0}_{4N}.
\end{equation}
Consequently,
\begin{equation}
 \TE \ls \EEE \le \TE.
\end{equation}

\end{prop}
\begin{proof}
 The proof is essentially the same as that of Proposition 5.3 of \cite{GT_lwp} and is thus omitted.
\end{proof}

\subsection{Approximate solutions}

Suppose that $u, q, \eta$ are given.  We then define the forcing terms $F^1$ on $\Omega$ and $F^2_\pm$ on $\Sigma_\pm$ according to
\begin{equation}\label{forcing_def_1}
 F^1[u,q,\eta] = -(\bar{\rho} + q + \p_3 \bar{\rho} \theta)(-K \dt \theta \p_3 u + u \cdot \naba u)
 - \bar{\rho} \naba (h'(\bar{\rho}) q) - \naba \mathcal{R}
 - g(q + \p_3 \bar{\rho} \theta) \naba \theta,
\end{equation}
\begin{equation}\label{forcing_def_2}
 F^{2}_+[q,\eta]  = -P_+'(\bar{\rho}_+) q_+ \n_++g\eta_+ \n_+ - \mathcal{R}_+ \n_+ - \sigma_+ \diverge_\ast\left((1+ \abs{\nab_\ast \eta_+}^2)^{-1/2} -1) \nab_\ast \eta_+  \right) \n_+,
\end{equation}
and
\begin{equation}\label{forcing_def_3}
-F^2_-[q,\eta] = -\jump{P'(\bar{\rho}) q}\n_- +\rj g\eta_-\n_- - \jump{\mathcal{R}} \n_- + \sigma_- \diverge_\ast\left((1+ \abs{\nab_\ast \eta_-}^2)^{-1/2} -1) \nab_\ast \eta_-  \right) \n_-.
\end{equation}
Here the function $\mathcal{R}$ is determined by \eqref{R_def}.  We also define the forcing term $f$ on $\Omega$ via
\begin{equation}\label{forcing_def_4}
 f[u,\eta] = -\diva(\bar{\rho} u) +  K \dt \theta \p_3^2 \bar{\rho} \theta -     \diva( \p_3 \bar{\rho} \theta u).
\end{equation}
Finally we define the density function
\begin{equation}\label{forcing_def_5}
 \rho[q,\eta] =  \bar{\rho} +q +  \p_3 \bar{\rho} \theta  .
\end{equation}

We define the following functionals for use in the forcing estimates:
\begin{multline}
 \mathfrak{Y}_\infty[u] := \sum_{j=0}^{2N-1} \ns{\dt^j u}_{4N-2j-1},  \mathfrak{Y}_\infty[q] := \sum_{j=0}^{2N-1} \ns{\dt^j q}_{4N-2j-1}, \\
 \mathfrak{Y}_\infty[\eta] := \ns{\eta}_{4N-1/2} + \sum_{j=1}^{2N} \ns{\dt^j \eta}_{4N-2j+1/2}
\end{multline}
and
\begin{multline}
 \mathfrak{Y}_2[u] := \sum_{j=0}^{2N} \ns{\dt^j u}_{4N-2j},  \mathfrak{Y}_2[q] := \sum_{j=0}^{2N} \ns{\dt^j q}_{4N-2j}, \\
 \mathfrak{Y}_2[\eta] := \ns{\eta }_{4N} + \sum_{j=1}^{2N+1} \ns{\dt^j \eta}_{4N-2j+3/2}.
\end{multline}
For the sake of brevity we will write multiple arguments within brackets to indicate sums; for example, $\mathfrak{Y}_\infty[u,q,\eta] = \mathfrak{Y}_\infty[u] + \mathfrak{Y}_\infty[q]+ \mathfrak{Y}_\infty[\eta]$.

Our next two results contain the estimates of the forcing terms used in the problems \eqref{lame_free_bndry} and \ref{xport_fundamental}.  The proofs are again standard nonlinear estimates and are thus omitted.

\begin{prop}\label{force_est_1}
Let $F^1$ and $F^2 $ be determined by $(u,q,\eta)$ as in \eqref{forcing_def_1}--\eqref{forcing_def_3}.  Let $\f_2$ and $\f_\infty$ be given by this choice of $F^1, F^2$ as in \eqref{F2_def} and \eqref{Finf_def}.  Then we have the estimates
\begin{equation}
 \f_\infty \ls \mathfrak{Y}_\infty[q,\eta]( 1 + P(\mathfrak{Y}_\infty[\eta] )) +  P(\mathfrak{Y}_\infty[u,q,\eta])\mathfrak{Y}_\infty[u,q,\eta] + \sigma^2 \ns{\nab_\ast \eta}_{4N-1/2} P(\mathfrak{Y}_\infty[\eta] )
\end{equation}
and
\begin{equation}
 \f_2 \ls   \mathfrak{Y}_2[q,\eta]( 1 + P(\mathfrak{Y}_\infty[\eta] )) + P(\mathfrak{Y}_\infty[u,q,\eta]) \mathfrak{Y}_2[u,q,\eta] + \sigma^2 \ns{\nab_\ast \eta}_{4N+1/2} P(\mathfrak{Y}_\infty[\eta] )
\end{equation}
for some universal positive polynomial $P$ such that $P(0)=0$.

\end{prop}

\begin{prop}\label{force_est_2}
 Let $f$ be determined by $(u,\eta)$ as in \eqref{forcing_def_4}.  Let $\mathfrak{R}_e$ be defined by \eqref{rqe_def}.  Then we have the estimates
\begin{equation}
  \mathfrak{R}_e[f] \le \mathfrak{Y}_\infty[u] ( 1 + P(\mathfrak{Y}_\infty[\eta] ))  + P(\mathfrak{Y}_\infty[\eta] )\mathfrak{Y}_\infty[\eta]
\end{equation}
and
\begin{equation}
  \mathfrak{R}_d[f] + \ns{f}_{4N} \le \D[u] ( 1 + P(\hat{\E}^0[\eta] ))  + P(\hat{\E}^0[\eta] )\left(\ns{\eta}_{4N+1/2} + \hat{\D}^0[\eta]\right)
\end{equation}
for some universal positive polynomial $P$ such that $P(0)=0$.

\end{prop}

The following proposition allows us to estimate $\mathfrak{Y}_\infty[u,q,\eta]$ in terms of $\E$ and $\D$ with the benefit of introducing time factors.

\begin{prop}\label{Y_est}
We have the following estimates:
\begin{equation}
\sup_{0 \le t \le T} \mathfrak{Y}_\infty[u(t)] \ls \TE + T \int_0^T \D[u(t)]dt,
\end{equation}
\begin{equation}
\sup_{0 \le t \le T} \mathfrak{Y}_\infty[q(t)] \ls \TE +   T^2 \sup_{0 \le t \le T} \E[q(t)],
\end{equation}
and
\begin{equation}
\sup_{0 \le t \le T} \mathfrak{Y}_\infty[\eta(t)] \ls \TE  + T \int_0^T \hat{\D}^0[\eta(t)]dt.
\end{equation}
\end{prop}
\begin{proof}
 The estimates follow easily from the fundamental theorem of calculus and the Cauchy-Schwarz inequality.
\end{proof}

Now we present the construction of a sequence of approximate solutions.

\begin{thm}\label{approx_solns}
Suppose that $(u_0,q_0,\eta_0)$ satisfy the compatibility conditions \eqref{ccs} as well as the bound \eqref{q_0_assump}.  Further assume that
\begin{equation}
\Lf[\eta_0] \le  \frac{\delta_1}{2} ,
\end{equation}
where $\Lf[\eta_0]$  is given by  \eqref{l0_def} and $\delta_1$ is given by  Theorem \ref{kappa_apriori}.  Then there exists a $T_4 = T_4({\EEE})$ such that if  $0 < T \le T_4$ then there exists a sequence $\{(u^n,q^n,\eta^n)\}_{n=0}^\infty$ defined on the temporal interval $[0,T]$ satisfying the following three properties.  First, $(u^n,q^n,\eta^n)$ achieve the initial data at $t=0$.  Second, for $n \ge 1$ we have that
 \begin{equation}\label{aps_01}
  \begin{cases}
  \rho^{n-1} \dt u^n - \diverge_{\a^{n}} \mathbb{S}_{\a^{n}} u^n = F^1[u^{n-1},q^{n-1},\eta^{n-1}] & \text{in }\Omega\\
  \dt \eta^{n}  = u^{n} \cdot \n^{n}   &\text{on }\Sigma \\
  -\mathbb{S}_{\a^{n}} u^n  \n^{n} =   - \sigma_+ \Delta_\ast \eta^{n}  \n^{n} +  F^2_+[q^{n-1},\eta^{n-1}]   &\text{on } \Sigma_+ \\
  -\jump{\mathbb{S}_{\a^{n}} u^{n} } \n^{n} =  \sigma_- \Delta_\ast \eta^{n}  \n^{n} - F^2_-[q^{n-1},\eta^{n-1}] &\text{on } \Sigma_- \\
  \jump{u^{n}} =0 &\text{on } \Sigma_- \\
  u_-^{n} = 0 &\text{on } \Sigma_b \\
  u^{n}(\cdot,0) = u_0, \eta^{n}(\cdot, 0) = \eta_0,
 \end{cases}
\end{equation}
and
\begin{equation}\label{aps_02}
\begin{cases}
  \dt q^{n}  - K^{n} \dt \theta^{n} \p_3 q^{n} + \diverge_{\a^{n}}(q^{n}u^{n})
    = f[u^{n},\eta^{n}] & \text{in } \Omega \times (0,T) \\
 q^{n}(\cdot, 0) = q_0 &\text{in }\Omega,
\end{cases}
\end{equation}
where $F^1$, $F^2$, and $f$  are defined by \eqref{forcing_def_1}--\eqref{forcing_def_4} and $\rho^{n-1} = \rho[q^{n-1},\eta^{n-1}]$ is given by \eqref{forcing_def_5}.  Third, we have the estimates
\begin{multline}\label{aps_03}
 \sup_{0\le t \le T} \left( \E[u^{n}(t)] + \hat{\E}^\sigma[\eta^{n}(t)] \right)  \\
 + \int_0^T \left(  \D[u^{n}(t)] + \ns{\rho^{n-1} J^{n} \dt^{2N+1} u^{n}(t)}_{ \Hd}  + \hat{\D}^\sigma[\eta^{n}(t)] \right) dt \le  P_1( {\EEE}),
\end{multline}
\begin{equation}\label{aps_04}
 \sup_{0\le t \le T} \ns{\eta^{n}(t)}_{4N+1/2} \le P_2( {\EEE}),
\end{equation}
\begin{equation}\label{aps_05}
 \sup_{0\le t \le T}   \E[q^{n}(t)] + \int_0^T \D[q^{n}(t)] dt    \le P_3( {\EEE}) ,
\end{equation}
\begin{equation}\label{aps_06}
\Lf[\eta^n](T) \le \delta_1 \text{ and }\frac{1}{2}\rho_\ast\le \rho^n=\rho[q^{n},\eta^{n}]\le \frac{3}{2}\rho^\ast
\end{equation}
for all $n \ge 1$, where $P_i$ for $i=1,2,3$ is a universal  positive polynomial with $P_i(0)=0$.

\end{thm}

\begin{proof}

We divide the proof into three steps.

Step 1 - Seeding the sequence

To begin, we extend the initial data to a triple that belongs to the function spaces necessary for the construction of solutions.  We combine the data estimates of Proposition \ref{data_estimate} with the extension results of Propositions \ref{extension_u}, \ref{extension_eta}, and \ref{extension_q} in order to produce a triple $(u^0,q^0,\eta^0)$ defined on the temporal interval $[0,\infty)$, achieving the initial data, and satisfying the estimates
\begin{equation}\label{aps_1}
 \sup_{t \ge 0} \E[u^0(t)] + \int_0^\infty \D[u^0(t)] dt \le  P_0(\EEE),
\end{equation}
\begin{equation}
 \sup_{t \ge 0} \E[q^0(t)] + \int_0^\infty \D[q^0(t)] dt \le P_0(\EEE),
\end{equation}
and
\begin{equation}\label{aps_2}
 \sup_{t \ge 0} \hat{\E}^\sigma[\eta^0(t)] + \sup_{t \ge 0} \ns{\eta^0(t)}_{4N+1/2} + \int_0^\infty \hat{\D}^\sigma[\eta^0(t)]dt \le P_0(\EEE)
\end{equation}
for some universal polynomial $P_0>0$ with $P_0(0)=0$.

Step 2 - The iteration procedure

We claim that there exist universal positive polynomials $P_i$ for $i=1,2,3$ such that $P_i(0)=0$, and $\alpha >0$ (depending on ${\EEE}$) with the following two properties.  First,
\begin{equation}\label{aps_3}
\min\{ P_1(z),P_2(z), P_3(z)\} \ge P_0(z) \text{ for all } z \ge 0.
\end{equation}
That is, each of the coefficients of $P_i(z)$, $i=1,2,3$, is bounded below by the corresponding coefficient of $P_0(Z)$. Second, if $T \le \min\{\alpha,T_3\}$ (where $T_3= T_3(\TE)>0$ is given by Theorem \ref{lame_exist}),  and the triple $(u^{n-1},q^{n-1},\eta^{n-1})$ is given, achieves the initial data, and obeys the estimates
\begin{equation}\label{aps_4}
 \sup_{0\le t \le T} \left( \E[u^{n-1}(t)] + \hat{\E}^\sigma[\eta^{n-1}(t)] \right)  + \int_0^T \left(  \D[u^{n-1}(t)] + \hat{\D}^\sigma[\eta^{n-1}(t)] \right) dt \le P_1( {\EEE}),
\end{equation}
\begin{equation}\label{aps_5}
 \sup_{0\le t \le T} \ns{\eta^{n-1}(t)}_{4N+1/2} \le P_2( {\EEE}),
\end{equation}
and
\begin{equation}\label{aps_6}
 \sup_{0\le t \le T}   \E[q^{n-1}(t)]     \le P_3({\EEE}),
\end{equation}
then there exists a triple $(u^{n},q^{n},\eta^{n})$ that solves \eqref{aps_01} and \eqref{aps_02}, achieves the initial data, and obeys the estimates
\begin{multline}\label{aps_7}
 \sup_{0\le t \le T} \left( \E[u^{n}(t)] + \hat{\E}^\sigma[\eta^{n}(t)] \right)  \\
 + \int_0^T \left(  \D[u^{n}(t)] + \ns{\rho^{n-1} J^{n} \dt^{2N+1} u^{n}(t)}_{ \Hd}  + \hat{\D}^\sigma[\eta^{n}(t)] \right) dt \le P_1( {\EEE}),
\end{multline}
\begin{equation}\label{aps_8}
 \sup_{0\le t \le T} \ns{\eta^{n}(t)}_{4N+1/2} \le P_2( {\EEE}),
\end{equation}
and
\begin{equation}\label{aps_9}
 \sup_{0\le t \le T}   \E[q^{n}(t)] + \int_0^T \D[q^{n}(t)] dt    \le P_3( {\EEE}).
\end{equation}

The proof of the claim is very similar to Step 2 of Theorem 6.1 in \cite{GT_lwp}, so we will only provide a sketch of the idea.  It suffices to show that if \eqref{aps_4}--\eqref{aps_6} hold for some choice of $P_i$, $i=1,2,3$, satisfying \eqref{aps_3}, then $(u^n,q^n,\eta^n)$ can be constructed and must satisfy \eqref{aps_7}--\eqref{aps_9} so long as the constants and degree of the $P_i$ are sufficiently large and $T \le \alpha$ for some small $\alpha$.

The first step is to use Theorem \ref{lame_exist} to produce a $(u^n,\eta^n)$ solving \eqref{aps_01}.  For this we must verify that the hypotheses $\mathfrak{P}(\delta_1)$ are satisfied.  The hypotheses \eqref{lame_assump_1} and \eqref{lame_assump_2} are satisfied by assumption.  The hypothesis \eqref{lame_assump_3} follows by combining the estimates of Propositions \ref{data_estimate}, \ref{force_est_1}, and \ref{Y_est} with the bounds \eqref{aps_4}--\eqref{aps_6}.  The hypothesis \eqref{lame_assump_4} requires that $\rho^{n-1} = \rho[q^{n-1},\eta^{n-1}]$ satisfies \eqref{rho_assump_1} and \eqref{rho_assump_2}.  The condition \eqref{rho_assump_2} follows trivially from \eqref{aps_4} and \eqref{aps_5}.  To verify \eqref{rho_assump_1} we first  estimate
\begin{multline}\label{aps_10}
\sup_{0\le t \le T}\norm{\rho^{n-1}(t)-\rho_0 }_{L^\infty} \le  \int_0^T \norm{\dt \rho^{n-1}(t)}_{L^\infty} dt
\ls  \int_0^T \left(\norm{\dt q^{n-1}(t)}_{L^\infty}+\norm{\dt \eta^{n-1}(t)}_{L^\infty}\right) dt \\
\le  T   \sup_{0 \le t \le T}\sqrt{\E[q^{n-1}(t)]+ \hat{\E}^0[\eta^{n-1}(t)]}
\le   T \sqrt{P_1({\EEE})+P_3({\EEE})} \le \alpha\sqrt{P_1({\EEE})+P_3({\EEE})}.
\end{multline}
Then we find that
\begin{equation}\label{aps_12}
 \sup_{0\le t \le T}\norm{\rho^{n-1}(t)-\rho_0 }_{L^\infty} \le \frac{\rho_\ast}{2}
\end{equation}
if $\alpha$ is chosen sufficiently small with respect to the $P_1$, $P_3$, ${\EEE}$ and $\rho_\ast$.  Hence  \eqref{rho_assump_1} is satisfied by the assumption on $\rho_0$ in \eqref{q_0_assump}.

We may thus apply Theorem \ref{lame_exist} to produce the solution pair $(u^n,\eta^n)$ solving \eqref{aps_01}.  To derive the estimate \eqref{aps_7} we sum \eqref{lamex_01} and \eqref{lamex_04} and again employ  Propositions \ref{data_estimate}, \ref{force_est_1}, and \ref{Y_est} and  \eqref{aps_4}--\eqref{aps_6} to estimate the forcing terms.  The actual derivation of \eqref{aps_7} is tedious and will be omitted, but we will point out the key observation.  The estimates of Proposition \ref{Y_est} guarantee that any appearance of $P_i$ in the resulting estimates is multiplied by at least one factor of $T$ and so by choosing $\alpha$ small enough (in particular a bound like $\alpha \le \alpha_0(1+{\EEE})^{-m}$ for $\alpha_0$ small and $m$ large is needed to reduce the degrees of various polynomials appearing in the estimates)  and the constants and degrees of $P_1$ large enough, we can show that \eqref{aps_7} holds. The estimate \eqref{aps_8} follows from \eqref{lamex_05} via a similar argument.

The second step is to use the newly-constructed pair $(u^n,\eta^n)$ to construct $q^n$, the solution to \eqref{aps_02}, through an application of Theorem \ref{xport_well-posed}.  The hypotheses of the theorem are satisfied due to \eqref{aps_7}, \eqref{aps_7}, and \eqref{aps_01}.  The estimates \eqref{xwp_01} and
\eqref{xwp_02} then lead to the estimate \eqref{aps_9} by employing Propositions \ref{force_est_2} and \ref{Y_est} and arguing as above, except that we use \eqref{aps_7} and \eqref{aps_8} since the forcing terms are generated by $(u^n,\eta^n)$.

Step 3 - Constructing the sequence

To conclude the proof we combine the previous two steps as follows.  We set $(u^0,q^0,\eta^0)$ to be the triple constructed in Step 1.  The bounds \eqref{aps_1}--\eqref{aps_2} imply \eqref{aps_4}--\eqref{aps_6} with $n=1$ due to \eqref{aps_3}.  We then set $T_3 = \min\{T_2,\alpha\}$ and use Step 2 to construct $(u^1,q^1,\eta^1)$ satisfying \eqref{aps_7}--\eqref{aps_9} and solving \eqref{aps_01} and \eqref{aps_02}.  The bounds allow us to iteratively apply Step 2 to produce $(u^n,q^n,\eta^n)$ for $n \ge 2$.  This produces the  sequence $\{(u^n,q^n,\eta^n)\}_{n=0}^\infty$ satisfying \eqref{aps_01}--\eqref{aps_05} for $n \ge 1$.  It remains only to prove \eqref{aps_06}.  The estimates of $ \rho^n=\rho[q^{n},\eta^{n}]$ can be derived exactly as in \eqref{aps_10}--\eqref{aps_12}.  The estimate of $\Lf[\eta^n](T)$ can be derived similarly:
\begin{multline}
\Lf[\eta^n](T) \le \frac{3}{2} \Lf[\eta_0] + 3 T \int_0^T \ns{\dt \eta^n(t)}_{4N-1/2} dt \le \frac{3 \delta_1}{4} + 3T^2  \sup_{0 \le t \le T} \hat{\E}^0[\eta^{n}(t)] \\
\le  \frac{3 \delta_1 }{4} + 3T^2 P_1( {\EEE}) \le \frac{3\delta_1}{4} + 3\alpha^2 P_1( {\EEE}) \le \delta_1
\end{multline}
if $\alpha$ is further restricted.
\end{proof}

\subsection{Contraction}

We wish to ultimately show that the sequence $\{(u^n,q^n,\eta^n)\}_{n=0}^\infty$ contracts in some lower-order regularity space than that given by \eqref{aps_03}--\eqref{aps_05}.  Our goal now is to prove such a contraction result.  We will prove the result in a somewhat more general context than within the sequence $\{(u^n,q^n,\eta^n)\}_{n=0}^\infty$ in order for the result to be applicable in proving uniqueness of solutions to \eqref{geometric}.

Before stating the result we define the low-regularity norms in which contraction occurs.  We define
\begin{equation}\label{w_def_start}
 \Wf_\infty[u] = \ns{u}_{2} + \ns{\dt u}_{0}
\text{ and }
 \Wf_2[u]  = \ns{u}_{3} + \ns{\dt u}_{1},
\end{equation}
\begin{equation}
 \Wf_\infty[\eta]  = \ns{\eta}_{5/2} + \ns{\dt \eta}_{3/2}  + \sigma \ns{\nab_\ast \eta}_{2} + \sigma \ns{\nab_\ast \dt  \eta}_{0}  \text{ and }\Wf_2[\eta]  =  \sigma^2 \ns{\eta}_{7/2} + \ns{\dt^2 \eta}_{1/2},
\end{equation}
and
\begin{equation}\label{w_def_end}
 \Wf_\infty [q]  = \ns{q}_{2} + \ns{\dt q}_{1}.
\end{equation}

Now we state our contraction result.

\begin{thm}\label{contraction_thm}

Suppose that the triples $(u^i,q^i,\eta^i)$ and $(v^i,p^i,\zeta^i)$ for $i=1,2$ satisfy
\begin{equation}\label{cot_01}
\begin{cases}
  \rho^{i} \dt u^i - \diverge_{\a^{i}} \mathbb{S}_{\a^{i}} u^i = F^1[v^{i},p^{i},\zeta^{i}] & \text{in }\Omega \\
  \dt \eta^{i}  = u^{i} \cdot \n^{i}   &\text{on }\Sigma \\
  -\mathbb{S}_{\a^{i}} u^i  \n^{i} =   - \sigma_+ \Delta_\ast \eta^{i}  \n^{i} +  F^2_+[p^{i},\zeta^{i}]   &\text{on } \Sigma_+ \\
  -\jump{\mathbb{S}_{\a^{i}} u^{i} } \n^{i} =  \sigma_- \Delta_\ast \eta^{i} \n^{i} - F^2_-[p^{i},\zeta^{i}] &\text{on } \Sigma_- \\
  \jump{u^{i}} =0 &\text{on } \Sigma_- \\
  u_-^{i} = 0 &\text{on } \Sigma_b \\
  u^{i}(\cdot,0) = u_0, \eta^{i}(\cdot, 0) = \eta_0,
 \end{cases}
\end{equation}
and
\begin{equation}\label{cot_02}
\begin{cases}
  \dt q^{i}  - K^{i} \dt \theta^{i} \p_3 q^{i} + \diverge_{\a^{i}}(q^{i}u^{i})
    = f[u^{i},\eta^{i}] & \text{in } \Omega \times (0,T) \\
 q^{i}(\cdot, 0) = q_0 &\text{in }\Omega,
\end{cases}
\end{equation}
where $\a^i,$ $\n^i,$ $\theta^i$, and $K^i$ are given by $\eta^i$, and $F^1$, $F^2_\pm$, and $f$  are defined by \eqref{forcing_def_1}--\eqref{forcing_def_4} and $\rho^{i} = \rho[p^{i},\zeta^{i}]$ is given by \eqref{forcing_def_5}.  Further suppose that
\begin{equation}
\max\left\{ \sup_{i} \sup_{0 \le t \le T}  \E[u^i,q^i,\eta^i] , \sup_{i} \sup_{0 \le t \le T}  \E[v^i,p^i,\zeta^i] \right\} \le  M^2,
\end{equation}
and that
\begin{equation}\label{cot_04}
 \Lf[\eta^1](T) \le \delta_1 \text{ and }\frac{1}{2}\rho_\ast\le \rho^1=\rho[p^{1},\zeta^{1}]\le \frac{3}{2}\rho^\ast,
\end{equation}
where $\Lf$  is given by  \eqref{l_def} and $\delta_1$ is given by  Theorem \ref{kappa_apriori}.

There exist universal constants $\gamma >0$ and a constant  $T_5 = T_5(M) \in (0,1)$ such that if $0 < T \le T_5$ and
\begin{equation}\label{cot_05}
  \max\left\{ \sup_{i} \sup_{0 \le t \le T} \ns{\eta^i(t)}_7, \sup_{i} \sup_{0 \le t \le T} \ns{\zeta^i(t)}_7  \right\}  \le \gamma^2,
\end{equation}
 then
\begin{multline}\label{cot_03}
 \sup_{0\le t \le T} \Wf_\infty[u^1(t) - u^2(t),q^1(t) - q^2(t),\eta^1(t) - \eta^2(t)]  + \int_0^T \Wf_2[u^1(t) - u^2(t),\eta^1(t) - \eta^2(t)] dt \\
 \le \hal   \sup_{0\le t \le T} \Wf_\infty[v^1(t) - v^2(t),p^1(t) - p^2(t) ,\zeta^1(t) - \zeta^2(t)] + \hal \int_0^T \Wf_2[v^1(t) - v^2(t),\zeta^1(t)-\zeta^2(t)] dt .
\end{multline}

\end{thm}

\begin{proof}

We divide the proof into several steps.

Step 1 -- Differences

To begin we  define $u= u^1 - u^2$, $q =q^1 - q^2$, $\eta = \eta^1 - \eta^2$, $v = v^1 - v^2$, $p = p^1 - p^2$, and $\zeta = \zeta^1 - \zeta^2$.   Then we subtract the equations \eqref{cot_01}  with $i=2$ from the same equations with $i=1$ to deduce equations for $(u,\eta)$.  We then apply $\pal$ for $\alpha \in \mathbb{N}^{1+2}$ with $\abs{\alpha}\le 2$.  This results in the equations
\begin{equation}\label{cot_1}
\begin{cases}
  \rho^{1} \dt \pal u - \diverge_{\a^{1}} \mathbb{S}_{\a^{1}} \pal u = \diverge_{\a^{1}} \mathbb{S}_{\pal (\a^{1}-\a^{2})}  u^2 + \pal H^1 + H^{1,\alpha} & \text{in }\Omega \\
  \dt \pal \eta + u^2 \cdot \nab_\ast \pal \eta  = \pal u  \cdot \n^{1} + H^{3,\alpha}  &\text{on }\Sigma\\
  -\mathbb{S}_{\a^{1}} \pal u  \n^{1} =  - \sigma_+ \Delta_\ast \pal \eta  \n^{1} +  \mathbb{S}_{\pal (\a^{1}-\a^{2})} u^2 \n^1 +  \pal  H^2_+ + H^{2,\alpha}_+   &\text{on } \Sigma_+ \\
  -\jump{\mathbb{S}_{\a^{1}} \pal u } \n^{1} =  \sigma_- \Delta_\ast \pal \eta   \n^{1} + \jump{\mathbb{S}_{\pal (\a^{1}-\a^{2})} u^2} \n^1 - \pal H^2_- - H^{2,\alpha}_- &\text{on } \Sigma_- \\
  \jump{\pal u} =0 &\text{on } \Sigma_- \\
  \pal u_- = 0 &\text{on } \Sigma_b \\
  \pal  u(\cdot,0) = 0, \pal  \eta(\cdot, 0) = 0.
 \end{cases}
\end{equation}
A similar argument with \eqref{cot_02} but not employing derivatives yields an equation for $q$:
\begin{equation}\label{cot_2}
\begin{cases}
  \dt q  - K^{1} \dt \theta^{1} \p_3 q + \diverge_{\a^{1}}(q u^{1})
    =   H^4  & \text{in } \Omega \times (0,T) \\
 q(\cdot, 0) =0 &\text{in }\Omega.
\end{cases}
\end{equation}
Here we have written the forcing terms as follows:
\begin{equation}\label{cot_3}
\begin{split}
 H^1 & := F^1[v^1,p^1,\zeta^1] - F^1[v^2,p^2,\zeta^2] - (p + \p_3 \bar{\rho} \theta[\zeta]) \dt u^2  + G^1 \\
 H^2_\pm  & :=  F^2_\pm[p^1,\zeta^1] - F^2_\pm[p^2,\zeta^2]  + G^2_\pm \\
 H^4 &:= f[u^1,\eta^1] - f[u^2,\eta^2] + G^4,
\end{split}
\end{equation}
where $\theta[\zeta]$ is determined by $\zeta$,
\begin{equation}\label{cot_4}
\begin{split}
 G^1 &:=     \diverge_{(\a^{1}-\a^{2})} \mathbb{S}_{\a^{2}} u^2 \\
 G^2_+ &:=  - \sigma_+ \Delta_\ast \eta^2  (\n^1 - \n^2)  + \mathbb{S}_{\a^{2}} u^2 (\n^1 - \n^2) \\
 G^2_- &:= -  \sigma_- \Delta_\ast \eta^2 (\n^1 - \n^2)   + \jump{\mathbb{S}_{\a^{2}} u^2} (\n^1 - \n^2) \\
 G^4 &:= ((K^1-K^2) \dt \theta^1 + K^2 (\dt \theta^1 -\dt \theta^2))\p_3 q^2 -  \diverge_{(\a^1-\a^{2})}(q^2 u^1) -  \diverge_{\a^{2}}(q^2 u),
\end{split}
\end{equation}
and
\begin{equation}
\begin{split}
 H^{1,\alpha} & := \left(\pal(\diverge_{\a^{1}} \mathbb{S}_{\a^{1}}   u) - \diverge_{\a^{1}} \mathbb{S}_{\a^{1}} \pal u \right) + \left( \pal (\diverge_{\a^{1}} \mathbb{S}_{(\a^{1}-\a^{2})}  u^2) - \diverge_{\a^{1}} \mathbb{S}_{\pal (\a^{1}-\a^{2})}  u^2 \right)\\
 H^{2,\alpha}_+  & := \left( \pal(\mathbb{S}_{\a^{1}}   u  \n^{1}) -\mathbb{S}_{\a^{1}} \pal u  \n^{1} \right) +  \left(  \pal ( \mathbb{S}_{ (\a^{1}-\a^{2})} u^2 \n^1)- \mathbb{S}_{\pal (\a^{1}-\a^{2})} u^2 \n^1 \right) \\
 & \quad - \sigma_+\left( \pal [  \Delta_\ast   \eta   \n^{1}] -  \Delta_\ast \pal \eta   \n^{1} \right)\\
 H^{2,\alpha}_-  & := \left( -\pal\jump{\mathbb{S}_{\a^{1}}   u  \n^{1}} +\jump{\mathbb{S}_{\a^{1}} \pal u } \n^{1} \right) +  \left(  -\pal \jump{ \mathbb{S}_{ (\a^{1}-\a^{2})} u^2 \n^1}+ \jump{\mathbb{S}_{\pal (\a^{1}-\a^{2})} u^2} \n^1 \right) \\
 & \quad - \sigma_-\left(  \pal [ \Delta_\ast   \eta ) \n^{1}] -\Delta_\ast \pal \eta  \n^{1} \right)\\
 H^{3,\alpha}  & :=  \left( - \pal (u^2 \cdot \nab_\ast \eta) + u^2 \cdot \nab_\ast \pal \eta \right) +
 \left( \pal (u \cdot \n^1) - \pal u \cdot \n^1  \right). \\
\end{split}
\end{equation}
We have written the forcing terms in this manner in order to isolate those terms depending on $(v^i,p^i,\zeta^i)$ from those depending on $(u^i,q^i,\eta^i)$ and to single out some special delicate terms.

Step 2 -- Energy estimate

The starting point for the contraction analysis is a basic energy estimate for  \eqref{cot_1}. Arguing as in Lemma \ref{kappa_en_ident} leads us to the equality
\begin{multline}\label{cot_5}
 \frac{d}{dt} \left( \int_\Omega \rho^1J^1 \frac{\abs{\pal u}^2}{2} + \int_{\Sigma} \frac{\abs{\pal \eta}^2}{2} + \sigma \frac{\abs{\pal\nab_\ast \eta}^2}{2}   \right)
 + \int_\Omega \frac{\mu J^1}{2} \abs{\sgz_{\a^1} \pal u}^2 + J^1 \mu' \abs{\diverge_{\a^1} \pal u}^2
\\
 = \int_\Omega J^1 \pal u \cdot (  \pal H^1 + H^{1,\alpha}) - \int_\Sigma \pal u \cdot ( \pal H^2 + H^{2,\alpha}) \\
 + \int_{\Sigma} (-  \pal \eta + \sigma \Delta_\ast \pal \eta) (u^2 \cdot \nab_\ast \pal \eta- H^{3,\alpha})
  + \int_{\Sigma}   \pal \eta (\pal u \cdot \n^1) \\
  + \int_\Omega \dt (J^1 \rho^1) \frac{\abs{\pal u}^2}{2}
-  \int_\Omega \frac{\mu J^1}{2} \sgz_{\a^1} \pal u : \sgz_{\pal(\a^1-\a^2)} u^2  + J^1 \mu'  (\diverge_{\a^1} \pal u )(  \diverge_{\pal(\a^1 - \a^2)}u^2).
\end{multline}

Let us define
\begin{equation}
\mathfrak{U}(t) = \sum_{ \substack{\alpha \in \mathbb{N}^{1+2} \\ \abs{\alpha} \le 2}  } \left( \int_\Omega \rho^1J^1 \frac{\abs{\pal u(t)}^2}{2} + \int_{\Sigma}   \frac{\abs{\pal \eta(t)}^2}{2} + \sigma \frac{\abs{\pal\nab_\ast \eta(t)}^2}{2} \right)
\end{equation}
and
\begin{equation}
 \mathfrak{V}(t) = \sum_{ \substack{\alpha \in \mathbb{N}^{1+2} \\ \abs{\alpha} \le 2}  } \ns{\pal u(t)}_{1}.
\end{equation}

We sum \eqref{cot_5}  over $\alpha \in \mathbb{N}^{1+2}$ with $\abs{\alpha} \le 2$;  applying Proposition \ref{korn} and arguing as per usual (as in Lemmas \ref{lame_forcing_est} and \ref{k_F_ests}, Theorem \ref{kappa_contract}, and Propositions \ref{kappa_improved} and \ref{force_est_1}) to estimate the various nonlinearities, we derive the differential inequality
\begin{multline}\label{cot_6}
 \frac{d}{dt} \mathfrak{U}(t) + C \mathfrak{V}(t) \ls  P(\gamma) \sqrt{\Wf_2[u]} \sqrt{\Wf_2[\zeta]} \\
 + (1 + P(M)) \sqrt{\Wf_\infty[u]} \left( \sqrt{\Wf_\infty[v,p,\zeta]} + \sqrt{\Wf_\infty[v,\zeta]}  \right)
 +  (1 + P(M)) \sqrt{\Wf_2[u]} \sqrt{\Wf_\infty[p,\zeta]} \\
 + (1+ P(M)) \Wf_\infty[u,\eta] + (1+ P(M)) \sqrt{\Wf_\infty[u,\eta]} \sqrt{\Wf_2[u,\eta]}
\end{multline}
for some universal positive polynomial with $P(0)=0$ and a universal constant $C>0$.  We should note that two of the terms appearing on the right side of \eqref{cot_5} require some delicate treatment.  The first are terms involving $\nab p$ in $H^1$.  In order to handle these when two horizontal spatial derivatives are applied, we must integrate by parts to move one horizontal derivative onto $J^1 \p^\alpha u$ and reduce to only two derivatives on $p$, which is all that is controlled by $\Wf_\infty[p]$.  The second are terms involving $\sigma_\pm$ multiplying two spatial derivatives of $\zeta_\pm$ in $H^2_\pm$; these give rise to the term  $P(\gamma) \sqrt{\Wf_2[u]} \sqrt{\Wf_2[\zeta]}$.   The key part of this is $P(\gamma)$, which appears because the nonlinear terms only involve spatial derivatives of $\zeta^i$.

Integrating  \eqref{cot_6} in time, using the fact that $\mathfrak{U}(0)=0$, and applying the Cauchy-Schwarz inequality then yields the bound
\begin{multline}\label{cot_7}
 \sup_{0\le t \le T} \mathfrak{U}(t) + \int_0^T \mathfrak{V}(t) dt \ls P(\gamma) \int_0^T \left(\Wf_2[\zeta] + \Wf_2[u] \right) dt  \\
 + \sqrt{T}(1 + P(M)) \int_0^T \left( \Wf_2[v,\zeta] + \Wf_2[u,\eta] \right) dt
 \\
 + (\sqrt{T} +T) (1 + P(M)) \sup_{0 \le t \le T} \left( \Wf_\infty[v,p,\zeta] + \Wf_\infty[u,\eta] \right),
\end{multline}
where again $P$ is a universal positive polynomial such that $P(0)=0$.

Throughout the rest of the proof we will let $\z$ denote a quantity of the form
\begin{equation}
 \z \simeq (1+ P(\gamma)) \left( \text{RHS of } \eqref{cot_7}  \right),
\end{equation}
where $P$ is some universal positive polynomial such that $P(0)=0$.  From one estimate to another the polynomials and constants may change, but the structure of $\z$ does not.

Step 3 -- Improved $u$ estimates

The usual trace theory allows us to estimate
\begin{equation}
 \int_0^T \ns{u}_{H^{5/2}(\Sigma)} \ls \int_0^T \sum_{\substack{\alpha \in \mathbb{N}^2 \\ \abs{\alpha}\le 2}}  \ns{\pal u}_{1} \ls \int_0^T \mathfrak{V}(t)dt \ls \z.
\end{equation}
We may  apply the elliptic estimate of Proposition \ref{lame_elliptic}, which is applicable due to the first estimate in \eqref{cot_04},  to bound
\begin{multline}
 \ns{u}_3 \ls \ns{\rho^1 \dt u}_1 + \ns{\diverge_{\a^{1}} \mathbb{S}_{ (\a^{1}-\a^{2})}  u^2 +  H^1}_1 + \ns{u}_{H^{5/2}(\Sigma)} \\
\ls  \ns{\rho^1}_{L^\infty}\ns{\dt u}_1+ \ns{\nab \rho^1}_{L^\infty} \ns{\dt u}_0 +   (1+P(M)) \left( \Wf_\infty[v,p,\zeta] + \Wf_\infty[\eta] \right) +   \ns{u}_{H^{5/2}(\Sigma)} \\
\ls   \mathfrak{V}  + (1+P(M)) \Wf_\infty[u] +   (1+P(M)) \left( \Wf_\infty[v,p,\zeta] + \Wf_\infty[\eta] \right) +   \ns{u}_{H^{5/2}(\Sigma)}.
\end{multline}
Here in the third inequality we have used the second estimate in \eqref{cot_04}. Hence
\begin{equation}\label{cot_8}
 \int_0^T \ns{u(t)}_3 dt \ls \int_0^T \mathfrak{V}(t) dt +   \z \ls \z.
\end{equation}
 We improve the $L^\infty$ in time estimate for $u$ by employing  Lemma \ref{time_interp}:
\begin{equation}\label{cot_9}
 \sup_{0\le t \le T} \ns{u(t)}_2 \ls \int_0^T (\ns{u(t)}_3 + \ns{\dt u(t)}_1 )dt \ls \int_0^T \left( \ns{u(t)}_3 +\mathfrak{V}(t)\right) dt \ls \z.
\end{equation}
Combining \eqref{cot_7}, \eqref{cot_8}, and \eqref{cot_9} then provides us with the bound
\begin{equation}\label{cot_30}
 \sup_{0 \le t \le T} \Wf_\infty[u] + \int_0^T \Wf_2[u] dt \ls \z.
\end{equation}

Step 4 -- Improved $\eta$ estimates

Now we improve the estimates for $\eta$.  Note first that we already have the bound
\begin{equation}\label{cot_10}
\sup_{0\le t \le T} \left(  \ns{\eta(t)}_2 + \ns{\dt \eta (t) }_0 + \sigma \ns{\nab_\ast \eta(t)}_2 + \sigma \ns{\nab_\ast \dt \eta(t) }_0  \right) \le \sup_{0\le t \le T} \mathfrak{U}(t) \ls \z.
\end{equation}
By solving the second and third equations in \eqref{cot_1} for $\sigma \Delta \eta$ and employing \eqref{cot_8}, \eqref{cot_9}, and \eqref{cot_10}, we may then bound
\begin{multline}\label{cot_11}
\int_0^T \sigma^2 \ns{\eta(t)}_{7/2}dt \ls  \int_0^T \left( \ns{\eta(t)}_{0} + \ns{\sigma \Delta_\ast \eta(t) }_{3/2} \right)dt \\
\ls T \sup_{0 \le t \le T}  \ns{\eta(t)}_0 + \int_0^T \left( (1+P(\gamma)) \Wf_2[u] + P(\gamma) \Wf_2[\zeta]   \right)dt  \\
+ T (1+P(M)) \sup_{0 \le t \le T} \left( \Wf_\infty[p,\zeta] + \Wf_\infty[\eta] \right)
\ls \z.
\end{multline}
Next we employ the kinematic equation in \eqref{cot_1} along with the transport estimates of Proposition 2.1 of \cite{danchin}  to bound
\begin{equation}
 \sup_{0\le t\le T} \ns{\eta(t)}_{5/2} \ls \exp\left(CT \int_0^T \ns{u^2(t)}_3 dt\right) T \int_0^T \ns{u\cdot \n^1 (t)}_{5/2} dt.
\end{equation}
Hence, if we assume that $T_5 M \le 1$ we may bound the exponential term above by a universal constant and then estimate
\begin{equation}\label{cot_12}
  \sup_{0\le t\le T} \ns{\eta(t)}_{5/2}  \ls T(1+P(\gamma)) \int_0^T \ns{u(t)}_3 dt \ls T  \z.
\end{equation}
Then we solve for $\dt \eta$ in \eqref{cot_01} to bound
\begin{equation}\label{cot_13}
 \sup_{0\le t \le T} \ns{\dt \eta(t)}_{3/2} \ls M  \sup_{0\le t \le T} \ns{\eta(t)}_{5/2} + (1+P(\gamma)) \sup_{0\le t \le T} \ns{u(t)}_{2}
 \ls  MT \z + \z  \ls \z
\end{equation}
since $M T \le M T_5 \le 1$.  Similarly, we solve for $\dt^2 \eta$ to estimate
\begin{multline}\label{cot_14}
 \int_0^T \ns{\dt^2 \eta(t)}_{1/2} dt \ls MT \sup_{0 \le t \le T} \ns{\dt \eta(t)}_{3/2} + (1+\gamma) \int_0^T \ns{\dt u(t)}_1 dt \\
 + MT \sup_{0\le t\le T} \left( \ns{u(t)}_2 + \ns{\eta(t)}_{5/2} \right) \ls \z.
 \end{multline}
Summing \eqref{cot_10}, \eqref{cot_11}, and \eqref{cot_12}--\eqref{cot_14} then yields the bound
\begin{equation}\label{cot_31}
 \sup_{0 \le t \le T} \Wf_\infty[\eta] + \int_0^T \Wf_2[\eta] dt \ls \z.
\end{equation}

Step 5 -- Estimates of $q$

Next we employ Proposition \ref{xport_low_reg_est} to get estimates for $q$.  First we find that
\begin{equation}
 \sup_{0 \le t \le T} \ns{q(t)}_2 \ls \exp(C (1+ P(M)) T) T \int_0^T \ns{H^4(t)}_{2} dt.
\end{equation}
If we further restrict $T_5$ so that $(1+P(M))T \le 1$ we can again treat the exponential as a universal constant.  Then
\begin{multline}\label{cot_15}
  \sup_{0 \le t \le T} \ns{q(t)}_2 \ls  T \int_0^T \ns{H^4(t)}_{2} dt \ls T (1+P(M))\int_0^T \left( \Wf_2[u]  + \Wf_\infty[u,\eta]  \right)dt  \\
  \ls T(1+P(M)) \z.
\end{multline}

Next we use \eqref{cot_01} to solve for  $\dt q$ and estimate
\begin{equation}\label{cot_16}
 \sup_{0 \le t \le T} \ns{\dt q(t)}_1 \ls P(M) \sup_{0 \le t \le T} \ns{q(t)}_2  + \sup_{0 \le t \le T} \ns{H^4}_1
\end{equation}
The term $H^4$ may be estimated as follows.  First we use the fact that $u(t=0)=0$ and $\eta(t=0)=0$ to estimate
\begin{equation}\label{cot_33}
 \sup_{0\le t \le T} \ns{\dt \eta(t)}_{1/2} + \sup_{0\le t \le T} \ns{u(t)}_{1} \le T \int_0^T \left( \ns{\dt^2 \eta(t)}_{1/2} + \ns{\dt u(t)}_1\right) dt \le T \z.
\end{equation}
Then we bound
\begin{multline}\label{cot_17}
 \sup_{0 \le t \le T} \ns{H^4(t)}_1 \ls P(M)  \sup_{0 \le t \le T}  \ns{\eta(t)}_{3/2} + P(M)   \sup_{0 \le t \le T}  \ns{\dt \eta(t)}_{1/2} \\
 + P(M) \sup_{0\le t \le T} \ns{u(t)}_1  + (1+ P(\gamma))  \sup_{0 \le t \le T} \ns{u(t)}_2
 \\
 \ls T P(M) \z + (1+ P(\gamma))  \z \ls  \z
\end{multline}
if $T_5$ is further restricted so that $T_5 P(M) \le 1$.  Note here that we have crucially employed the $T$ factor appearing on the right side of \eqref{cot_12} and \eqref{cot_33}.  We can now combine \eqref{cot_15}, \eqref{cot_16}, and \eqref{cot_17} to deduce that
\begin{equation}\label{cot_32}
 \sup_{0 \le t \le T} \Wf_\infty[q] \ls  T(1+P(M)) \z +\z\ls \z
\end{equation}
if $T_5$ is further restricted.

Step 6 -- Synthesis

Now we sum the estimates \eqref{cot_30}, \eqref{cot_31}, and \eqref{cot_32} to deduce that
\begin{equation}
 \sup_{0\le t \le T} \Wf_\infty[u,q,\eta](t) + \int_0^T \Wf_2[u,\eta](t) dt \ls \z.
\end{equation}
Assuming that $\gamma$  and $T_5$ are sufficiently small and using the previous inequality, we may absorb the terms involving $(u,q,\eta)$ from the right side to the left; this results in the estimate
\begin{multline}
 \sup_{0\le t \le T} \Wf_\infty[u,q,\eta](t) + \int_0^T \Wf_2[u,\eta](t) dt \ls P(\gamma) \int_0^T \Wf_2[\zeta]   dt  \\
 + \sqrt{T}(1 + P(M)) \int_0^T  \Wf_2[v,\zeta]  dt
 + (\sqrt{T} +T) (1 + P(M)) \sup_{0 \le t \le T}  \Wf_\infty[v,p,\zeta].
\end{multline}
By further restricting $\gamma$  and $T_5$ we deduce that \eqref{cot_03} holds.
\end{proof}

\subsection{Local well-posedness}

We now have all of the ingredients necessary to prove a more general version of Theorem \ref{local_existence_intro} in the case $\sigma_\pm >0$.

\begin{thm}\label{local_existence}
Assume that $\sigma_\pm >0$.  Suppose that $(u_0,q_0,\eta_0)$ satisfy the  compatibility conditions \eqref{ccs} as well as the bound  \eqref{q_0_assump}, and that
\begin{equation}\label{le_00}
 \Lf[\eta_0] \le  \frac{\delta_1}{2},
\end{equation}
where $\Lf[\eta_0]$  is given by  \eqref{l0_def} and $\delta_1$ is given by  Theorem \ref{kappa_apriori}.  Further assume that
\begin{equation}\label{le_00_2}
 \ns{\eta_0}_7 \le \frac{\gamma^2}{2},
\end{equation}
where $\gamma$  is as given in Theorem \ref{contraction_thm}.   Set $T_6 = T_6({\EEE}) = \min\{T_4({\EEE}),T_5(P_1({\EEE}) + P_3({\EEE})\}$, where $T_4$ and $P_1, P_3$ are given by Theorem \ref{approx_solns} and  $T_5$ is given by Theorem \ref{contraction_thm}.

If $0 < T \le T_6$ then there exists a triple $(u,q,\eta)$ defined on the temporal interval $[0,T]$ satisfying the following three properties.  First, $(u,q,\eta)$ achieve the initial data at $t=0$.  Second, the triple uniquely solve \eqref{geometric}.  Third, the triple obey the estimates
\begin{equation}\label{le_01}
 \sup_{0\le t \le T} \left( \E[u(t)] + \hat{\E}^\sigma[\eta(t)] \right)
 + \int_0^T \left(  \D[u(t)] + \ns{\rho J \dt^{2N+1} u(t)}_{ \Hd}  + \hat{\D}^\sigma[\eta(t)] \right) dt \le P_1({\EEE}),
\end{equation}
\begin{equation}\label{le_02}
 \sup_{0\le t \le T} \ns{\eta(t)}_{4N+1/2} \le P_2({\EEE}),
\end{equation}
\begin{equation}\label{le_03}
 \sup_{0\le t \le T}   \E[q(t)] + \int_0^T \D[q(t)] dt    \le P_3({\EEE}),
\end{equation}
and
\begin{equation}\label{le_04}
 \Lf[\eta](T) \le \delta_1 \text{ and }\frac{1}{2}\rho_\ast\le \rho=\rho[q,\eta]\le \frac{3}{2}\rho^\ast.
\end{equation}
\end{thm}

\begin{proof}

The proof is very similar to that of Theorem 6.2 in \cite{GT_lwp}, so we will only provide a quick sketch.

First we use Theorem \ref{approx_solns} to produce a sequence of approximate solutions $\{(u^n,q^n,\eta^n)\}_{n=0}^\infty$ on the temporal interval $[0,T]$.  The uniform estimates \eqref{aps_03}--\eqref{aps_06} along with standard compactness and weak compactness arguments yield a subsequence converging to a limiting triple $(u,q,\eta)$ that achieves the initial data and satisfies the estimates \eqref{le_01}--\eqref{le_04}.

Because we only know the convergence of a subsequence, we cannot immediately pass to the limit in \eqref{aps_01} and \eqref{aps_02}.  Instead we first use Theorem \ref{contraction_thm} to deduce that the sequence $\{(u^n,q^n,\eta^n)\}_{n=0}^\infty$ actually contracts in the lower regularity norm defined by \eqref{cot_03}.  In order to apply the theorem we must verify that \eqref{cot_04} and \eqref{cot_05} are satisfied; these follow from \eqref{le_00} and \eqref{le_00_2},  an argument like that used in \eqref{aps_10}--\eqref{aps_12}, and a further restriction of time.  This low regularity convergence, when combined with the bounds \eqref{aps_03}--\eqref{aps_06} and various interpolation arguments, shows that the original sequence actually converges to $(u,q,\eta)$ in a regularity class slightly larger than that defined by \eqref{aps_03}--\eqref{aps_06} but more than sufficient for passing to the limit in \eqref{aps_01}--\eqref{aps_02}.  We deduce then that $(u,q,\eta)$ satisfy \eqref{geometric}.  The uniqueness claim follows from another application of Theorem \ref{contraction_thm}.
\end{proof}

\section{Local well-posedness of \eqref{geometric} with $\sigma_\pm =0$}

We now state a result on the local existence of solutions to \eqref{geometric} without surface tension.

\begin{thm}\label{local_existence_no_ST}
Assume that $\sigma_\pm =0$.  Suppose that $(u_0,q_0,\eta_0)$ satisfy the compatibility conditions \eqref{ccs} as well as the bounds  \eqref{q_0_assump}, and that
\begin{equation}
 \Lf[\eta_0] \le  \frac{\delta_1}{2},
\end{equation}
where $\Lf[\eta_0]$  is given by  \eqref{l0_def} and $\delta_1$ is given by  Theorem \ref{kappa_apriori}.  Further assume that
\begin{equation}
 \ns{\eta_0}_7 \le \frac{\gamma^2}{2},
\end{equation}
where $\gamma$  are as given in Theorem \ref{contraction_thm}.   Set $T_6 = T_6({\EEE}) = \min\{T_4({\EEE}),T_5(P_1({\EEE}) + P_3({\EEE})\}$, where $T_4$ and $P_1, P_3$ are given by Theorem \ref{approx_solns} and  $T_5$ is given by Theorem \ref{contraction_thm}.

If $0 < T \le T_6$ then there exists a triple $(u,q,\eta)$ defined on the temporal interval $[0,T]$ satisfying the following three properties.  First, $(u,q,\eta)$ achieve the initial data at $t=0$.  Second, the triple uniquely solve \eqref{geometric}.  Third, the triple obey the estimates
\begin{equation}\label{le_st_01}
 \sup_{0\le t \le T} \left( \E[u(t)] + \hat{\E}^0[\eta(t)] \right)
 + \int_0^T \left(  \D[u(t)] + \ns{\rho J \dt^{2N+1} u(t)}_{ \Hd}  + \hat{\D}^0[\eta(t)] \right) dt \le P_1({\EEE}),
\end{equation}
\begin{equation}
 \sup_{0\le t \le T} \ns{\eta(t)}_{4N+1/2} \le P_2({\EEE}),
\end{equation}
\begin{equation}
 \sup_{0\le t \le T}   \E[q(t)] + \int_0^T \D[q(t)] dt    \le P_3({\EEE}),
\end{equation}
and
\begin{equation}\label{le_st_02}
\Lf[\eta](T) \le \delta_1\text{ and }\frac{1}{2}\rho_\ast\le \rho=\rho[q,\eta]\le \frac{3}{2}\rho^\ast.
\end{equation}
\end{thm}
\begin{proof}

The proof follows from an argument similar to that used in the proof of Theorem \ref{local_existence}, but actually somewhat easier and more akin to that used in Section 6 of \cite{GT_lwp}.  We will provide only a sketch of the ideas.

The main difference between the method to produce solutions to \eqref{geometric} with $\sigma_\pm >0$ and the method used with $\sigma_\pm =0$ lies in the use of the $\kappa$ approximation, which replaces the kinematic transport equation for $\eta$ with a parabolic problem.  This is essential in studying the problem with surface tension, as it leads to a regularity gain for $\eta$ that enables us to treat $\sigma_\pm \Delta_\ast \eta_\pm$ as a forcing term when solving for $u$.  However, when $\sigma_\pm =0$, this regularity gain is unnecessary, as only $\eta$ appears as a forcing term in the $u$ equation.  In place of the parabolic problem we simply study the kinematic transport problem directly, using Theorem 5.4 of \cite{GT_lwp} to produce solutions and derive estimates.

We then proceed essentially as in Section \ref{sec_lwp_st}.  First we prove that Theorem \ref{approx_solns} holds with $\sigma_\pm =0$.  The iteration scheme begins with a triple $(u^{n-1},q^{n-1},\eta^{n-1})$ and then uses Theorem \ref{lame_high} to produce $u^n$.  Then the equation $\dt \eta^n = u^n \cdot \n^n$ is solved using Theorem 5.4 of \cite{GT_lwp}.  Then $(u^n,\eta^n)$ are used to solve for $q^n$ in \eqref{aps_02} by way of Theorem \ref{xport_well-posed}.   Next we observe that Theorem \ref{contraction_thm} remains true as stated with $\sigma_\pm =0$.  Finally, we combine these two theorems to produce a sequence $\{(u^n,q^n,\eta^n\}_{n=1}^\infty$ of approximate solutions that remain uniformly bounded at high regularity and contract in a lower-regularity norm.
\end{proof}

\appendix

\section{Energy and dissipation functionals} \label{sec_en_dis}

Here we collect the definitions of various functionals that are used throughout the paper.  We define the energies associated to $(u,q,\eta)$ via
\begin{equation}\label{energy_def_u}
 \E[u] =  \sum_{j=0}^{2N} \ns{\dt^j  u}_{4N-2j},
\end{equation}
\begin{equation}\label{energy_def_q}
 \E[q] = \ns{q}_{4N} +\sum_{j=1}^{2N} \ns{\dt^j q}_{4N-2j+1},
\end{equation}
and
\begin{equation}\label{energy_def_eta}
 \E^\sigma[\eta] = \sum_{j=0}^{2N} \ns{\dt^j  \eta}_{ 4N-2j} + \sigma \ns{\dt^j \nab_\ast  \eta}_{ 4N-2j}.
\end{equation}

We define the corresponding dissipation functionals via
\begin{equation}\label{dissipation_def_u}
 \D[u] = \sum_{j=0}^{2N}   \ns{\dt^j u}_{ 4N-2j+1}
\end{equation}
\begin{equation}\label{dissipation_def_q}
 \D[q] = \ns{q}_{4N} +  \ns{\dt q}_{4N-1} + \sum_{j=2}^{2N+1} \ns{\dt^j q}_{4N-2j+2},
\end{equation}
and
\begin{equation}\label{dissipation_def_eta}
 \D^\sigma[\eta] = \sigma^2 \ns{\eta}_{ 4N+3/2} + \ns{\dt \eta}_{ 4N-1} + \sum_{j=2}^{2N+1} \ns{\dt^j \eta}_{ 4N-2j+2}.
\end{equation}

For $\eta$ we also need to define some improved terms:
\begin{equation}\label{energy_eta_improved}
 \hat{\E}^\sigma[\eta] = \E^\sigma[\eta] + \sum_{j=1}^{2N} \ns{\dt^j \eta }_{4N-2j+3/2}
\end{equation}
and
\begin{equation}\label{dissipation_eta_improved}
\hat{\D}^\sigma[\eta] = \D^\sigma[\eta]+ \ns{\dt \eta}_{4N-1/2}  +\sum_{j=2}^{2N} \ns{\dt^j \eta }_{4N-2j+5/2}.
\end{equation}
We must also define the term
\begin{equation}\label{l_def}
 \Lf[\eta](T) = \sup_{0\le t \le T} \ns{\eta(t)}_{4N-1/2}.
\end{equation}
In estimating this term we often refer to the following term associated with the data:
\begin{equation}\label{l0_def}
 \Lf[\eta_0] = \ns{\eta_0}_{4N-1/2}.
\end{equation}

For the data $(u_0,q_0,\eta_0)$ we define
\begin{equation}
 {\EEE} = \ns{u_0}_{4N} +\ns{q_0}_{4N} +  \ns{\eta_0}_{4N+1/2} + \sigma \ns{\nab_\ast \eta_0}_{4N},
\end{equation}
and when $\{(\dt^j u(0),\dt^j q(0),\dt^j\eta(0))\}_{j=0}^{2N}$ are known we write
\begin{multline}\label{TE_def1}
 \TE[u_0] = \sum_{j=0}^{2N} \ns{\dt^j  u(0)}_{4N-2j}, \;  \TE[q_0] = \ns{q(0)}_{4N} +\sum_{j=1}^{2N} \ns{\dt^j q(0)}_{4N-2j+1}, \\
\text{ and }
\TE[\eta_0] = \ns{\eta(0)}_{4N+1/2} + \sigma \ns{\nab_\ast  \eta(0)}_{ 4N} + \sum_{j=1}^{2N} \ns{\dt^j \eta(0) }_{4N-2j+3/2} 
\end{multline}
We will often abbreviate
\begin{equation}\label{TE_def2}
 \TE =  \TE[u_0] +  \TE[q_0] +  \TE[\eta_0] \text{ and }  \TE[u_0,\eta_0] =  \TE[u_0] +  \TE[\eta_0].
\end{equation}

\section{Compatibility conditions  }\label{app_ccs}
Here we record the system of compatibility conditions that the initial data $(u_0,q_0,\eta_0)$ must satisfy in order to produce high-regularity solutions to \eqref{geometric}.  To state the compatibility conditions we must first show how to construct  $(\dt^j u(\cdot,0), \dt^j q(\cdot,0), \dt^j \eta(\cdot,0))$ for $j=1,\dotsc,2N$ from the triple $(u_0,q_0,\eta_0)$.

For the purposes of constructing the temporal-derivative data we rewrite the first, second and third equations in \eqref{geometric} in the form
\begin{equation}\label{cc_data}
 \begin{split}
\dt \eta &= F_1(u,\eta) \\
\dt q &= F_2(u,q,\eta,\dt \eta) \\
\dt u &= F_3(u,q,\eta,\dt \eta).
 \end{split}
\end{equation}
Assuming that we are given  $\{(\dt^k u(\cdot,0), \dt^k q(\cdot,0), \dt^k \eta(\cdot,0))\}_{k=0}^j$ for some $j\in \{0,\dotsc,2N-1\}$, we construct  $(\dt^{j+1} u(\cdot,0), \dt^{j+1} q(\cdot,0), \dt^{j+1} \eta(\cdot,0))$ as follows.  First we apply $\dt^j$ to the first equation in \eqref{cc_data} and define
\begin{equation}
 \left.\dt^{j+1} \eta(\cdot,0) = \dt^j F_1(u,\eta) \right\vert_{t=0},
\end{equation}
which is possible because all terms appearing on the right are already known.   We can now perform a similar operation on the second and third equations in \eqref{cc_data}, setting
\begin{equation}
 \left.\dt^{j+1} u(\cdot,0) = \dt^j F_2(u,q,\eta,\dt \eta) \right\vert_{t=0} \text{ and } \left.\dt^{j+1} q(\cdot,0) = \dt^j F_3(u,q,\eta,\dt \eta) \right\vert_{t=0},
\end{equation}
both of which can be computed in terms of known quantities since we have already computed  $\dt^{j+1}\eta(\cdot,0)$.  Using this argument, we may inductively define $\{(\dt^j u(\cdot,0), \dt^j q(\cdot,0), \dt^j \eta(\cdot,0))\}_{j=1}^{2N}$ as desired.

We may now state the compatibility conditions.  We say that $(u_0,q_0,\eta_0)$ satisfy the compatibility conditions at level $2N$ if
\begin{equation} \label{ccs}
 \begin{cases}
  \dt^j\left(  P'(\bar\rho)\q \n - \S_{\a}(  u)\n \right) \vert_{t=0} = \dt^j\left( \bar{\rho}_1  g \eta \n-\sigma_+ \mathcal{H}_+  \n
- \mathcal{R} \n\right)\vert_{t=0} &\text{on } \Sigma_+ \\
  \jump{\dt^j\left( P'(\bar\rho)\q \n- \S_\a(u)\n \right) } \vert_{t=0} = \dt^j \left( \rj g\eta\n+\sigma_- \mathcal{H}_- \n  - \jump{ \mathcal{R} }\n \right) \vert_{t=0} &\text{on } \Sigma_- \\
  \jump{\dt^j u}\vert_{t=0} =0 &\text{on } \Sigma_- \\
  \dt^j u_-\vert_{t=0} = 0 &\text{on } \Sigma_b
 \end{cases}
\end{equation}
for $j=0,\dotsc,2N-1$.

\section{Poisson extension}

We will now define the appropriate Poisson integrals that allow us to extend $\eta_\pm$, defined on the surfaces $\Sigma_\pm$, to functions defined on $\Omega$, with ``good'' boundedness.

Suppose that $\Sigma_+ = \mathrm{T}^2\times \{\ell\}$, where $\mathrm{T}^2:=(2\pi L_1 \mathbb{T}) \times (2\pi L_2 \mathbb{T})$. We define the Poisson integral in $\mathrm{T}^2 \times (-\infty,\ell)$ by
\begin{equation}\label{P-1def}
\mathcal{P}_{-,\ell}f(x) = \sum_{\xi \in    (L_1^{-1} \mathbb{Z}) \times
(L_2^{-1} \mathbb{Z}) }  \frac{e^{i \xi \cdot x' }}{2\pi \sqrt{L_1 L_2}} e^{|\xi|(x_3-\ell)} \hat{f}(\xi),
\end{equation}
where for $\xi \in  (L_1^{-1} \mathbb{Z}) \times (L_2^{-1} \mathbb{Z})$ we have written
\begin{equation}\label{horiz_ft_def}
 \hat{f}(\xi) = \int_{\mathrm{T}^2} f(x')  \frac{e^{- i \xi \cdot x' }}{2\pi \sqrt{L_1 L_2}} dx'.
\end{equation}
Here ``$-$'' stands for extending downward and ``$\ell$'' stands for extending at $x_3=\ell$, etc. It is well-known that $\mathcal{P}_{-,\ell}:H^{s}(\Sigma_+) \rightarrow H^{s+1/2}(\mathrm{T}^2 \times (-\infty,\ell))$ is a bounded linear operator for $s>0$. Certain improvements of this are available when we restrict to $\Omega$; we refer to the appendix of \cite{WTK} for details.

We extend $\eta_+$ to be defined on $\Omega$ by
\begin{equation}\label{P+def}
\bar{\eta}_+(x',x_3)=\mathcal{P}_+\eta_+(x',x_3):=\mathcal{P}_{-,\ell}\eta_+(x',x_3),\text{ for } x_3\le \ell.
\end{equation}
If $\eta_+\in H^{s-1/2}(\Sigma_+)$ for $s\ge 0$, then $\bar{\eta}_+\in H^{s}(\Omega)$.

Similarly, for $\Sigma_- = \mathrm{T}^2\times \{0\}$ we define the Poisson integral in $\mathrm{T}^2 \times (-\infty,0)$ by
\begin{equation}\label{P-0def}
\mathcal{P}_{-,0}f(x) = \sum_{\xi \in    (L_1^{-1} \mathbb{Z}) \times (L_2^{-1} \mathbb{Z}) }  \frac{e^{ i \xi \cdot x' }}{2\pi \sqrt{L_1 L_2}} e^{ |\xi|x_3} \hat{f}(\xi).
\end{equation}
It is clear that $\mathcal{P}_{-,0}$ has the  same regularity properties as $\mathcal{P}_{-,\ell}$. This allows us to extend $\eta_-$ to be defined on $\Omega_-$. However, we do not extend $\eta_-$ to the upper domain $\Omega_+$ by the reflection  since this will result in the discontinuity of the partial derivatives in $x_3$ of the extension. For our purposes,  we instead to do the extension through the following. Let $0<\lambda_0<\lambda_1<\cdots<\lambda_m<\infty$ for $m\in \mathbb{N}$ and define the $(m+1) \times (m+1)$ Vandermonde matrix $V(\lambda_0,\lambda_1,\dots,\lambda_m)$ by $V(\lambda_0,\lambda_1,\dots,\lambda_m)_{ij} = (-\lambda_j)^i$ for $i,j=0,\dotsc,m$.  It is well-known that the Vandermonde matrices are invertible, so we are free to let $\alpha=(\alpha_0,\alpha_1,\dots,\alpha_m)^T$ be the solution to
\begin{equation}\label{Veq}
V(\lambda_0,\lambda_1,\dots,\lambda_m)\,\alpha=q_m,\ q_m=(1,1,\dots,1)^T.
\end{equation}
Now we define the specialized Poisson integral in $\mathrm{T}^2 \times (0,\infty)$
by
\begin{equation}\label{P+0def}
\mathcal{P}_{+,0}f(x) = \sum_{\xi \in    (L_1^{-1} \mathbb{Z}) \times
(L_2^{-1} \mathbb{Z}) }  \frac{e^{ i \xi \cdot x' }}{2\pi \sqrt{L_1 L_2}}  \sum_{j=0}^m\alpha_j
e^{- |\xi|\lambda_jx_3} \hat{f}(\xi).
\end{equation}
It is easy to check that, due to \eqref{Veq}, $\partial_3^l\mathcal{P}_{+,0}f(x',0)=
\partial_3^l\mathcal{P}_{-,0}f(x',0)$  for all $0\le l\le m$ and hence
\begin{equation}
\partial^\alpha\mathcal{P}_{+,0}f(x',0)=
\partial^\alpha\mathcal{P}_{-,0}f(x',0), \ \forall\, \alpha\in \mathbb{N}^3 \text{ with }0\le |\alpha|\le m.\end{equation}
These facts allow us to  extend $\eta_-$ to be defined on $\Omega$
by
\begin{equation}\bar{\eta}_-(x',x_3)=
\mathcal{P}_-\eta_-(x',x_3):=\left\{\begin{array}{lll}\mathcal{P}_{+,0}\eta_-(x',x_3),\quad
x_3> 0 \\
\mathcal{P}_{-,0}\eta_-(x',x_3),\quad x_3\le
0.\end{array}\right.\label{P-def}\end{equation}
It is clear now that if $\eta_-\in H^{s-1/2}(\Sigma_-)$ for $ 0\le s\le m$, then $\bar{\eta}_-\in H^{s}(\Omega)$.  Since we will only work with $s$ lying in a finite interval, we may assume that $m$ is sufficiently large in \eqref{Veq} for $\bar{\eta}_- \in H^s(\Omega)$ for all $s$ in the interval.

\section{Estimates of Sobolev norms}

Here we record an estimate involving space-time norms.

\begin{lem}\label{time_interp}
Let $\Gamma$ denote either $\Sigma$ or $\Omega$.  Suppose that $\zeta \in L^2([0,T]; H^{s_1}(\Gamma))$ and $\dt \zeta \in L^2([0,T]; H^{s_2}(\Gamma))$ for $s_1 \ge s_2 \ge 0$.  Let $s = (s_1+s_2)/2$.  Then  $\zeta \in C^0([0,T]; H^{s}(\Gamma))$ (after possibly being redefined on a set of measure $0$), and
\begin{equation}\label{l_sobi_01}
 \ns{\zeta}_{L^\infty H^{s}} \le \ns{\zeta(0)}_{H^s} + C \ns{ \zeta}_{L^2 H^{s_1}} + C \ns{\dt \zeta}_{L^2 H^{s_2}}
\end{equation}
for some universal constant $C>0$.
\end{lem}
\begin{proof}
 This is a slight variant of Lemma A.4 of \cite{GT_lwp} that follows from the same argument.
\end{proof}

\section{Some estimates involving the geometric terms}

\subsection{Coefficient estimates}

Here we are concerned with how the size of $\eta$ can control the ``geometric'' terms that appear in the equations.
\begin{lem}\label{eta_small}
There exists a universal $0 < \delta < 1$ so that if $\ns{\eta}_{5/2} \le \delta,$ then
\begin{equation}\label{es_01}
\begin{split}
 & \norm{J-1}_{L^\infty(\Omega)} +\norm{A}_{L^\infty(\Omega)} + \norm{B}_{L^\infty(\Omega)} \le \hal, \\
 & \norm{\n-1}_{L^\infty(\Gamma)} + \norm{K-1}_{L^\infty(\Gamma)} \le \hal, \text{ and }  \\
 & \norm{K}_{L^\infty(\Omega)} + \norm{\mathcal{A}}_{L^\infty(\Omega)} \ls 1.
 \end{split}
\end{equation}
Also, the map $\Theta$ defined by \eqref{cotr} is a diffeomorphism, and
\begin{equation}\label{es_02}
 \hal \int_\Omega \abs{\varphi}^2 \le \int_\Omega J \abs{ \varphi}^2 \le 2 \int_\Omega \abs{\varphi}^2
\end{equation}
for all $\varphi \in L^2(\Omega)$.
\end{lem}
\begin{proof}
The estimate \eqref{es_01} is guaranteed by Lemma 2.4 of \cite{GT_per}.  The estimate \eqref{es_02}  then follows trivially from \eqref{es_01}.
\end{proof}

\subsection{Korn's inequality}

Here we record a version of Korn's inequality that is needed throughout our analysis.   First we record a version involving only the deviatoric part of the symmetric gradient, $\sgz$, defined by \eqref{deviatoric_def}.

\begin{prop}\label{layer_korn}
 There exists a universal constant $C>0$ so that
\begin{equation}
 \ns{u}_1 = \ns{u_+}_{1} + \ns{u_-}_{1} \le C(\ns{\sgz{u_+}}_0 + \ns{\sgz{u_-}}_0)
\end{equation}
for all $u_\pm \in H^1(\Omega_\pm)$ with $\jump{u}=0$ along $\Sigma$ and $u_- =0$ on $\Sigma_b$.
\end{prop}
\begin{proof}
The proof is based on the ``deviatoric Korn inequality'' of Dain, Theorem 1.1 of \cite{dain}.  For details of the proof see the appendix of \cite{JTW_GWP}.
\end{proof}

Next we extend Proposition \ref{layer_korn} to  $\sgz_{\a}$.

\begin{prop}\label{korn}
Assume that $\mu >0$ and $\mu' \ge 0$.  There exists a universal $0 < \delta < 1$, smaller than the $\delta$ appearing in Lemma \ref{eta_small}, such that if $\ns{\eta}_{5/2} \le \delta,$  then
\begin{equation}\label{ko_01}
 \ns{u}_1 \ls \int_\Omega \frac{\mu}{2} J \abs{\sgz_\a u}^2 + \mu' J \abs{\diva u}^2 \ls \ns{u}_1
\end{equation}
for all $u_\pm \in H^1(\Omega_\pm)$ with $\jump{u}=0$ along $\Sigma$ and $u_- =0$ on $\Sigma_b$.

\end{prop}
\begin{proof}
 Let $\delta$ be as small as in Lemma \ref{eta_small}.  With Proposition \ref{layer_korn} in hand we may then argue as in  Lemma 2.1 of \cite{GT_lwp}, further restricting $\delta$ as needed, to derive \eqref{ko_01}.
\end{proof}

\subsection{Elliptic estimates}

Here we record elliptic estimates for the two-phase geometric Lam\'{e} problem with Dirichlet boundary conditions:
\begin{equation}\label{lam_dir}
 \begin{cases}
   - \diva \Sa u = G & \text{in }\Omega \\
  u_+ = h_+ &\text{on } \Sigma_+ \\
  u_+ = u_-  = h_- &\text{on } \Sigma_- \\
  u_- = 0 &\text{on } \Sigma_b.
 \end{cases}
\end{equation}

Our elliptic regularity result is contained in the following.

\begin{prop}\label{lame_elliptic}
 Let $k\ge 4$ be an integer and suppose that $\eta \in H^{k+1/2}$.  There exists $0 < \delta_0 \le \delta$ (where $\delta$ is given by Proposition \ref{korn}) so that if $\ns{\eta}_{k-1/2} \le \delta_0$, then solutions to \eqref{lam_dir} satisfy
\begin{equation}\label{ld_01}
  \norm{u}_{r}  \ls \left( \norm{G}_{r-2} + \norm{h}_{r-1/2}   \right)
\end{equation}
for $r=2,\dotsc,k$, whenever the right side is finite.

In the case $r=k+1$,  solutions to \eqref{lam_dir} satisfy
\begin{equation}\label{ld_02}
\norm{u}_{k+1}  \ls \left( \norm{G}_{k-1} + \norm{h }_{k+1/2}   \right)
+  \norm{\eta}_{k+1/2} \left( \norm{F^1}_{2} +  \norm{h}_{7/2}  \right)
\end{equation}
whenever the right side is finite.
\end{prop}
\begin{proof}
Estimates of the form \eqref{ld_01} for all $r \ge 2$ follow from the standard elliptic regularity theory for the problem
\begin{equation}
 \begin{cases}
   - \diverge \S u = \tilde{G} & \text{in }\Omega \\
  u_+ = h_+ &\text{on } \Sigma_+ \\
  u_+ = u_-  = h_- &\text{on } \Sigma_- \\
  u_- = 0 &\text{on } \Sigma_b.
 \end{cases}
\end{equation}
With these estimates in hand, we may argue as in Section 3 of \cite{WTK} to deduce \eqref{ld_01} and \eqref{ld_02} under the smallness assumption $\ns{\eta}_{k-1/2} \le \delta_0$.

\end{proof}

\section{Data extension results}

Next we need some extension results.  Our first one allows us to take data with parabolic scaling and extend them to space-time functions with the same scaling.  The proof can be found in Lemma A.5 of \cite{GT_lwp} and is thus omitted.

\begin{prop}\label{extension_u}
Suppose that for $j=0,\dotsc,2N$ we have that $\dt^j v(0) \in H^{4N-2j}(\Omega)$.  Then there exists an extension $u$, achieving the initial data, such that
\begin{equation}
\dt^j u \in L^\infty([0,\infty); H^{4N-2j}(\Omega)) \cap L^2([0,\infty);H^{4N-2j+1}(\Omega)) \text{ for }j=0,\dotsc,2N.
\end{equation}
Moreover,
\begin{equation}
 \sup_{t \ge 0} \sum_{j=0}^{2N} \ns{\dt^j v(t)}_{4N} + \int_0^\infty \sum_{j=0}^{2N}\ns{\dt^j v(t)}_{4N-2j+1} dt \ls  \sum_{j=0}^{2N} \ns{\dt^j v(0)}_{4N-2j}
\end{equation}
\end{prop}

We also need a version of Proposition \ref{extension_u} that works to extend $\eta$.  The difference between $\eta$ and $u$ is that Proposition \ref{data_estimate}  provides higher regularity for the time derivatives that we would get from strict parabolic scaling.  Nevertheless, Lemma A.5 of \cite{GT_lwp} may be readily modified to construct an extension.

\begin{prop}\label{extension_eta}
Suppose that $\zeta(0) \in H^{4N+1/2}(\Sigma)$ and that $\sqrt{\sigma} \nab_\ast \zeta(0) \in H^{4N}(\Sigma)$.  Further suppose that for $j=1,\dotsc,2N$ we have that $\dt^j \zeta(0) \in H^{4N-2j+3/2}(\Sigma)$.  Then there exists an extension $\zeta$, achieving the initial data, such that
\begin{equation}
\begin{split}
 \zeta &\in  L^\infty([0,\infty); H^{4N+1/2}(\Sigma)) \cap L^2([0,\infty);H^{4N+1}(\Sigma)), \\
 \sqrt{\sigma} \nab_\ast \zeta &\in L^\infty ([0,\infty); H^{4N}(\Sigma)) \cap L^2([0,\infty);H^{4N+1/2}(\Sigma)) , \\
 \dt \zeta   &\in L^\infty ([0,\infty); H^{4N-1/2}(\Sigma)) \cap L^2([0,\infty);H^{4N}(\Sigma)) , \\
 \dt^j \zeta & \in L^\infty([0,\infty); H^{4N-2j+3/2}(\Sigma)) \cap L^2([0,\infty);H^{4N-2j+5/2}(\Sigma))  \text{ for }j=2,\dotsc,2N,\\
 \dt^{2N+1}\zeta &\in L^2([0,\infty);H^{1/2}(\Sigma)).
\end{split}
\end{equation}
Moreover, we have the estimates
\begin{multline}\label{ex_e01}
\sup_{t \ge 0} \left( \ns{\zeta(t)}_{4N+1/2} + \sum_{j=1}^{2N} \ns{\dt^j \zeta(t)}_{4N-2j+3/2} \right) \\
+ \int_0^\infty \left( \ns{ \zeta(t)}_{4N+1} +\ns{\dt \zeta(t)}_{4N}+ \sum_{j=2}^{2N+1}\ns{ \dt^j \zeta(t)}_{4N-2j +5/2} \right) dt \\
\ls  \ns{\zeta(0)}_{4N+1/2} +\sum_{j=1}^{2N} \ns{\dt^j \zeta(0)}_{4N-2j+3/2}
\end{multline}
and
\begin{equation}\label{ex_e02}
 \sup_{t\ge 0} \sigma \ns{\nab_\ast \zeta}_{4N} +   \int_0^\infty \sigma \ns{\nab_\ast \zeta(\cdot,t)}_{4N+1/2} dt \ls \sigma \ns{\nab_\ast \zeta(0)}_{4N} + \sigma \sum_{j=1}^{2N} \ns{\dt^j \zeta(0)}_{4N-2j+3/2}.
\end{equation}
\end{prop}
\begin{proof}

For each $j=0,\dotsc,2N$ let $\varphi_j \in C_c^\infty(\Rn{})$ be such that $\varphi_j^{(k)}(0) = \delta_{j,k}$ for $k=0,\dotsc,2N$ (here $(k)$ is the number of derivatives and $\delta_{j,k}$ is the Kronecker delta).  For the same $j$ let $f_j  = \dt^j \zeta(0)$.  Then $f_0 \in H^{4N+1/2}(\Sigma)$, $\sqrt{\sigma} \nab_\ast f_0 \in H^{4N}(\Sigma)$, and $f_j \in H^{4N-2j+3/2}(\Sigma)$ for $j=1,\dotsc,2N$.

We will construct $\zeta$ as a sum $\zeta = \sum_{j=0}^{2N} F_j$.  To construct the $F_j$ we must break to cases, considering $j=0$ and $j\ge 1$ separately.  We begin with the latter case, in which case we use the Fourier transform given by \eqref{horiz_ft_def} for the construction.  We define $F_j$ via its Fourier coefficients:
\begin{equation}
 \hat{F}_j(\xi,t) = \varphi_j(t \br{\xi}^2) \hat{f}_j(\xi) \br{\xi}^{-2j},
\end{equation}
where $\br{\xi} = \sqrt{1+ \abs{\xi}^2}$.

By construction, $\dt^k \hat{F}_j(\xi,t) = \varphi_j^{(k)}(t \br{\xi}^2) \hat{f}_j(\xi) \br{\xi}^{2(k-j)}$, and hence $\dt^k F(\cdot,0) = \delta_{j,k} f_j$.  We estimate
\begin{equation}
\begin{split}
 \ns{\dt^k F_j(\cdot,t)}_{4N-2k+3/2} & = \sum_{\xi} \br{\xi}^{2(4N-2k+3/2)}  \abs{\varphi_j^{(k)}(t \br{\xi}^2)}^2  \abs{\hat{f}_j(\xi)}^2  \br{\xi}^{2(2k-2j)} d\xi \\
& = \sum_\xi   \abs{\varphi_j^{(k)}(t \br{\xi}^2)}^2  \abs{\hat{f}_j(\xi)}^2  \br{\xi}^{2(4N-2j+3/2)} d\xi \\
&\ls  \ns{\varphi_j^{(k)}}_{L^\infty(\Rn{})}  \ns{f_j}_{4N-2j+3/2},
\end{split}
\end{equation}
so
\begin{equation}\label{ex_e1}
 \sup_{t\ge 0} \ns{\dt^k F_j(\cdot,t)}_{ 4N-2k+3/2} \ls \ns{f_j}_{4N-2j+3/2} \text{ for }k=0,\dotsc,2N.
\end{equation}
Similarly,
\begin{equation}
\begin{split}
\int_0^\infty\ns{\dt^k F_j(\cdot,t)}_{4N-2k+5/2}dt &= \int_0^\infty \sum_\xi \br{\xi}^{2(4N-2k+5/2)}  \abs{\varphi_j^{(k)}(t \br{\xi}^2)}^2  \abs{\hat{f}_j(\xi)}^2  \br{\xi}^{2(2k-2j)}  dt\\
& = \int_0^\infty \sum_\xi   \abs{\varphi_j^{(k)}(t \br{\xi}^2)}^2  \abs{\hat{f}_j(\xi)}^2  \br{\xi}^{2(4N-2j+5/2)}  dt  \\
& = \sum_\xi \abs{\hat{f}_j(\xi)}^2  \br{\xi}^{2(4N-2j+5/2)} \left( \int_0^\infty    \abs{\varphi_j^{(k)}(t \br{\xi}^2)}^2 dt\right)   \\
& = \sum_\xi \abs{\hat{f}_j(\xi)}^2  \br{\xi}^{2(4N-2j+5/2)} \left( \frac{1}{\br{\xi}^2}\int_0^\infty    \abs{\varphi_j^{(k)}(r)}^2 dr\right)   \\
&= \ns{\varphi_j^{(k)}}_{L^2(\Rn{})} \sum_\xi \abs{\hat{f}_j(\xi)}^2  \br{\xi}^{2(4N-2j+3/2)}  \\
&\ls \ns{\varphi_j^{(k)}}_{L^2(\Rn{})}\ns{f_j}_{4N-2j+3/2},
\end{split}
\end{equation}
so
\begin{equation}\label{ex_e2}
\int_0^\infty\ns{\dt^k F_j(\cdot,t)}_{4N-2k+5/2}dt \ls \ns{f_j}_{4N-2j+3/2} \text{ for }k=0,\dotsc,2N+1.
\end{equation}

It remains to handle the case $j=0$.  We define $\hat{F}_0(\xi,t) = \varphi_0(t\br{\xi}) \hat{f}_0(\xi)$.  Then we may argue as above to deduce the bounds
\begin{equation}\label{ex_e3}
 \sup_{t\ge 0} \left( \ns{F_0(\cdot,t)}_{4N+1/2} + \sum_{k=1}^{2N} \ns{ \dt^k  F_0(\cdot,t)}_{4N-2k+3/2} \right)  \ls \ns{f_0}_{4N+1/2},
\end{equation}
and
\begin{equation}\label{ex_e4}
 \int_0^\infty\left(\ns{F_0(\cdot,t)}_{4N+1}+ \ns{\dt F_0(\cdot,t)}_{4N}  +\sum_{k=2}^{2N+1}\ns{\dt^k F_0(\cdot,t)}_{4N-2k+5/2}   \right)dt  \ls \ns{f_0}_{4N+1/2}
\end{equation}
for all $k \ge 0$.  We may also estimate
\begin{equation}\label{ex_e5}
 \sup_{t\ge 0} \ns{ \nab_\ast F_0(\cdot,t)}_{4N}   + \int_0^\infty \ns{\nab_\ast F_0(\cdot,t)}_{4N+1/2} dt  \ls \ns{\nab_\ast f_0}_{4N}.
\end{equation}

Now, since we set $\zeta = \sum_{j=0}^{2N} F_j$ we may sum \eqref{ex_e1}, \eqref{ex_e2}, \eqref{ex_e3}, and \eqref{ex_e4} to deduce that \eqref{ex_e01} holds.   To prove \eqref{ex_e02} we   sum \eqref{ex_e1} and \eqref{ex_e2} with $k=0$ and \eqref{ex_e5}:
\begin{multline}
 \sup_{t\ge 0} \sigma  \ns{\nab_\ast \zeta(\cdot,t)}_{4N} +  \int_0^\infty \sigma \ns{\nab_\ast \zeta(\cdot,t)}_{4N+1/2} dt
\ls  \sup_{t\ge 0}\sigma \ns{\nab_\ast F_1(\cdot,t)}_{4N} \\
+  \int_0^\infty \sigma \ns{\nab_\ast F_0(\cdot,t)}_{4N+1/2} dt
+ \sup_{t\ge 0} \sigma \sum_{j=1}^{2N}\ns{ \nab_\ast F_j(\cdot,t)  }_{4N}
+ \int_0^\infty \sigma \sum_{j=1}^{2N}\ns{\nab_\ast F_j(\cdot,t)}_{4N+1/2}
\\
\ls \sigma \ns{\nab_\ast \zeta(0)}_{4N} + \sigma \sum_{j=1}^{2N}\ns{   \dt^j \zeta(0)  }_{4N-2j+3/2}.
\end{multline}
\end{proof}

We will also need the following simple variant of Proposition \ref{extension_eta}.  We omit the proof for the sake of brevity.

\begin{prop}\label{extension_Xi}
Suppose that $\Xi(0) \in H^{4N-1}(\Sigma)$ and that for $j=1,\dotsc,2N$ we have that $\dt^j \Xi(0) \in H^{4N-2j-1/2}(\Sigma)$.  Then there exists an extension $\Xi$, achieving the initial data, such that
\begin{equation}
 \dt^j \Xi  \in L^\infty([0,\infty); H^{4N-2j-1}(\Sigma)) \cap L^2([0,\infty);H^{4N-2j}(\Sigma))  \text{ for }j=1,\dotsc,2N.
\end{equation}
Moreover, we have the estimates
\begin{multline}
 \sum_{j=0}^{2N-1} \sup_{t>0} \ns{\dt^j \Xi(t)}_{4N-2j-1} + \sum_{j=0}^{2N} \int_0^\infty \ns{\dt^j \Xi(t)}_{ 4N-2j}dt
\\
\ls \ns{\Xi(0)}_{4N-1} + \sum_{j=1}^{2N} \ns{\dt^j \Xi(0)}_{4N-2j-1}
\end{multline}

\end{prop}

Finally, we record a similar extension for use on $q$.  The proof follows from an argument like that used in Proposition \ref{extension_eta} and is thus omitted.

\begin{prop}\label{extension_q}
Suppose that $p(0) \in H^{4N}(\Omega)$ and that for $j=1,\dotsc,2N$ we have that $\dt^j p(0) \in H^{4N-2j+1}(\Omega)$.  Then there exists an extension $p$, achieving the initial data, such that
\begin{equation}
\begin{split}
  p &\in L^\infty([0,\infty); H^{4N}(\Omega)) \cap L^2([0,\infty);H^{4N+1/2}(\Omega)), \\
  \dt p & \in L^\infty([0,\infty); H^{4N-1}(\Omega)) \cap L^2([0,\infty);H^{4N-1/2}(\Omega)), \\
 \dt^j p & \in L^\infty([0,\infty); H^{4N-2j+1}(\Omega)) \cap L^2([0,\infty);H^{4N-2j+2}(\Omega)) \text{ for }j=2,\dotsc,2N, \\
 \dt^{2N+1} p & \in L^2([0,\infty);H^{0}(\Omega))
\end{split}
\end{equation}
Moreover, we have the estimate
\begin{multline}
\sup_{t \ge 0} \left( \ns{p(t)}_{4N} + \sum_{j=1}^{2N} \ns{\dt^j p(t)}_{4N-2j+1} \right) \\
+ \int_0^\infty \left( \ns{ p(t)}_{4N+1/2} +\ns{\dt p(t)}_{4N-1/2}+ \sum_{j=2}^{2N+1}\ns{ \dt^j p(t)}_{4N-2j +2} \right) dt \\
\ls  \ns{p(0)}_{4N} +\sum_{j=1}^{2N} \ns{\dt^j p(0)}_{4N-2j+1}.
\end{multline}
\end{prop}

\section*{Acknowledgements}
I. Tice would like to thank the Laboraroire Jacques-Louis Lions at Universit\'{e} Pierre et Marie Curie for their hospitality during his visit.

\end{document}